\documentclass[a4paper,oneside,11pt]{article}
\usepackage{geometry}                
\usepackage[active]{srcltx}
\usepackage{hyperref}

\usepackage{xcolor}

\usepackage[
    backend=bibtex,
    doi=false,
    isbn=false,
    url=false,
    eprint=true,
    maxbibnames=99,
    giveninits=true
]{biblatex}

\renewbibmacro{in:}{}
\usepackage{amsmath}
\usepackage{amsthm}
\usepackage{amssymb}
\usepackage{latexsym}
\usepackage{mathtools}
\usepackage{mathrsfs}
\usepackage{graphics}
\usepackage{latexsym}
\usepackage{psfrag}
\usepackage{import}
\usepackage{verbatim}
\usepackage{enumerate}
\usepackage{enumitem}

\newcommand{\R}{\mathbb{R}}
\newcommand{\N}{\mathbb{N}}

\newcommand{\B}{\mathcal{B}}
\newcommand{\Q}{\mathbb{Q}}

\newcommand{\weak}{\rightharpoonup}

\newcommand{\T}{\mathcal{T}}

\renewcommand{\L}{\mathcal{L}}

\newcommand{\Ric}{{\rm Ric}}
\newcommand{\norm}[1]{\left\Vert#1\right\Vert}

\newcommand{\eps}{\varepsilon}
\newcommand{\tr}{\textrm{tr}}
\newcommand{\K}{\mathcal{K}}

\renewcommand{\L}{\mathcal{L}}
\newcommand{\m}{\mathfrak{m}}
\newcommand{\q}{\mathfrak{q}}

\renewcommand{\P}{\mathbb P}

\renewcommand{\P}{\mathcal{P}}

\renewcommand{\H}{\mathcal{H}}

\newcommand{\spt}{\operatorname{spt}\ }

\theoremstyle{plain}
\newtheorem{lemma}{Lemma}[section]
\newtheorem{theorem}[lemma]{Theorem}

\newtheorem{proposition}[lemma]{Proposition}
\newtheorem{corollary}[lemma]{Corollary}
\newtheorem*{theorem*}{Theorem}
\newtheorem*{maintheorem*}{Main Theorem}

\theoremstyle{definition}

\newtheorem{definition}[lemma]{Definition}
\newtheorem*{definition*}{Definition}
\newtheorem{remark}[lemma]{Remark}

\numberwithin{equation}{section}

\newcounter{mycounter}
\global\long\def\R{\mathbb{R}}%

\global\long\def\T{\mathbb{\mathcal{T}}}%

\global\long\def\g{\mathbb{\mathcal{\gamma}}}%

\global\long\def\P{\mathbb{\mathcal{\mathcal{P}}}}%

\global\long\def\L{\mathcal{L}}%

\global\long\def\m{\mathfrak{m}\mathcal{}}%

\global\long\def\Gu{\Gamma_{u}}%

\global\long\def\q{\mathfrak{q}}%

\global\long\def\spt{\text{spt}\,}%

\global\long\def\restrict{\text{\ensuremath{\llcorner}}}%

\begin{document}
\title{ Monge's transport problem in synthetic Lorentzian spacetimes}
\author{Afiny Akdemir} 
		\maketitle

\begin{abstract}
We prove the existence of solutions to the Lorentzian Monge problem on synthetic Lorentzian spacetimes that satisfy the forward timelike measure contraction property $TMCP^+(K,N)$. The cost for this problem is the time separation function $\ell(x,y)$ (the Lorentz distance), which represents the maximum amount a particle can age when traveling from $x$ to $y$.  The solution is based on the needle decomposition technique, and amounts to a reduction of the full problem to its one-dimensional counterparts. As an application of needle decomposition, we introduce the blended $TMCP$ condition, and obtain a quantitative refinement of the synthetic Hawking singularity theorem. We additionally prove a timelike nonbranching reversibility theorem for $TCD_p(K,N)$ spacetimes, showing that forward $p$-essentially nonbranching spacetimes are backward $p$-essentially nonbranching.
\end{abstract}

\tableofcontents

\section{Introduction}\label{Ss:introduction}
Monge's optimal transportation problem is one of the oldest problems in the fields of geometric analysis and partial differential equations. The objective of the problem is to find a map that moves one distribution of mass onto another as efficiently as possible, where efficiency is defined by the cost function $c(x,y)$ that encodes the cost to move mass from position $x$ onto $y$. 

Letting $(X,d,\m)$ be a metric measure space, and given the Borel probability measures $\mu,\nu\in \P(X)$ --- the source and target distributions, respectively, Monge's  1781 Euclidean formulation \cite{Monge} takes the cost to be the distance function, that is $c(x,y)=d(x,y)$, and accordingly, the problem can thus be stated as the following $L^1$ minimization problem:
\begin{equation}\label{Monge Problem}
    \inf_{T_\#\mu =\nu }\int _X d(x,T(x))d\mu(x),\end{equation} where we are interested in finding the map $T:X\rightarrow X$ with the constraint $T_\# \mu = \nu$ meaning $$\mu(T^{-1}(A)) = \nu(A), \quad \textrm{for all  Borel subsets $A$ of $X$}.$$

In this paper, we investigate the corresponding Lorentzian Monge problem. This problem is set on a globally hyperbolic measured metric spacetime denoted by $(X,d,\m,\ell)$ satisfying the $TMCP^+(K,N)$ condition, where the cost is the time separation function $c(x,y)=\ell(x,y)$, which encodes the chronology of the spacetime.

We briefly recall the history of the Monge problem as well as the underlying assumptions in both the Riemannian and Lorentzian signatures. In Monge's original 1781 formulation, where the cost is proportional to the distance moved, the object being minimized is interpreted as the total work. Despite being physically the most natural cost, the Monge problem with the distance cost presented significant difficulties due to the lack of strict convexity of the distance cost, as well as the non-linearity of constraint $T_\# \mu = \nu$. 

To overcome these difficulties, Kantorovich in his seminal work \cite{Kantorovich} proposed a relaxation of the Monge problem, whereby the focus is on finding an optimal probability coupling $\pi$ of the source distribution $\mu$ and the target distribution $\nu$ over the convex set $\Pi(\mu,\nu)$: that is $\pi\in \P(X\times X)$ where the projections $P_1(x,y)=x,$ and $P_2(x,y)=y$ satisfy $(P_1)_\#\pi = \mu$, $(P_2)_\#\pi = \nu$. Thus, the Kantorovich formulation of the $L^1$ optimal transportation problem is stated as 
\begin{equation}\label{Kantorovich}
    d_1(\mu,\nu) :=\inf _{\pi\in \Pi(\mu,\nu)}\int_{X\times X}d(x,y)d\pi(x,y),
\end{equation}
and this reformulation converts the nonlinear, non-convex optimization problem \eqref{Monge Problem} into a linear program over a convex set, for which classical duality and compactness methods apply directly. By invoking Kantorovich's duality \cite{Kantorovich}, we get the dual formulation to the primal Kantorovich problem \eqref{Kantorovich}:  

    \begin{equation}\label{Duality}
\inf _{\pi\in \Pi(\mu\,\nu)}\int_{X\times X}d(x,y)d\pi(x,y) = \sup_{u\in \textrm{1-Lip}}\int_Xu(y)d\nu -\int_Xu(x)d\mu,
    \end{equation}
    where 1-Lip denotes the class of Lipschitz continuous functions with respect to  $d$ with their Lipschitz constant not greater than 1. Accordingly, the function $u$ that attains the supremum is called the \emph{Kantorovich potential} for the problem.

The question of whether the Kantorovich \eqref{Kantorovich} and the Monge \eqref{Monge Problem} formulations  agree with one another is twofold: first, whether their optimal values agree, and second, whether there exists an   optimal coupling $\pi$ is concentrated on the graph of a Monge map $T$, that is, $\pi = (Id,T)_\# \mu.$ The values can agree even when no optimal map $T$ exists, and the existence of a Monge map is thus a stronger statement, requiring that the optimal coupling is  induced by a map. The starting point for finding a Monge map is to consider the structure of the support of the optimal coupling $\pi$  provided by the dual formulation \eqref{Duality}. Indeed, for the optimal coupling $\pi$ and the corresponding Kantorovich potential $u$,  $\pi$ is concentrated on the set where $u$ is affine, that is:

\begin{equation}\label{support pi}
    \pi[(x,y)\in X\times X: u(y)-u(x)=d(x,y)] =1.
\end{equation} Additionally, \eqref{support pi} provides the correspondences between $\mu$-a.e. $x$ and $\nu$-a.e. $y$, and so the strategy from here is to use these correspondences to construct a Monge map $T:X\rightarrow X$ such that $T(x)=y.$

The first proposed solution to the above strategy for finding the Monge map in the case where $(X,d,\m)= (\R^n,|\cdot|,\L^n)$  was given by Sudakov in 1978 \cite{Sudakov}, where $|\cdot|$ is the Euclidean distance, and $\L^n$ is the Lebesgue measure. Sudakov's initial proposal applied to arbitrary norms in $\R^n$, where his idea was to restrict the space and the Lebesgue measure $\L^n$ onto lower dimensional subsets using measure decomposition techniques, thus decomposing the full problem into sub-problems for which the solutions can be explicitly constructed. This reduction, together with a restriction of the Kantorovich potential $u$ onto these sub-problems allowed Sudakov to decompose the space into lower dimensional subsets where the Kantorovich potential is an affine function. The particular case where this reduction yields one-dimensional sets can be seen by considering \eqref{support pi}, whereby the $d$-cyclical monotonicity of the support of $\pi$ shows that $u$ must be affine along a geodesic $\gamma:[0,1]\rightarrow X$, that is: $$(x,y)\in \spt \pi \implies \exists\gamma:[0,1]\rightarrow X \,\,\textrm{with} \,\, u(\gamma_t)-u(\gamma_s)=d(\gamma_s,\gamma_t) \, \, \forall s,t\in[0,1], s\leq t. $$ This particular one dimensional restriction is called a \emph{transport ray}, and the restricted problem is now equivalent to $L^1$ optimal transportation problem on $\R$ for which explicit solutions can be classified under the condition that the source measure $\mu$ does not concentrate on points.

The flaw in Sudakov's argument was discovered in 2000 by Alberti-Kirchheim-Preiss \cite{AKP}, where a Besicovitch style construction applied to Sudakov's decomposition of $\L^n$ onto a family of one dimensional conditional probabilities yields a measure that is no longer absolutely continuous. The issue lies in the lack of geometric control of the transport rays, whereby the counterexample is produced by taking an arbitrary collection of rays which altogether form an absolutely continuous measure in $\R^3$, yet the restriction onto the transport rays produces segments whose  midpoints will have a Dirac part. This counterexample shows that the reduction onto transport rays requires additional geometric estimates that control the decomposition across rays, and that this reduction must produce the restricted measures that are absolutely continuous.

The first correct solution to the Monge problem in $(\R^n,|\cdot|,\L^n)$ came from the work of Evans-Gangbo \cite{EvansGangbo}, where their argument uses the $p$-Laplacian equation and constructs the optimal map by taking the limit $p\rightarrow \infty$. Shortly afterwards, Caffarelli-Feldman-McCann \cite{CaffarelliFeldmanMcCann2002} and simultaneously Trudinger-Wang \cite{Trudinger-Wang} gave a construction for the cases where $\mu$ and $\nu$ are compactly supported absolutely continuous measures. The work of Trudinger-Wang specifically focuses on the Euclidean distance cost, whereas Caffarelli-Feldman-McCann consider norms that need not be Euclidean, providing the first insight into the possibility of extending the proposed construction to non-smooth spaces. Both papers  explicitly establish the geometric control across transport rays through a Lipschitz change of variables that explicitly controls the geometry of the transport rays through the level sets of the Kantorovich potential $u$, providing a restriction on how much these transport rays can vary. Additionally, Ambrosio \cite{Ambrosio} was able to provide another proof for the case where only the source measure $\mu$ is absolutely continuous, whereby the geometric control between transport rays is provided by the co-area formula. The extension to Riemannian manifolds was given by Feldman-McCann \cite{Feldman-McCann}.

 The next major breakthrough came from the works of Bianchini-Cavalletti \cite{BianchiniCavalletti2013}. Originally set on $(\R^n,|\cdot|,\L^n),$ the authors decompose the measure onto  one-dimensional transport rays using the disintegration of measure technique from measure theory. The geometric control on the collection of rays is ensured by showing that this particular disintegration is \emph{strongly consistent}. Being a purely measure theoretic result, the disintegration of measure technique allowed Cavalletti and Bianchini-Cavalletti \cite{Cavalletti2014, BianchiniCavalletti2013} to extend the theory to the domain of metric measure spaces $(X,d,\m)$ that satisfy the measure contraction property $MCP(K,N)$, see also \cite{Cavalletti-Overview}. The curvature properties from the $MCP(K,N)$ condition are required in order to ensure that the restriction of the measure onto the one dimensional rays are absolutely continuous, which is essential for ruling out the counterexamples to Sudakov's approach. Cavalletti \cite{Cavalletti2014} also showed that  the  Monge problem can be solved under the weaker assumption of \emph{essentially non-branching} $-$ roughly, that geodesics used by the optimal transport do not branch, except possibly on a set of measure zero. These results establish that the Monge problem is solvable in a wide class of non-smooth metric measure spaces, provided the space satisfies a synthetic curvature bound and is essentially non-branching.

Having reviewed the results in the Riemannian setting, we now turn to the Lorentzian setting. Optimal transportation in the Lorentzian signature has been developed by McCann \cite{McCann2020}, building on the foundational work done previously by Brenier, Bertrand-Pratelli-Puel, Bertrand-Puel, Eckstein-Miller, Kell-Suhr, McCann-Puel, and Suhr  \cite{ Brenier-Extended-MK, Bertrand-Pratelli-Puel, Bertrand-Puel, Eckstein-Miller, McCann-Puel,   Kell-Suhr, Suhr}. Initially set on a globally hyperbolic Lorentzian manifold $(M,g)$, the cost is defined to be the Lorentzian distance, or proper time separation, defined by $$ \ell(x,y):=\sup \bigg\{\int\sqrt{-g(\dot\gamma,\dot{\gamma})}: \gamma\textrm{ is a future-directed causal curve from $x$ to $y$} \bigg\},$$ and is set to $-\infty$ if no such curve exists. The Lorentzian distance  encodes the chronology of the Lorentzian manifold, and physically $\ell(x,y)$ is the maximal proper time that an observer can experience when traveling from an event $x$ to an event $y$. The $L^1$ Lorentzian optimal transport problem with the cost $\ell$ seeks to \emph{maximize} the total proper time: given two distributions of matter --- the past source $\mu$ and the future target $\nu$ --- one looks for a coupling $\pi$ that pairs events in the past with events in the future to make the total proper time as large as possible. The causal structure enters through the cost: only causally related pairs can be coupled, and the mass moves forward in time along optimal timelike geodesics determined by the transport problem.  The restriction of $\ell(x,y)=-\infty$ if $y$ is not in the causal future of $x$ ensures that only causal events can be coupled with a finite cost. The key insight of McCann \cite{McCann2020} was  to identify the precise condition  of \emph{timelike 1-separatedness} (see definition \ref{dualisable}) that ensures that the Lorentzian optimal transport problem is well posed. Under timelike 1-separatedness, the optimal plan $\pi$ exists, the total cost is finite, and mass is transported entirely along strictly timelike pairs $\{\ell>0\}$, whereby $\pi$ does not charge the null boundary. The Kantorovich problem \eqref{Kantorovich} in the Lorentzian setting
becomes the following $L^1$ maximization problem,

\begin{equation}\label{Lorentzian Kantorovich}
    \sup_{\pi\in \Pi(\mu,\nu)}\int _{X\times X}\ell(x,y)d\pi(x,y),\end{equation}which corresponds to the maximum expected proper time that can elapse between the events whose uncertainty is described by the probability distributions $\mu$ and $\nu$. Accordingly, the Lorentzian Monge problem asks whether the optimal coupling $\pi$ is concentrated on the graph of a map $T:X\rightarrow X$, i.e. $\pi = (Id,T)_\#\mu$, in which case we can rewrite the Lorentzian Kantorovich problem as
$$\sup_{T_\#\mu =\nu}\int _{X}\ell(x,T(x))d\mu(x),$$ and now the Lorentzian Monge map is a deterministic assignment that matches each source event $x$ to a target event $T(x)$ in its timelike future, maximizing the expected proper time between the source distribution $\mu$ and target distribution $\nu.$   The solution to the Lorentzian Monge problem in the case of globally hyperbolic Lorentz-Finsler spacetimes has been explored in the works of Suhr and Kell-Suhr \cite{Suhr, Kell-Suhr}.

In his work, McCann \cite{McCann2020} gave an equivalent definition of timelike Ricci curvature bounds $\Ric\geq Kg$ in terms of displacement convexity, mirroring the  Riemannian  Curvature-Dimension theory of Lott-Villani and Sturm \cite{LottVillani2009,SturmI}, and Von Renesse-Sturm in the smooth case\cite{vonRenesseSturm}; see also the independent Lorentzian development of Mondino-Suhr \cite{Mondino-Suhr}. 
Subsequently, Cavalletti and Mondino \cite{CavallettiMondino2024} generalized McCann's work to the setting of globally hyperbolic Lorentzian spacetimes $(X,d,\m,\ell)$, which are non-smooth generalizations of Lorentzian manifolds $-$ here, $d$ is the background metric inducing the chronological topology, which is Polish, $\m$ is the reference volume measure, and $\ell$ is the Lorentzian time separation. The  causal structure and global hyperbolicity are encoded in a time-separation function $\ell$, following the synthetic approach of Kunzinger-S\"amann \cite{KunzingerSaemann2018}.
Cavalletti and Mondino also generalized the Riemannian  $CD_2(K,N)$ Curvature-Dimension conditions to non-smooth Lorentzian spacetimes, which allowed them to define Lorentzian spacetimes that satisfy the Timelike Curvature Dimension condition, and the Timelike Measure Contraction Property, denoted by $TCD_p(K,N)$ and $TMCP^+(K,N)$ respectively, where $K$ is the curvature parameter, $N$ is the dimension, and the subscript $p\in(0,1)$ denotes the fractional power in the Lorentzian optimal transportation problem \eqref{Lorentzian Kantorovich}. Both conditions encode the synthetic lower bound $\Ric(\dot\gamma,\dot\gamma)\geq -K$ along timelike geodesics, and replace the classical Ricci bound by a measure-theoretic inequality. In particular, the $TMCP^+(K,N)$ condition controls how the volume measure $\m$ contracts or expands along timelike geodesics, encoding the curvature bound $\Ric\geq-K$ in purely variational terms. Additionally, the lack of smoothness means that branching of timelike geodesics is not ruled out. This is handled by the forward timelike $p$-essentially non-branching hypothesis of Braun \cite{BraunTCD}, which ensures that optimal transport plans are concentrated on non-branching geodesics.

With the above overview in mind, we state our main result:

\begin{theorem}[Solution to the Lorentzian Monge problem]\label{Solution to Monge}

    Let $(X,d,\m,\ell)$ be a globally hyperbolic Lorentzian measured metric spacetime that satisfies $TMCP^+(K,N)$ for $K\in \R, N\in(1,\infty),$ and assume that the Lorentzian spacetime is $p$-essentially non-branching for some $p\in(0,1)$. Assume that $(\mu,\nu)\in \P_c(X)\times \P_c(X)$ is a timelike 1-separated pair, with $\mu$ being absolutely continuous, and let $\pi$ be the solution to the Lorentzian Kantorovich problem \eqref{Lorentzian Kantorovich}.  Then there exists a Borel map $T:X\rightarrow X$ such that $T_{\#}\mu=\nu$, and $$\int_X \ell(x,T(x)d\mu(x)=\int_{X\times X} \ell(x,y)d\pi(x,y).$$

\end{theorem}
In fact, we can push Theorem \ref{Solution to Monge} to greater generality. The compactly supported chronological optimal coupling may be described by a saturated $\ell^1$-cyclically monotone relation. Moreover, only the forward $p$-essentially nonbranching condition is used in the construction, where backward branching can be treated by assuming that this set has $\m$ measure zero. Thus, the solution to the Monge problem extends to a one-sided nonsmooth setting which is forward $p$-essentially nonbranching and satisfies the $TMCP^+(K,N)$ condition. Finally, in Theorem \ref{forward p-enb implies backward p-enb}, we prove that under the $TCD_p(K,N)$, condition, forward $p$-essentially nonbranching implies backward essentially $p$-nonbranching, and hence the additional assumption on the backward branching set is automatically satisfied.

The strategy of the proof follows the disintegration of measure technique of Cavalletti \cite{Cavalletti2014}, which has appeared previously in the Lorentzian literature in the works of Cavalletti-Mondino \cite{CavallettiMondino2024, Cavalletti-Mondino2025+} and Braun-McCann \cite{BraunMcCann+}. Starting off with the optimal plan $\pi$ that couples the source $\mu$ and target $\nu$, the timelike 1-separatedness assumption ensures that  $\pi$ is a chronological coupling: all mass is transported along strictly timelike pairs with $\ell>0$. Kantorovich duality then produces a continuous  Kantorovich potential $u$ that is \emph{reverse Lipschitz}, that is:
$$\ell(x,y)>0\implies u(y)-u(x)\geq \ell(x,y),$$ where the inequality is reversed compared to the Riemannian case because the Lorentzian cost is maximized rather than minimized, reflecting the fact that timelike geodesics maximize proper time.  Furthermore duality gives us that $\pi$ is supported on an $\ell$-cyclically monotone set, meaning $$\pi[(x,y)\in X\times X: u(y)-u(x)=\ell(x,y)>0]=1.$$  Along each transport ray --- a timelike geodesic $\gamma$ parameterize by the aforementioned proper time $\ell$ so that $\ell(\gamma(s),\gamma(t))=t-s$ for $s\leq t$ --- the set where $u$ saturates its reverse Lipschitz constraint: $$u(\gamma(t))-u(\gamma(s))=t-s=\ell(\gamma(s),\gamma(t)).$$
 The disintegration of measure technique now produces a strongly consistent disintegration onto one dimensional timelike transport rays (analogous to how Fubini's Theorem splits an integral over a product space into iterated integrals), and the $TMCP^+(K,N)$ condition is required to show that the conditional probabilities of $\m$ restricted onto the transport rays are absolutely continuous. The Lorentzian Monge problem is then solved individually along every transport ray, and we obtain the solution to the full Lorentzian Monge problem by inverting the disintegration procedure.

The main tool underlying the above result is the disintegration of measure technique, known in the literature as \emph{needle decomposition} originating from Klartag's seminal proof of the L\'evy-Gromov isoperimetric inequality in the smooth Riemannian setting \cite{Klartag2017}. In the Riemannian signature, Klartag's argument allowed Cavalletti and Milman \cite{CavallettiMilman2021} to develop the  $CD^1_{Lip}(K,N)$ condition, where   the background measure $\m$ is disintegrated  along integral curves of an arbitrary Lipschitz function produced by the $L^1$-optimal transportation problem, with the additional requirement that the restriction of $\m$ onto these integral curves must satisfy the $CD_2(K,N)$ condition. With this definition, Cavalletti and Milman  gave a remarkable proof of the globalization property: spaces satisfying $CD_2(K,N)$ locally must also satisfy $CD_2(K,N)$ globally, where the main step in their argument was to show that  $CD_2(K,N)$ and $CD^1_{Lip}(K,N)$ are in fact equivalent conditions. Unlike the classical $CD_2(K,N)$ conditions formulated via $W_2$-displacement convexity, $CD_{Lip}^1(K,N)$ uses the $L^1$ cost. This inspired Cavalletti-Gigli-Santarcangelo to define the space $CD_1(K,N)$ which is formulated in terms of displacement convexity along $W_1$-Wasserstein geodesics, and  they showed that the two definitions $CD^1_{Lip}(K,N)$ and $CD_1(K,N)$ are in fact equivalent \cite{CavallettiGigliSantarcangelo2021}.

In this work, we follow the approach of Cavalletti-Gigli-Santarcangelo and carry out the analogous program in the Lorentzian setting. We prove the equivalence of the global $TMCP^+(K,N)$ condition with its localized counterpart $TMCP^+_{rLip}(K,N)$, showing that the timelike measure contraction property localizes to every transport ray. This localization allows us to characterize $TMCP^+(K,N)$ spacetimes through one-dimensional comparison inequalities along almost every transport ray: on almost every transport ray, the volume density $h$ satisfies a one-sided comparsion inequality for its logarithmic expansion $\theta = (\log h)'$, analogous to the expansion estimates arising from the Raychaudhuri equation.
The equality cases of these comparison inequalities yield explicit model volume densities of $TMCP^+(K,N)$ spacetimes,  corresponding to synthetic analogues of forward Minkowski $(K=0)$, de Sitter $(K<0)$ and anti-de Sitter $(K>0)$ spacetimes. Since the Lorentzian cost is asymmetric, the forward and backward evolution can have different curvature bounds --- a feature that is absent in Riemannian geometry but natural in cosmology, where the past and future are physically distinct. This asymmetry allows us to define  \emph{blended curvature bounds}: spacetimes satisfying $TMCP^+(K_1,N_1)$ and $TMCP^-(K_2,N_2)$ with different parameters $K_1\neq K_2$ and $N_1\neq N_2$. While for ordinary smooth Lorentzian manifolds, the infinitesimal Ricci bounds are invariant under reversal of the timelike direction, such an asymmetry is natural in the broader Lorentz-Finsler setting, where the geometry itself need not be reversible under $v\mapsto -v$. Thus, the blended curvature bounds provide a synthetic framework for treating future and past timelike curvature bounds independently --- for instance, one may have positive timelike Ricci curvature towards the past $(K_2>0$ anti-de Sitter-like) and negative in the future $(K_1<0,$ de Sitter-like). We show that the forward and backward curvature bounds impose compatibility constraints on the same transport-ray densities, yielding sharper Bonnet-Myers type timelike diameter bounds, including quantitative improvements in certain mixed-curvature and dimension regimes. As an application, we obtain a refinement of the synthetic Hawking singularity of Cavaletti-Mondino \cite{CavallettiMondino2024}, combining the forward mean-curvature focusing estimate with the additional backward curvature bound to obtain a sharper upper bound on the lengths of future-directed timelike geodesics. Finally, as a direct Lorentzian analogue of the Cavalletti-Gigli-Santarcangelo result, we define the Lorentzian \emph{``Limiting"} $TCD_1(K,N)$ condition, and show that it is equivalent to the \emph{``Localized"} $TCD^{1}_{rLip}(K,N)$ condition.   Our motivation for  this equivalence  lies in the lack of the canonical choice of the exponent $p$ for the Lorentzian $TCD_p(K,N)$ condition: whereas the Riemannian case admits a self-dual canonical choice $p=2$, in the Lorentzian setting, there is no such preferred exponent for the fractional cost $\ell^p$ with $p\in(0,1)$, and  our equivalence corresponds to the study of special cases of  $p\in(0,1)$, starting  with $p\rightarrow 1^-$. This complements our independence-of-$p$ result for the timelike curvature-dimension condition $TCD_p(K,N)$, established under the corresponding essential nonbranching assumptions \cite{AkdemirTCDIndependence}.

The paper is organized as follows: in Section 2 we set up the relevant notions for Lorentzian metric spacetimes and Lorentzian optimal transport that will be used throughout the work. In Section 3 we introduce the Lorentzian Monge Kantorovich problem  and prove the corresponding Kantorovich duality result. In particular, we deduce that the Kantorovich potentials that attain the supremum of the dual problem  are continuous reverse-Lipschitz functions whenever the initial and final densities are $1$-separated in the sense of McCann \cite{McCann2020}. In Section 4 we delve into the disintegration paradigm, which has already appeared in the works of Cavalletti-Mondino \cite{CavallettiMondino2024, Cavalletti-Mondino2025+} and Braun-McCann \cite{BraunMcCann+}. Our approach is more restrictive  than the assumptions presented in  \cite{CavallettiMondino2024,BraunMcCann+}, due to us working with the family of continuous reverse-Lipschitz functions. We present the solution to the Lorentzian Monge problem in Section 5. Section 6 is  devoted to applications of the $TMCP^+(K,N)$ needle decomposition, including the equivalence with its localized formulation, model and blended curvature bounds, a synthetic Hawking singularity theorem, and the equivalence of the limiting and localized timelike curvature-dimension conditions. Finally, in the Appendix, we prove a timelike reversibility theorem, showing that forward $p$-essential nonbranching $TCD_p(K,N)$ spacetimes are backward $p$-essentially nonbranching.

\section{Setup}\label{Ss:setup}

\subsection{Lorentzian length spaces}

\hspace*{4.5mm} We first recall the basic notions of synthetic Lorentzian geometry. We  work in the setting of \emph{globally hyperbolic Lorentzian length spaces} (hereafter \emph{gh LLS}) introduced by
 Kunzinger and S\"amann
\cite{KunzingerSaemann2018}, and  we will follow the  equivalent formulation of gh LLSs given by McCann \cite{McCann2024}.

Let $(X,d)$ be a Polish and proper metric space, and define the \emph{time separation function $\ell$} as a function $\ell:X^2 \rightarrow [0,\infty] \cup \{-\infty \}$ satisfying $\ell(x,x)\geq 0$ and the reverse triangle inequality $$ \ell(x,z)\geq \ell(x,y)+\ell(y,z) \quad \forall x,y,z\in X$$
whenever $\min\{\ell(x,y),\ell(y,z)\}>-\infty$.

\indent The time separation function $\ell$ induces a causal structure on the set $X$. We say that $x$ lies in the \emph{causal past} of $y$, written $x\leq y$ if $\ell(x,y)\geq 0 $, and 
$x$ lies in the \emph{timelike past} of $y$,  written $x\ll y$, if $\ell(x,y)>0$. We will also denote $x<y$ whenever $x\leq y$ and $x\neq y$. We then define $$X^{{2}}_{\leq}:=\{(x,y)\in X\times X: x\leq y\},$$ $$X^{2}_{\ll}:=\{(x,y)\in X\times X: x\ll y\}.$$

For each $x\in X$,  the \emph {timelike and causal future,} $I^+(x),$ and $J^+(x)$ respectively, are defined as $$I^+(x):=\{ y\in X:x\ll y\},\quad J^+(x):=\{ y\in X:x\leq y\}$$
and the \emph{timelike and causal past}, respectively, by  
$$I^-(x):=\{ y\in X:y\ll x\},\quad J^-(x):=\{ y\in X:y\leq x\}.$$
These definitions can be naturally extended to sets, e.g. $I^+(A)=\cup_{x\in A} I^{+}(x)$. 

We will require that the time separation function $\ell$ satisfies the \emph{anti-symmetry} condition: for every $x,y\in X,$  $$\min\{\ell(x,y),\ell(y,x)\}>-\infty \quad \iff \quad x=y,$$ which implies that the relations $X_{\ll}\subset X_{\leq}$ are antisymmetric \cite[Lemma 1]{McCann2024}. 

Following McCann \cite{McCann2024}, we refer to the triple  $(X,d,\ell)$ as a \emph{metric spacetime}, where $(X,d)$ is a metric space equipped with its metric topology, and $\ell$ is the time-separation function satisfying the above assumptions.

The causal structure naturally extends to paths. We define a path $\sigma:[0,1]\rightarrow X$ to be $causal$ if $\ell(\sigma(s), \sigma (t))\geq 0 $ for all $0\leq s<t\leq 1$. Paths for which the inequality  is always strict are called  $timelike$, and paths $\sigma$ satisfying $\ell(\sigma(s),\sigma(t))=0$ for all $s<t$ are called $null$. Following  McCann \cite{McCann2024}, we define a \emph{timelike $\ell$-path} as a path $\sigma:[0,1]\rightarrow X$ that satisfies
\begin{equation}\label{l path}
	\ell(\sigma(s),\sigma(t))=(t-s)\ell(\sigma(0),\sigma(1)) \not\in \{0,\pm \infty\} \quad \forall \, 0\leq s<t\leq 1.
\end{equation} We additionally assume that  the metric spacetime is \emph{regular} \cite{McCann2024}, which means that no maximizing path $\sigma:[a,b]\rightarrow X$ with $\ell(\sigma(a),\sigma(b))>0$ contains a non-constant null subsegment, where maximization is with respect to the $\ell$-length function  $$ L_{\ell}(\sigma) := \inf \big\{ \sum_{k=1}^N\ell(\sigma(t_{k-1}),\sigma(t_{k})) |\,\, N\in \N,\, a = t_0<\cdots<t_N=b  \big\}.$$ It follows that such paths are continuous \cite[Lemma 5]{McCann2024}, and after reparameterization to eliminate any intervals where $\sigma$ is constant, such maximizing paths become timelike.
We thus define a \emph{timelike $\ell$-geodesic} as a timelike $\ell$-path (hence $d$-continuous). We say that the space $X$ is a \emph{timelike path-connected space} whenever $x\ll y$ implies the existence of a timelike $\ell$-path from $x$ to $y$.  We will denote the space of all timelike $\ell$-geodesics as 
$$
TGeo^{\ell}(X)=\{\sigma\in C([0,1],X): \eqref{l path} \textrm{ holds}\},
$$
which becomes a Polish in the topology induced by the uniform metric $d_{\infty}:$ $$d_{\infty}(\sigma,\tilde{\sigma}):=\sup_{s\in[0,1]}d(\sigma(s),\tilde{\sigma}(s)).$$

One important ingredient that is used throughout the work is the
continuity of non-negative part of the time separation function $\ell_+:=\max\{\ell,0 \}$ whenever the space
is globally hyperbolic \cite{KunzingerSaemann2018}. We will adopt the following equivalent definition of global hyperbolicity due to McCann \cite[Lemma 10]{McCann2024}.
\begin{definition}\label{gh LLS}
A metric spacetime $(X,d,\ell)$ is called a \emph{globally hyperbolic Lorentzian length space} (gh LLS hereafter) whenever the following conditions hold:

\begin{enumerate}
\item If $A,B$ are compact sets,  then the set $J^{+}(A)\cap J^{-}(B)$ is compact, and  there exists a constant $C$ such that every causal path in $J^{+}(A)\cap J^{-}(B)$
has its $d$-length bounded by $C$;
\item $I^{\pm}(\{x\})$ $\ne$ $\emptyset$ for every $x\in X$;
\item $X$ is timelike path-connected;
\item $\ell_+ := \max\{\ell,0\} \in C(X^2;\R)$ and $\ell$ is upper semicontinuous.
\end{enumerate}

\end{definition}

\noindent Given a set $S\subset X_{\leq}^2$, and $s\in [0,1]$ we will denote the set of $s$-midpoints  as 
\begin{equation}\label{Midpoints}
Z_s(S):=\cup_{(x,y)\in S} Z_s(x,y), \quad \textrm{where} 
\end{equation}
 \begin{equation}
Z_s(x,y):=\{z\in X: \ell (x,z) = s \ell (x,y),\quad  \ell(z,y)=(1-s)\ell(x,y)\}.
\end{equation} Under global hyperbolicity, if $S$ is compact, then $Z_s(S)$ and $Z(S):=\cup_{s\in[0,1]}Z_s(S)$   are also compact \cite[Proposition 1.6]{CavallettiMondino2024}.

\begin{remark}
    This definition of a gh LLS implies that the causal relation $X_\leq^2$ is closed in $X^2$ \cite[Lemma 10]{McCann2024}.
\end{remark}

Throughout the paper, we will require the notion of a causally reversed metric spacetime. We say that $(X,d,\ell)$ has a \emph{causally-reversed structure} $(X,d,\bar{\ell})$ if $\bar{\ell}(x,y) := \ell(y,x)$ $\forall x,y\in X.$

\subsection{Lorentzian optimal transport }
The study of Ricci curvature bounds, as well as the formulation of the Timelike Curvature Dimension condition using Lorentzian optimal transportation has been initiated by the work of McCann \cite{McCann2020}, and Mondino-Suhr \cite{Mondino-Suhr} in the smooth setting, extended to nonsmooth spacetimes further by the work of Cavalletti-Mondino \cite{CavallettiMondino2024, Cavalletti-Mondino2025+} and Braun-McCann \cite{BraunMcCann+}. We  recall the notions of Lorentzian optimal transport, which we will use throughout this work.  We define a \emph{measured metric spacetime} as $(X,d,\m,\ell),$ where $(X,d,\ell)$ is a metric spacetime, and $\m$ is a fixed Radon measure with full support. $\spt \m = X.$
For a metric space $(X,d),$ we  denote the space of Borel probability measures by  $\P(X)$,  and $\P_c(X) \subset \P(X)$  are probability measures with compact support.

 Given a metric spacetime $(X,d,\ell),$ and
$\mu,\nu\in\P(X),$ the causal structure allows us to define the following inclusion of joint probability measures 
$$\Pi_{\ll}(\mu,\nu) \subset \Pi_{\leq}(\mu,\nu) \subset \Pi(\mu,\nu) \subset \P(X\times X),$$ where:

\[
\Pi(\mu,\nu):=\{\pi\in\P(X\times X):(P_{1})_{\#}\pi=\mu,(P_{2})_{\#}\pi=\nu\}
,\]

\[
\Pi_{\leq}(\mu,\nu):=\{\pi\in\Pi(\mu,\nu):\pi(X_{\leq}^{2})=1\},
\]

\[
\Pi_{\ll}(\mu,\nu):=\{\pi\in\Pi(\mu,\nu):\pi(X_{\ll}^{2})=1\},
\]

\noindent with $P_i$ denoting the projection onto the $i$'th coordinate. We refer to the elements of $\Pi_\leq (\mu,\nu)$ and $\Pi_\ll(\mu,\nu)$ as \emph{causal} and \emph{chronological} couplings, respectively.

Given $\mu,\nu\in\P(X)$ and $p\in(0,1]$, we define the \emph{p-Lorentz-Wasserstein}
$distance$ $\ell_{p}$   by 

\begin{equation} \label{p-Lor-Was}
\ell_{p}(\mu,\nu):=\sup_{\pi\in\Pi_{\leq}(\mu,\nu)} \bigg(\int_{X\times X}\ell(x,y)^{p}\pi(dxdy)\bigg) ^{1/p},
\end{equation} where 
 we set $\ell_{p}(\mu,\nu)=-\infty$ if $\Pi_{\leq}(\mu,\nu)=\emptyset.$  

We additionally adopt the convention $$(-\infty)^{\frac{1}{p}}=-\infty = (-\infty)^p.$$ 

Given any pair $\mu,\nu\in\P(X),$ the $primal$ $problem$ is to
find a $\pi\in\Pi_{\leq}(\mu,\nu$) that attains the supremum for
$\ell_{p}.$ A coupling $\pi\in\Pi_{\leq}(\mu,\nu)$ that maximizes
\eqref{p-Lor-Was} is said to be \emph{$\ell_{p}$-optimal.} The set of all
$\ell_{p}$-optimal couplings is denoted by $\Pi_{\leq}^{p \text{-}opt}(\mu,\nu),$ and we moreover define $\Pi_\ll ^{p\textrm{-opt}}(\mu,\nu):=\{ \pi \in \Pi_\leq ^{p\textrm{-opt}}(\mu,\nu): \pi(X_\ll)=1\}.$

For the corresponding dual problem, we will need the notion of timelike $p$-dualisability for the pair
$(\mu,\nu) \in \P(X)^2$. This notion is essentially that of McCann \cite[Section 7]{McCann2020}, however our  terminology is  due to  Cavalletti-Mondino \cite[Def. 2.18] {CavallettiMondino2024}.
\begin{definition}\label{dualisable}
Let $p\in(0,1].$ The pair $(\mu,\nu)\in\P(X)^{2}$ is \emph {timelike
p-dualisable} by $\pi\in\Pi_{\ll}(\mu,\nu$) if:

\begin{enumerate}
\item $\ell_{p}(\mu,\nu)\in(0,\infty)$;
\item $\pi\in\Pi_{\leq}^{p \text{-} opt}(\mu,\nu)$;
\item there exist measurable functions $a,b:X\rightarrow\R$ with $a\oplus b\in L^{1}(\mu\times\nu)$
such that $\ell^{p}\leq a\oplus b$ on $spt\mu\times spt\nu$.
\end{enumerate}
\end{definition}

For a timelike $p$-dualisable pair $(\mu,\nu)\in \P(X)^2$ we define the \emph{Kantorovich dual 
problem}  as
\begin{equation}\label{Kantorovich Dual Problem} K_p(\mu,\nu):=\inf\int_{X}ud\mu+\int_{X}vd\nu,
\end{equation}
\noindent where the infimum is over the set of measurable functions $u:\spt\mu\rightarrow\R\cup\{+\infty\},$
$v:\spt\nu\rightarrow\R\cup\{+\infty\}$ with $u\oplus v\geq \frac{1}{p}\ell^{p}$
on $\spt\mu\times\spt\nu,$ with $u\oplus v\in L^1(\mu \otimes \nu)$.

The functions that attain the infimum in the dual problem belong to
a special family of functions known as $\ell^{p}$-convex functions
\cite[Def. 2.22]{CavallettiMondino2024}.
\begin{definition}
($\ell^{p}$-concave functions, $\ell^{p}$ subdifferential) Fix $p\in(0,1]$
and $U,V\subset X$. The measurable function $\varphi:U\rightarrow\R$
is called \emph{$\ell^{p}$-concave} relative to $(U,V)$ if there exists a function
$\psi:V\rightarrow\R\cup \{ +\infty\}$ such that 
\[
\varphi(x)=\inf_{y\in V}\psi(y)-\ell^{p}(x,y)\quad\forall x\in U.
\]
The function $$ \varphi^{\ell^p}:V \rightarrow \R \cup\{-\infty\}, \quad \varphi^{\ell^p}(y):= \sup_{x\in U} \varphi(x)+ \ell^p(x,y) $$
is called the \emph{$\ell^p$-transform of $\varphi$}. We define the $\ell^p$-subdiferential of $\varphi$ by $$\partial^{\ell^p}\varphi:=\{ (x,y)\in U\times V\cap X^2_\leq: \varphi^{\ell^p}(y)-\varphi(x)=\ell(x,y)^p\}.$$
\end{definition}

Given $p\in(0,1]$,
we say that $(\mu,\nu)\in\P(X)^{2}$ satisfies \emph{strong $\ell^{p}$-Kantorovich
duality} if $\ell_{p}(\mu,\nu)\in(0,\infty)$, and there exist $A_{1},A_{2}\subset X$  with $\mu(A_{1})=\nu(A_{2})=1$
and a $\varphi:A_{1}\rightarrow\R$ which is $\ell^{p}$-concave relative to $(A_{1},A_{2})$ and satisfies 
\[
\ell_{p}(\mu,\nu)^p=\int_{X}\varphi^{\ell^{p}}(y)d\nu(y)-\int_{X}\varphi(x)d\mu(x).
\]

We will need the following version of cyclical monotonicity \cite{CavallettiMondino2024}.
\begin{definition}
Let $p\in(0,1]$ and let $(X,d,\ell)$ be a gh LLS. The set $\Gamma \subset X^2_{\leq}$ is  \emph{$\ell^p$-cyclically monotone} if for any $N\in \mathbb{N}$ and any collection of points $(x_1,y_1),...,(x_N,y_N)$ in $\Gamma$, we have 
$$ \sum_{i=1}^{N} \ell(x_i,y_i)^p \geq  \sum_{i=1}^{N} \ell(x_{i+1},y_i)^p, $$
with the convention that $x_{N+1}=x_1$.
\end{definition}

While timelike $p$-dualisability guarantees the existence of chronological optimal couplings, it does not ensure that every optimal coupling is chronological. To distinguish these two situations, we recall the notion of \emph{strong timelike $p$-dualisability} of Cavalletti--Mondino \cite{CavallettiMondino2024}, which characterizes optimal couplings through a common $\ell^p$-cyclically monotone set contained in the chronological relation.
\begin{definition}
Let $p\in(0,1]$. The pair of measures $(\mu,\nu)\in \P(X)^2$ is \emph{strongly timelike p-dualisable} if it is timelike $p$-dualisable, and 
there exists a measurable $\ell^p$-cyclically monotone set $\Gamma\subset X^2_{\ll}\cap (\spt\mu\times \spt \nu)$ such that for every  coupling $\pi\in \Pi_{\leq}(\mu,\nu)$, we have that  $\pi$ is $\ell_p$-optimal if and only if $\pi(\Gamma)=1$.

\end{definition}
A sufficient condition on the pair of measures that guarantees equality of the primal and dual problems is summarized by the following result   \cite[ Cor. 2.29]{CavallettiMondino2024}.
\begin{corollary}\label{duality}
Let $(X,d,\ell)$ be a
gh LLS, and $p\in(0,1].$ Given $\mu,\nu\in \P(X)$ assume that $\spt\mu\times\spt\nu\subset X_{\ll}^{2}$, and there exist measurable $a,b$ such that $\ell^p\leq a\oplus b$ on $\spt \mu \times \spt \nu$, and $a\oplus b\in L^1(\mu\otimes\nu)$.
Then $(\mu,\nu)$ satisfy the strong $\ell^{p}$-Kantorovich duality and
are strongly timelike $p$-dualisable. 
\end{corollary}

We will consider Borel probability measures  concentrated on $TGeo^{\ell}$, which we denote by $\P(TGeo^{\ell})$. Let $p\in (0,1]$, and suppose that the pair $(\mu,\nu)$  is timelike $p$-dualisable. Define the set of \emph{optimal timelike dynamical transport plans}  as

$$
OptTGeo_p^{\ell}(\mu,\nu):=\{ \eta \in \P(TGeo^{\ell}): (e_0,e_1)_{\#} \eta \in \Pi^{p\textrm{-opt}}_{\ll}(\mu,\nu) \},
$$
where $e_t:TGeo^\ell(X)\rightarrow X$ is the evaluation map, $e_t(\gamma):=\gamma_t$. 
As shown by McCann, \cite[Lemma 14]{McCann2024},  for compactly supported measures $\mu,\nu\in\P_c(X)$, every $\eta\in OptTGeo_p^\ell(\mu,\nu)$ induces an $\ell^p$-geodesic $\mu_t:=(e_t)_\# \eta,$ for $t\in [0,1],$ satisfying $$\ell_p(\mu_s,\mu_t) =(t-s)\ell_p(\mu,\nu),\quad 0\leq s<t\leq 1,$$ where moreover $t\mapsto \mu_t$ is continuous with respect to the Kantorovich distance $d_1$ \eqref{Kantorovich} on $\P_c(X)$ induced by the distance $d$ on $X$.

We next recall the timelike essentially non-branching conditions used throughout the paper. We say that a set $A\subset TGeo^{\ell}$ is  \emph{forward timelike non-branching} if for any $\gamma^1,\gamma^2 \in A$ the following holds:
\begin{equation}\label{forward-branching}
\exists t \in(0,1) \:\: \gamma^1(s)=\gamma^2(s) \:\:\forall s\in [0,t] \quad   \, \implies \gamma^1(s)=\gamma^2(s),  \:\:\: \forall s \in [0,1],
\end{equation}
and  define analogously the notion of backwards timelike nonbranching.
Finally, given $p\in(0,1)$, we say that the spacetime $(X,d, \m,\ell)$  is \emph{forward timelike p-essentially nonbranching},  if for every compactly supported timelike $p$-dualisable pair $(\mu_0,\mu_1)$  satisfying $\mu_0,\mu_1\ll \m$, 
we have that every element of  $OptTGeo_p^{\ell}(\mu_0,\mu_1)$ is concentrated on a set of forward nonbranching timelike geodesics. By considering the causally reversed structure, we analogously define the \emph{backward timelike p-essentially nonbranching} condition. Finally, we say that the spacetime is \emph{$p$-essentially non-branching} if it is both forward and backward $p$-essentially nonbranching. We shall refer to these properties as \emph{forward $p$-enb, backwards $p$-enb, and $p$-enb respectively.}

We next recall the forward timelike measure contraction property $TMCP^+(K,N),$ which provides the synthetic curvature bounds that we will use throughout this paper.  To this end, we first define the \emph{R\'enyi entropy} $S_N:\P(X)\rightarrow [-\infty,0]$ as 
\begin{equation}\label{Renyi}
    S_N(\mu|\m):= -\int _X \rho ^{1-\frac{1}{N}}d\m
\end{equation}
where $\mu = \rho d\m+\mu^{sing}$ is the Lebesgue decomposition of $\mu,$ and  $N\in [1,\infty)$ is the dimension parameter. 

To define the curvature $K\in \R$ and dimension $N\in (1,\infty]$ distortion coefficients, we first define the profile function $s_K(\theta):[0,D_K)\rightarrow \R$ by 

$$ s_K(\theta):=
	\begin{cases}
		\frac{1}{\sqrt{K}}\sin(\sqrt{K}\theta), &  K>0, \\
		\theta,  & K=0 \\
        \frac{1}{\sqrt{-K}}\sinh(\sqrt{-K}\theta),&K< 0,
	\end{cases}$$

\noindent where $D_K:=\pi/\sqrt{K}$ for $K>0$ and $D_K:=+\infty$ for $K\leq 0.$ The profile function $s_K$ is the unique solution to the ODE \begin{equation}\label{profile function}\ddot{s}_K+Ks_K = 0,\quad s_K(0)=0,\quad \dot{s_K}(0)=1.\end{equation} The synthetic dimension parameter $N$ is incorporated into $D_K$ by defining $D_{K,N}:=\sqrt{N}D_K$, and we define \emph{distortion ratio} $\sigma_{K,N-1}$ for $t\in[0,1]$ and $0<\theta<D_{K,N-1}$  as $$ \sigma^{(t)}_{K,N-1}(\theta) : = \frac{s_K(t\theta/\sqrt{N-1})}{s_K(\theta/\sqrt{N-1})},$$ where $\sigma^{(t)}_{K,N-1}(0):=t$, and $\sigma^{(t)}_{K,N-1}(\theta):=+\infty$ for $\theta\geq D_{K,N-1}.$ Finally, we define  the \emph{distortion coefficient} $\tau_{K,N}^{(t)}:$ $$\tau_{K,N}^{(t)}(\theta):=t^{\frac{1}{N}}\sigma^{(t)}_{K,N-1}(\theta)^{1-\frac{1}{N}},\quad t\in[0,1],\quad 0<\theta<D_{K,N-1},$$ and when $N=1,$ we set $\tau_{K,1}^{(t)}(\theta):=t$ if $K\leq 0$ and $+\infty$ if $K>0.$

We adopt the forward version of $TMCP^+(K,N)$ condition from the works of Braun \cite[Def 4.1]{BraunTCD}, and Beran et al \cite[Def 5.1]{octet+}, which controls the evolution of an absolutely continuous measure towards a fixed point in its chronological future through a R\'enyi entropy inequality.

\begin{definition}\label{TMCP} Let $K\in \R, N\in [1,\infty)$, and fix$p\in(0,1).$ The measured gh LLS $(X,d,\m,\ell)$ obeys the $TMCP^+(K,N)$ condition if for every absolutely continuous $\mu_0\in \P_c(X)$, that is $d\mu_0 = \rho_0 d\m$, and $\mu_1=\delta_{x_1},$ where $\spt \mu_0 \subset I^-(x_1)$, there exists an optimal timelike dynamical transport plan $\eta\in OptTGeo_p^{\ell}(\mu_0,\mu_1)$ such that for every $t\in[0,1)$ and every $N'\geq N$ we have:
\begin{equation}\label{Renyi Entropy Inequality}
    S_{N'}(\mu_t)\leq -\int_X\tau^{(1-t)}_{K,N'}(\ell(x,x_1))\rho_0(x)^{1-\frac{1}{N'}}d\m(x).
\end{equation}
where $\mu_t = (e_t)_\#\eta.$
\end{definition}
\begin{remark}\label{indep of p}
The above definition is independent of $p\in(0,1)$, and in particular the $TMCP^+(K,N)$ condition is independent of $p\in(0,1)$, see \cite[Rmk. 3.8]{CavallettiMondino2024} and \cite[Rmk. 4.2]{BraunTCD}.
\end{remark}
The forward evolution of $\mu_0$ towards $\delta_{x_1}$ in $TMCP^+(K,N)$ is emphasized by the "+", and we define the $TMCP^-(K,N)$ condition if the causal reversal of the spacetime satisfies  $TMCP^+(K,N).$ We will denote the spacetimes satisfying both the forward and backward timelike measure contraction property by $TMCP(K,N).$

Finally, we will require the following  extension of \cite[Theorem 4.16, 4.17]{BraunTCD}, where we establish the existence and uniqueness of chronological optimal couplings and dynamical plans under only the forward $p$-enb assumption. To avoid interrupting the flow of ideas, we defer the proof of this Theorem to the Appendix.
\begin{theorem}\label{Brenier Map}
    Set $K\in \R, N\in [1,\infty), p\in(0,1),$ and let $(X,d,\m,\ell)$ be a forward $p$-enb measured gh LLS that satisfies the $TMCP^+(K,N)$ condition. Let $\mu_0$ be absolutely continuous and let $(\mu_0,\mu_1)\in \P_c(X)^2$ be timelike $p$-dualisable. Then, there exists a unique $\ell^p$-optimal coupling $\pi \in \Pi_{\ll}^{p\textrm{-opt}}(\mu_0,\mu_1)$ induced by a a $\mu_0$-measurable map $T:X\rightarrow X$ such that $\pi = (Id,T)_\#\mu_0$ and $$\ell_p(\mu_0,\mu_1)^p=\int \ell(x,T(x))^pd\mu_0(x).$$ Moreover, there exists a unique chronological $\ell^p$-optimal dynamical plan $\boldsymbol{\alpha}\in OptTGeo^\ell_p(\mu_0,\mu_1)$ induced by a $\mu_0$-measurable map $G:X\rightarrow TGeo^\ell(X)$ such that $\boldsymbol{\alpha}=G_\# \mu_0,$ and its intermediate marginals satisfy $(e_t)_\# \boldsymbol{\alpha}\ll \m,$ for $t\in[0,1).$ 
    \end{theorem}
Throughout the rest of this paper, we will work on a  \emph{measured gh LLS} $(X,d,\m,\ell)$.
In the next section, we will specialize to the Lorentzian $L^1$ transport problem, corresponding to $p=1$. In contrast to strictly concave costs $\ell^p$ with $p\in (0,1)$, the cost $\ell$ admits a special Kantorovich dual formulation in terms of reverse 1-Lipschitz functions.

\section{Kantorovich duality for the Lorentzian \texorpdfstring{$L^1$}{L¹} cost}\label{Section 2}

 In this section, we specialize to $p=1$, with the aim of classifying the solutions to the Kantorovich dual problem \eqref{Kantorovich Dual Problem} for the Lorentzian time-separation cost $\ell$. Our goal is to construct  continuous reverse 1-Lipschitz Kantorovich potentials associated with the $\ell^1$ optimal transport problems.  The proofs are
 inspired by the pioneering work of Caffarelli-Feldman-McCann and Feldman-McCann in the Riemannian signature \cite{CaffarelliFeldmanMcCann2002, Feldman-McCann}.
Throughout this section, we fix a gh LLS $(X,d,\ell)$.

\begin{definition}
Let $Z\subset X.$ The function $u:Z\rightarrow\R$ is said to be \emph{reverse 1-Lipschitz on $Z$}
if 

\[
x,y\in Z,\,\,\, x\leq y\implies u(y)-u(x)\geq\ell(x,y).
\]
\end{definition}

The use of reverse 1-Lipschitz functions in the context of one-dimensional localization already appeared in the works of Cavalletti-Mondino \cite{CavallettiMondino2024, Cavalletti-Mondino2025+} and Braun-McCann \cite[Appendix C]{BraunMcCann+}. We will use the reverse Lipschitz condition to characterize the solution of the Kantorovich dual problem in the Lorentzian setting.

Given a compactly supported chronological optimal coupling $\pi,$  we define its \emph{transport interpolation set} by  $$Z:=\bigcup_{s\in[0,1]}Z_s(\spt \pi)$$  We now show that, under suitable assumptions on the Kantorovich dual potentials, these two potentials combine into a single continuous reverse 1-Lipschitz function on $Z$.

\begin{theorem}[Reverse Lipschitz Maximizer]\label{Kantorovich Potential}
Let $(X, d, \m, \ell)$ be a gh LLS, and  
let $\mu, \nu \in \mathcal{P}_c(X)$ be timelike $1$-dualisable by 
$\pi \in \Pi_\ll(\mu, \nu)$. Suppose there exist lower semicontinuous functions
$\varphi: U \to \mathbb{R} \cup \{+\infty\}$, 
$\psi: V \to \mathbb{R} \cup \{+\infty\}$ 
with $U = \spt \mu$, $V = \spt \nu$ such that
\begin{align}
    \varphi(x) + \psi(y) &\ge \ell(x,y) \quad \forall (x,y) \in U \times V, \label{eq:sep-ineq} \\
    \operatorname{spt}\pi \subset S &:= \{(x,y) \in U \times V : 
    \varphi(x) + \psi(y) = \ell(x,y)\} \subset \{\ell>0 \}. \label{eq:sep-eq}
\end{align}

Then $\varphi$ and $\psi$ are continuous on $U$ and $V$ respectively. 
Define the associated transport interpolation set $Z$ as
\begin{equation*}
    Z := \bigcup_{s \in [0,1]} Z_s(\operatorname{spt}\pi).
\end{equation*}
Then $Z$ is  compact, and there exists a continuous reverse-1-Lipschitz function $u:Z\rightarrow \R$ such that 
\begin{equation*}
    u = -\varphi \quad \text{on } U \cap Z, \qquad u = \psi \quad \text{on } V \cap Z.
\end{equation*}
\end{theorem}

\begin{proof}

Since $U,V$ are compact, $\varphi,\psi$ are lower semicontinuous, and $\ell$ is upper semicontinuous, the set $S$ is closed in $U\times V$ and hence compact. 

Global hyperbolicity then gives us that the transport interpolation set $Z$ is compact.

 Define $\tilde{\varphi}: Z \to \mathbb{R}$ and $\tilde{\psi}: Z \to \mathbb{R}$ by
\begin{align}
    \tilde{\varphi}(z) &:= \sup_{y \in V}  \ell(z, y) - \psi(y) , \label{eq:ctrans-phi} \\
    \tilde{\psi}(z) &:= \sup_{x \in U}  \ell(x, z) - \varphi(x) , \label{eq:ctrans-psi}
\end{align}

 \noindent and notice that since  $\ell(\cdot, y) = -\infty$ on non-causal pairs, the supremum only considers the  
 $y \in V$ such that $\ell(z, y)\geq0. $ For $z \in Z$, there exists at 
least one such $y$ (since $Z$ is defined as interpolating points between optimal 
pairs), so the suprema are finite on $Z$.

We now show that $\tilde{\varphi}$ and $\tilde \psi $ are upper semi continuous. Since $\ell$ is upper semicontinuous, and $\varphi,\psi$ are lower semicontinuous, the functions $$(z,y)\mapsto \ell(z,y)-\psi(y),\quad (x,z)\mapsto \ell(x,z)-\varphi (x)$$ are upper semicontinuous. Since $U,V$ are compact, taking the supremum over $V$ and $U$ respectively preserves upper semicontinuity, and hence $\tilde\varphi, \tilde \psi$ are upper semicontinuous on $Z$. 

We now relate $\tilde{\varphi}$ and $\tilde{\psi}$ to construct the reverse-Lipschitz $u:Z\rightarrow \R.$ For every $z \in Z$, there exists an optimal 
pair $(x_0, y_0) \in \operatorname{spt}\pi$ and $s \in [0,1]$ such that 
$z \in Z_s(x_0, y_0)$. Then $\ell(x_0, z) = s \ell(x_0, y_0)$ and 
$\ell(z, y_0) = (1-s) \ell(x_0, y_0)$, and 
$\varphi(x_0) + \psi(y_0) = \ell(x_0, y_0)$. By definition of $\tilde{\varphi},\tilde{\psi} $, we get $$
    \tilde{\varphi}(z) \ge \ell(z, y_0) - \psi(y_0) 
    = (1-s)\ell(x_0, y_0) - \psi(y_0),$$ $$ 
    \tilde{\psi}(z) \ge \ell(x_0, z) - \varphi(x_0) 
    = s \ell(x_0, y_0) - \varphi(x_0).
$$
Adding the above two inequalities gives us:
$$    \tilde{\varphi}(z) + \tilde{\psi}(z) \ge \ell(x_0, y_0) - \varphi(x_0) - \psi(y_0) 
    = 0,$$ while on the other hand, given any $y \in V$ and $x \in U$,
$$    \big[ \ell(z, y) - \psi(y) \big] + \big[ \ell(x, z) - \varphi(x) \big]
    = \ell(z, y) + \ell(x, z) - (\varphi(x) + \psi(y)),
$$so the reverse triangle inequality and ~\eqref{eq:sep-ineq} gives us 
$$
    \ell(z, y) + \ell(x, z) \le \ell(x, y) \le \varphi(x) + \psi(y).
$$ Thus, we have that 
$$    \big[ \ell(z, y) - \psi(y) \big] + \big[ \ell(x, z) - \varphi(x) \big] \le 0,
$$ and taking the supremum over $y \in V$ and then over $x \in U$ gives
$$
    \tilde{\varphi}(z) + \tilde{\psi}(z) \le 0.
$$

\noindent Thus $\tilde{\varphi}(z) + \tilde{\psi}(z) = 0$ for all $z \in Z$. Since both functions are upper semicontinuous, this equality implies that they are also lower semicontinuous, hence continuous. We define 
\begin{equation*}
    u(z) := -\tilde{\varphi}(z) = \tilde{\psi}(z), \quad z \in Z.
\end{equation*}
It follows that  $u: Z \to \mathbb{R}$ is continuous as $\tilde{\varphi}$ is continuous. To see that $u$ is reverse 1-Lipschitz, take any $z,w\in Z$ with $z\leq w$, and apply the reverse triangle inequality: $$u(w)=\sup_{x\in U}\ell(x,w)-\varphi(x)\geq \sup_{x\in U}\ell(x,z)+\ell(z,w)-\varphi(x)=u(z)+\ell(z,w).$$ Together with the dual inequality \eqref{eq:sep-ineq} and the equality in $\spt \pi$, we get $u=-\varphi$ in $U\cap Z$ and $u=\psi$ in $V\cap Z$. Since the transformed potentials are continuous, this also establishes the continuity of $\varphi$ and $\psi$.
\end{proof}

The preceding construction also extends to the transport interpolation set associated with the entire contact set $S$ in \eqref{eq:sep-eq}. Indeed, defining the compact set $$Z_S:=\bigcup_{s\in[0,1]} Z_s(S),$$ we can extend the 1-Kantorovich potential $u$ to a continuous reverse 1-Lipschitz function $Z_S$. Thus,  whenever $z,w\in Z_S$ satisfies $u(w)-u(z)=\ell(z,w)>0$, every maximizing timelike geodesic $\g$ from $z$ to $w$ extends to a maximizing geodesic joining some pair $(a,d)\in S$. In particular, $\g([0,1])\subset Z_S$ and we have \begin{equation}\label{propagation}
u(\g_t)-u(\g_s)=\ell(\g_s,\g_t),\quad 0\leq s\leq t\leq 1,
\end{equation}
which corresponds to the propagation of the Kantorovich equality along its subsegments.

We extend the preceding construction to a general chronological $\ell_1$ optimal couplings $\pi$ satisfying $\ell>0$ $\pi$-a.e., by decomposing them into countably many optimal pieces with compact strictly chronological endpoint supports. On each normalized piece, the hypotheses of Theorem \ref{Kantorovich Potential} are satisfied. This follows the approach of McCann \cite[Theorem 7.1]{McCann2020}.

\begin{theorem}\label{decomposing pi}
Let $(X, d, \mathfrak{m}, \ell)$ be gh LLS, and let  $\mu,\nu \in \mathcal{P}(X)$ satisfy $0<\ell_1(\mu,\nu)<\infty$. Let $\pi\in \Pi_\leq ^{1\textrm{-opt}}(\mu,\nu)$ be an $\ell_1$-optimal coupling with $\ell>0$ holding $\pi$-a.e.
Then $\pi$ decomposes as a sum of mutually singular nonnegative measures
\begin{equation}
    \pi = \sum_{i=1}^\infty \pi^i,
\end{equation}
where each $\pi^i$ has compact support contained in $\{ \ell>0\}$, and the marginals of each normalized coupling
$\hat{\pi}^i = \pi^i / \pi^i[X^2]$ is $\ell_1$-optimal between its marginals. Moreover, each normalized coupling $\hat \pi ^i$ satisfies the hypotheses of Theorem \ref{Kantorovich Potential} and therefore admits a continuous reverse 1-Lipschitz Kantorovich potential  $$u^i:Z_i\rightarrow \R,\quad Z_i:=\bigcup_{s\in [0,1]}Z_s(\spt \pi^i).$$

Moreover, this decomposition is finite if and only if $\operatorname{spt}\pi$ is compact 
and contained in $X_\ll^2$
\end{theorem}

\begin{proof}

Since $\ell > 0$ holds $\pi$-a.e., we have $\pi(X^2_\ll) = 1$, and  
$X^2_\ll = \{\ell > 0\}$ is an open set. By the gh LLS assumption, X is locally compact. Indeed, for any $x\in X$, choose $p\ll x\ll q$, for which we have $x\in I^+(p)\cap I^-(q)\subset J^+(p)\cap J^-(q)$, and the set $I^+(p)\cap I^-(q)$ is thus an open neighbourhood of $x$ while the causal diamond $J^+(p)\cap J^-(q)$ is compact by global hyperbolicity. Hence $x$ has a neighbourhood with compact closure, and therefore $X$ is locally compact. Since $X$ is  Polish, it is second countable, and  we can cover
$X^2_\ll$ by countably many open sets $U_k \times W_k$ such that
$$    \overline{U_k} \times \overline{W_k} \subset X^2_\ll,
$$with each $\overline{U_k} \times \overline{W_k}$ compact.

Set $\pi^0 = 0$. For $i = 1, 2, \dots$, define
\begin{equation}
    \pi^i :=  \bigg[\pi - \sum_{k=1}^{i-1} \pi^k\bigg] \restrict_{U_i \times W_i}.
\end{equation} Then each $\pi^i$ is a nonnegative Borel measure, the measures $\{ \pi^i\}$ are mutually singular, and we have $$\pi = \sum_{i=1}^\infty \pi ^i,\quad \spt \pi^i\subset \overline U_i\times \overline W_i\subset X_\ll^2.$$  For each nonzero $\pi^i$, denote its transport interpolation set by $Z_i:=\cup_{s\in[0,1]}Z_s(\spt\pi_i),$  let $\mu^i=(P_1)_\# \pi^i, \nu^i = (P_2)_\# \pi^i$, and normalize  
\begin{equation}
    \hat{\pi}^i := \frac{\pi^i}{\pi^i[X^2]}, \quad 
    \hat{\mu}^i := \frac{\mu^i}{\pi^i[X^2]}, \quad 
    \hat{\nu}^i := \frac{\nu^i}{\pi^i[X^2]}.
\end{equation} We claim that $\hat\pi^i$ is an $\ell_1$-optimal coupling between $\hat \mu^i$ and $\hat \nu^i$. Indeed, if it were not the case, there would exist a causal coupling $\eta\in \Pi_\leq (\hat \mu^i,\hat \nu^i)$ with strictly larger transport $\ell_1$ transport value, and then $\tilde \pi:=\pi-\pi^i+\pi^i(X^2)\eta$ would be a causal coupling between $\mu$ and $\nu$ with a strictly larger transport value than $\pi$, contradicting its $\ell_1$-optimality.

Since $\spt \hat \mu^i\times \spt \hat \nu^i\subset \overline U_i\times \overline W_i\subset X_\ll^2$, the function $\ell$ is finite and continuous on this compact product. Since $\hat \pi^i$ is $\ell_1$-optimal, the Kantorovich dual problem is attained, and there exist continuous potentials $u^i,v^i$ such that 
\begin{align}
    u^i(x) + v^i(y) &\ge \ell(x,y) \quad \forall (x,y) \in 
    \operatorname{spt}[\hat{\mu}^i \times \hat{\nu}^i], \\
    \operatorname{spt}\hat{\pi}^i &\subset S_i:=\{(x,y) \in \spt \hat \mu^i\times \spt \hat \nu^i: u^i(x) + v^i(y) = \ell(x,y)\} 
    \subset X^2_\ll,
\end{align}
Thus, $\hat \pi^i$ and $u^i,v^i$ satisfy the hypotheses of Theorem \ref{Kantorovich Potential}. Defining $Z_{S_i}:=\bigcup_{s\in[0,1]}Z_s(S_i)$, Theorem \ref{Kantorovich Potential} provides a continuous reverse $1$-Lipschitz Kantorovich potential $\bar u^i:Z_{S_i}\rightarrow\R$ satisfying $$\bar u^i(y)-\bar u^i(x)=\ell(x,y),\quad \forall (x,y)\in \spt \hat \pi^i.$$ Moreover, every maximizing timelike geodesic $\g$ joining $z,w\in Z_{S_i}$ with $\bar u^i(w)-\bar u^i(z)=\ell(z,w)>0$ is contained in $Z_{S_i}$ and satisfies $$\bar u^i(\g_t)-\bar u^i(\g_s) = \ell(\g_s,\g_t),\quad \forall 0\leq s\leq t\leq 1.$$

We now show that the decomposition can be chosen to be finite if and only if $\spt \pi$ is compact and contained in $X_\ll^2$. Assume $\spt \pi$ is compact and contained in $X_\ll^2$. 
Then $\spt \pi \subset X^2_\ll$, and since  $\spt \pi$ is 
compact and $X^2_\ll$ is open, there exists a finite open cover of 
$\spt\pi$ by open sets $\{U_k \times W_k\}_{k=1}^N$ with compact 
closures contained in $X^2_\ll$.

Arguing as before, we may choose the countable cover so that these 
$N$ open sets are the first $N$ elements. Then for every $i > N$, the open set 
$U_i \times W_i$ is disjoint from $\spt \pi$ up to a set of 
$\pi$-measure zero (because $\spt\pi$ is already covered by the 
first $N$ rectangles, and the inductive subtraction removes the mass already 
assigned). Therefore, $\pi^i = 0$ for all $i > N$. Thus the sum $\pi = \sum_{i=1}^N \pi^i$ is finite.

Conversely, assume the sum is finite: $\pi = \sum_{i=1}^N \pi^i$ with each $\pi^i \neq 0$. Then
$$    \spt\pi \subset \bigcup_{i=1}^N \spt\pi^i.
$$
Since each $\spt \pi^i$ is compact, and disjoint from $\{\ell\leq 0\}$, it follows that $\spt \pi$ is a closed subset of a compact set that is disjoint from $\{ \ell\leq 0\}.$
\end{proof}

The consequence of the above Theorems is that for any $\ell_1$-optimal transport coupling $\pi$ with finite transport value and satisfying $\ell>0$ $\pi$-a.e. admits a decomposition into countably many compactly supported pieces, each piece itself optimal to which we associate a continuous reverse 1-Lipschitz Kantorovich potential $u^i:Z_{S_i}\rightarrow \R$. in particular, defining  $$ \Gamma_{u^i}:=\{(x,y)\in(Z_{S_i}\times Z_{S_i})\cap X_\leq ^2: u^i(y)-u^i(x)=\ell(x,y)>0 \}\cup \Delta_{Z_{S_i}},$$ we have that $\pi(\bigcup_{i=1}^\infty \Gamma_{u^i}) = 1$. The continuous Kantorovich potentials $u^i$ are the starting point for the disintegration and localization program in the next section.

\section{Structure of the \texorpdfstring{$L^1$}{L¹} optimal transportation problem and the disintegration theorem}\label{Disintegration section}
The goal of this section is to perform a reduction of the $L^1$ transport problem onto its one dimensional counterparts. 
This Lorentzian construction has already appeared in greater generality in the works of Braun-McCann \cite{BraunMcCann+} and Cavalletti-Mondino \cite{CavallettiMondino2024, Cavalletti-Mondino2025+}. In particular, Braun-McCann work with a  Borel reverse Lipschitz function, and Cavalletti-Mondino specialize to the reverse Lipschitz functions that arise as the Lorentzian analogue of signed distance functions. In our construction, we will specialize to the case of continuous reverse Lipschitz functions, such as the Kantorovich potentials  from the previous section.  Thus, throughout this section, we fix a compact transport interpolation set $Z$ and a continuous reverse 1-Lipschitz function $u:Z\rightarrow \R.$ We assume that every maximizing timelike geodesic $\g$ joining $z,w\in Z$ with $u(w)-u(z)=\ell(z,w)>0$ is contained in $Z$: in particular, the preceding section show that this property holds when $Z=Z_S$ is the interpolation set of the full Kantorovich contact set $S$.

\begin{subsection}{Transport set and transport relation}\label{S: transport set and transport relation} To any continuous reverse 1-Lipschitz $u:Z\rightarrow \R$, we  associate a partial order which we call \emph{the transport order} $\Gamma_u$, defined as the following $\ell$-cyclically
monotone set
\end{subsection}

\[
\Gamma_{u}:=\{(x,y)\in (Z\times Z)\cap X_{\leq}^{2}:u(y)-u(x)=\ell(x,y)>0\} \cup\Delta_Z.
\] where $\Delta_Z:=\{(x,x):x\in Z \}$ denotes the diagonal.

Define the \emph{transport relation} $R_{u}$ as

\[
R_{u}:=\Gamma_{u}\cup\Gamma_{u}^{-1},
\]
where $\Gamma_u ^{-1}:=\{(x,y)\ \in Z\times Z: (y,x)\in \Gamma_u\}$, and also  define the \emph{transport set} $\T_{u}$ as follows
\[
\quad \T_{u}:=P_{1}(R_{u}\setminus\Delta_Z),
\]

We  define the set of \emph{initial and final points} respectively by 
$$ a_u:=\{z\in \T_u:\not\exists x \in \T_u, (x,z)\in \Gamma_u , x\neq z\}, $$
$$ b_u:=\{z\in \T_u:\not\exists x \in \T_u, (z,x)\in \Gamma_u , x\neq z\}. $$

\noindent These sets enjoy the following measurability properties: 

\begin{itemize}
    \item $\Gamma_u$ is $\sigma$-compact. Indeed, since $\Gamma_{u}:=\{(x,y)\in Z_{\leq}^{2}:u(y)-u(x)=\ell(x,y)>0\} \cup\Delta_Z$ and set
    \begin{equation}\label{sigma compact A}
    A :=\{(x,y)\in Z_{\leq}^{2}:u(y)-u(x)=\ell(x,y)>0\} = \bigcup_{n\geq 1}B_n, 
    \end{equation}
    where $ B_n = \{(x,y)\in Z_{\leq}^{2}:u(y)-u(x)=\ell^+(x,y)\geq\frac{1}{n}\}.$ We show that each $B_n$ is closed. Since $u$ is continuous and $\ell$ is continuous on $X^2_{\leq},$  the conditions  $u(y)-u(x)-\ell(x,y)=0$ and  $\ell(x,y)\geq \frac{1}{n}$ define a closed subset of $X^2_\leq$. Since $X^2_\leq$ is closed in $X^2$ it follows that $B_n$ is closed in $Z^2$. Since $X$ is locally compact and second countable, $X^2$ is hence $\sigma$-compact, and since each $B_n$ is closed in $X^2$,  each $B_n$ is $\sigma$-compact, and we can write $A$ as a countable union of $\sigma$-compact sets. Additionally, since $\Delta_Z$ is $\sigma$-compact, and thus $\Gamma_u = A\cup \Delta_Z$ is $\sigma$-compact.

    \item  Similarly, $\Gamma_u^{-1}$ and $R_u$ are $\sigma$-compact.
    \item $\T_u$ is a projection of a Borel set $R_u\setminus \Delta_Z $, therefore it is analytic. Additionally, $\T_u$ is $\sigma$-compact. Indeed, since $\Gamma_u = A\cup\Delta_Z,$ we have that $R_u = A\cup A^{-1}\cup \Delta_Z.$ Then $R_u\setminus\Delta_Z = A\cup A^{-1},$ which is a union of $\sigma$-compact sets, and its projection is thus $\sigma$-compact.
    \item $a_u,b_u$ are co-analytic sets. Indeed, consider the set $E := \{(x,z)\in \T_u\times \T_u: x\neq z, (x,z)\in \Gamma_u \} = (\T_u\times \T_u)\cap \Gamma_u\cap (X\times X\setminus \Delta_X).$ This is an analytic set since it is an intersection of an analytic set together with a Borel set $\Gamma_u\cap (X\times X\setminus \Delta_X)$. Defining the analytic set $A':=P_2(E),$ it follows that $a_u= \T_u\setminus A',$ hence $a_u$ co-analytic. The argument for $b_u$ is analogous.
\end{itemize}

The transport order $\Gamma_{u}$ imposes the following structure on timelike geodesics
with endpoints in $\Gamma_{u}.$
\begin{lemma} \label{geodesics align}
For any $\gamma\in TGeo^{\ell}$ such that $(\gamma_{0},\gamma_{1})\in\Gamma_{u}$
we see that 
\[
(\gamma_{s},\gamma_{t})\in\Gamma_{u}\quad\forall~0\leq s\leq t\leq1.
\]
\end{lemma}

\begin{proof}
By our assumption on $Z$, we have that $\g([0,1])\subset Z$. Take $0\leq s\leq t\leq1$, and note that 
\[
\begin{split}
u(\gamma_{t})-u(\gamma_{s}) & =  u(\gamma_{t})-u(\gamma_{s})+u(\gamma_{0})-u(\gamma_{0})+u(\gamma_{1})-u(\gamma_{1})\\
 & =  \ell(\gamma_{0},\gamma_{1})+u(\gamma_{0})-u(\gamma_{s})+u(\gamma_{t})-u(\gamma_{1})\\
 & \leq  \ell(\gamma_{0},\gamma_{1})-\ell(\gamma_{0},\gamma_{s})-\ell(\gamma_{t},\gamma_{1})\\
 & =  \ell(\gamma_{s},\gamma_{t}),
\end{split}
\]
and the reverse inequality holds by the reverse-Lipschitz condition.
\end{proof}

With the previous lemma in mind, we define $G$ as timelike geodesics whose endpoints are in the transport order: \begin{equation}\label{Geodesics in transport order} G:=\{\gamma\in TGeo^{\ell}: (\gamma_{0},\gamma_{1})\in\Gamma_{u} \}.\end{equation} By our standing assumption on $Z$, every $\g\in G$ is contained in $Z$, and Lemma \ref{geodesics align} gives us $(\g_s,\g_t)\in \Gamma_u$ for all $0\leq s\leq t\leq 1$. 

We will show that the transport relation $R_u$ is an equivalence relation over a subset of $\T_u$ that is of full measure. The equivalence classes will correspond to timelike geodesics, and we need to ensure that any possible branching structure is removed.  Forward timelike branching in $\Gamma_u$ is modeled by the existence of $x,z,w \in \T_u$ such that $$ (x,z), (x,w)\in \Gamma_u \cap X_{\ll}^2, \quad \textrm{}  \, (z,w)\not\in R_u,$$ and backwards branching is treated analogously by considering instead $\Gamma_u^{-1}$. We thus define the \emph{forward and backward branching} points $A_{+,u},A_{-,u}$
as

\[
A_{+,u}:=\{x\in\T_{u}:\exists(x,z),(x,w)\in\Gamma_{u}, z\neq x, w\neq x,(z,w)\not\in R_{u}\}
\]

\[
A_{-,u}:=\{x\in\T_{u}:\exists(x,z),(x,w)\in\Gamma_{u}^{-1}, z\neq x, w\neq x,(z,w)\not\in R_{u}\}
\]

Whenever branching occurs, we can construct two distinct timelike geodesics with endpoints in $\Gamma_u$ that are not related by the transport relation $R_u$, as described by the following lemma.
\begin{lemma} \label{branching geodesics}
	Consider a forward branching point $x\in A_{+,u}$, so that there exist  $z,w\in\T_{u}$ satisfying
	$(x,z),(x,w)\in\Gamma_{u} \cap X_{\ll}^2,$ but $(z,w)\not\in R_{u}.$ Then there
	exist two distinct non-constant timelike geodesics $\gamma^{1},\gamma^{2}\in TGeo^{\ell}$
	such that:

\begin{enumerate}
	\item $\gamma^{1},\gamma^{2}\in G$, where $G$ is given by \ref{Geodesics in transport order},
	\item $(x,\gamma_{s}^{1}),(x,\gamma_{s}^{2})\in\Gamma_{u}$ for all $s\in(0,1]$,
	\item $u(\gamma_{s}^{1})=u(\gamma_{s}^{2})$ for all $s\in[0,1].$
    \item $\g_s^1,\g_s^2\not \in R_u^{nb}$ for all $s\in [0,1]$.
\end{enumerate}
    The analogous statement for $A_{-,u}$ follows by causal reversal.
\end{lemma}

\begin{proof}
	Since $(x,z),(x,w)\in\Gamma_{u} \cap X_{\ll}$, consider the two timelike geodesics $\gamma_{1},\gamma_{2}$
	such that $(x,z)=(\gamma_{0}^{1},\gamma_{1}^{1})$, and $(x,w)=(\gamma_{0}^{2},\gamma_{1}^{2})$.
	From Lemma \ref{geodesics align}, it follows  that $\gamma^{1},\gamma^{2}\in G.$ Since
	$(z,w)\not\in R_{u},$ this implies that $z\neq w$.

	Note that both $z,w$ are different from $x.$ We also get that
	$x\not\in b_u$: indeed if we assume that $x\in b_u,$ then $\not\exists y\in X$
	with $(x,y)\in\Gamma_u$ and $\ell(x,y)>0$, but this gives a contradiction
	if we consider $y$ to be either $z$ or $w.$ This also implies that the geodesics $\gamma^{1},\gamma^{2}$
	are non-constant. 
	
	Since $z,w$ can be interchanged above, we can assume $u(z)\geq u(w),$
	and from $(x,w)\in\Gamma_{u} \cap X_{\ll}$ we get that $u(w)> u(x).$ By continuity of $u$ and $\gamma^1$ ,
	there exists $s^*\in(0,1]$ such that 
	\[
	u(w)=u(\gamma_{s^*}^{1}).
	\]
	
	Note that $w\neq\gamma_{s^*}^{1}$, as otherwise, $(w,z)\in\Gamma_{u}$
	because $z$ and $\gamma_{s^*}^{1}$ are along the same geodesic. Furthermore,
	$(w,\gamma_{s^*}^{1})\not\in R_{u}$ because if they were, then $0=u(\gamma_{s^*}^{1})-u(w)=\ell(w,\gamma_{s^*}^{1})$
	which implies either that $w\not\leq\gamma_{s^*}^{1}$ or $\gamma_{s^*}^{1}\not\leq w$,
	and since have established that $w\neq\gamma_{s^*}^{1}$, it contradicts
	the definition of $\Gamma_{u},\Gamma_{u}^{-1}$ as these are defined
	as pairs in $X_{\leq}^{2}.$
	
	By continuity of $u$,$\gamma^{1},\gamma^{2}$ and since we can affinely parameterize $\gamma^{1},\gamma^{2}$, there exists a
	$\delta>0$ such that 
	\[
	u(\gamma_{1-t}^{1})=u(\gamma_{s^*(1-t)}^{2}),\quad d(\gamma_{1-t}^{1},\gamma_{s^*(1-t)}^{2})>0
	\]
	for all $0\leq t\leq\delta$. 
	By  re-applying the argument above to every $0\leq t\leq\delta$ to
	deduce that $(\gamma_{1-t}^{1},\gamma_{s^*(1-t)}^{2})\not\in R_{u}.$
	The required geodesics are obtained by properly restricting and rescaling
	the geodesics $\gamma^{1},\gamma^{2}$. 
	
\end{proof}

Consider now the forward  branching point $x\in A_{+,u}$, so that there exists $w,z\in \T_u$ for which there are two distinct  $\gamma_1, \gamma_2$ branching timelike geodesics between $(x,w)$ and $(x,z)$ respectively. We will show that there is a measurable correspondence between the forward branching points $A_{+,u}$ and the corresponding geodesics. We will need the following selection result from Theorem 5.5.2 of  \cite{Srivastava1998}.

\begin{theorem}\label{measurable selection}
	Let $X,Y$ be Polish spaces, and let $F\subset X \times Y$ be analytic. Let $\mathcal{A}$ be the $\sigma$-algebra generated by analytic subsets of $X$. Then there is an $\mathcal{A}$-measurable section $S : P_1(F) \rightarrow Y$ of $F$, meaning that  $\textrm{graph} (S)\subset F$.
\end{theorem}

\begin{lemma}{\label{measurable correspondence branching}}
	Consider the set of forward branching points $A_{+,u}$. Then there exists an $\m$-measurable map $S:A_{+,u}\rightarrow G\times G$, where $G$ is given by \eqref{Geodesics in transport order}, such that if $S(x)=(\gamma^1,\gamma^2)$, then 
	\begin{enumerate}
		\item $\gamma^{1},\gamma^{2}\in G$
		\item $(x,\gamma_{s}^{1}),(x,\gamma_{s}^{2})\in\Gamma_{u}$ for all $s\in[0,1]$
		\item $(\gamma_{s}^{1},\gamma_{s}^{2})\not\in R_{u}$ for all $s\in[0,1]$
		\item $u(\gamma_{s}^{1})=u(\gamma_{s}^{2})$ for all $s\in[0,1].$
	\end{enumerate}
	Moreover, both geodesics are non-constant, and the set $A_{+,u}$ is $\sigma$-compact.
\end{lemma}

\begin{proof}

We first show that the set $G=\{ \gamma \in TGeo^{\ell}:  (\gamma_0, \gamma_1) \in \Gamma_u\},$ is $\sigma$-compact, and hence Borel subset of $TGeo^{\ell}$, with respect to the uniform distance $d_\infty$. We first show that $G$ is $\sigma$-compact. Recall that $\Gamma_u\setminus\Delta_Z = \cup_{n\geq 1}B_n$ where $$B_n:=\{ (x,y)\in Z\times Z:u(y)-u(x)=\ell^+(x,y)\geq \frac{1}{n}\}.$$ Since $Z$ is compact and $u,\ell^+$ are continuous, each $B_n$ is compact. Define $$G_n:=\{ \g \in TGeo^\ell:(\g_0,\g_1)\in B_n\}.$$ Since  timelike geodesics have non-diagonal endpoints, we have $$G = \cup_{n\geq 1} G_n.$$ We claim that each $G_n$ is compact in the uniform topology. Indeed, since $B_n\subset X_\ll^2$ is compact since every gh LLS is $\K$-globally hyperbolic \cite[Remark 12]{McCann2024}. In particular, all causal curves whose endpoints are in $B_n$ are contained in a common compact region and have uniformly bounded $d$-length. Together with the compactness theorem for timelike $\ell$-geodesics from Beran et al \cite{octet+}, this gives relative compactness of $G_n$ in the uniform topology.

We now show that $G_n$ is closed. Let $$\g^k\in G_n,\quad \g_k\rightarrow \g\quad \textrm{uniformly}.$$ Since $B_n$ is closed, we have $(\g_0,\g_1)\geq \frac{1}{n}>0$. For every $0\leq s<t\leq 1$ we have $$\ell(\g_s^k,\g_t^k)=(t-s)\ell(\g_0^k,\g_1^k)>0\rightarrow \ell(\g_s,\g_t)=(t-s)\ell(\g_0,\g_1)>0,$$ where we have passed to the limit using the continuity of $\ell^+$ on the timelike region. Hence $\g\in TGeo^\ell$, and therefore $\g\in G_n$. Hence $G_n$ is compact and therefore $G$ is $\sigma$-compact.

For every $\g,$ Lemma \ref{geodesics align} gives us $(\g_0,\g_s)\in \Gamma_u$ for all $s\in [0,1]$. Consequently, we have that $u(\g_s)-u(\g_0)=\ell(\g_0,\g_s)=s\ell(\g_0,\g_1)$, which rearranging gives \begin{equation}\label{affine u}u(\g_s)=(1-s)u(\g_0)+su(\g_1).\end{equation} For $\g^1,\g^2$, define the continuous map $h(\g^1,\g^2):=\min_{s\in [0,1]}d(\g_s^1,\g_s^2)$. Then,  define 
    $$\hat{F}:= \{(x,\g^1,\g^2)\in \T_u\times G\times G: (x,\g_0^1),(x,\g_0^2)\in\Gamma_u, u(\g_0^1)=u(\g_0^2), u(\g_1^1)=u(\g_1^2), h(\g^1,\g^2)>0\}.$$
This set is $\sigma$-compact, since $\T_u,\Gamma_u, G$ are $\sigma$-compact, and $\{ h>0\}=\cup_{m\geq 1}\{h\geq \frac{1}{m} \}.$ After choosing a compact exhaustion of $\T_u,\Gamma_u, G$, the defining conditions are closed on each compact exhaustion. Hence $\hat{F}$ is a countable union of compact sets.

By Lemma \ref{branching geodesics}, for every $x\in A_{+,u}$ there exists a pair $(\g^1,\g^2)\in G\times G$ with distinct final endpoints. By continuity, restricting both geodesics to a sufficiently short common interval near their final endpoints, and reparametrizing them to $[0,1]$, we obtain a pair of geodesics which we still denote by $(\g^1,\g^2)$ satisfying $h(\g^1,\g^2)=\min_{s\in[0,1]}d(\g_s^1,\g_s^2)>0,$ and such that $(x,\g^1,\g^2)\in \hat{F}.$ Thus, $A_{+,u}\subset P_1(\hat{F})$. Conversely, if we suppose that $(x,\g^1,\g^2)\in \hat F$, then setting $z:=\g_0^1,$ $w:=\g_0^2$, we have $(x,z),(x,w)\in \Gamma_u$. Moreover, $h(\g^1,\g^2)>0$ implies $z\neq w$, while $u(z)=u(w).$ Hence $(z,w)\not\in R_u$, and therefore $x\in A_{+,u}$. We conclude that $A_{+,u}=P_1(\hat{F}),$ and therefore $A_{+,u}$ is $\sigma$-compact.

We now apply the measurable selection theorem to $\hat{F}$, which we view as an analytic subset of the Polish space $Z\times TGeo^\ell\times TGeo^\ell$. This gives an $\m$-measurable section $$S:A_{+,u}\rightarrow TGeo^\ell\times TGeo^\ell,$$ and since every fiber of $\hat{F}$ lies in $G\times G$, we see that the image of $S$ lies in $G\times G$. Let $S(x) = (\g^1,\g^2).$ By construction, we have that $(x,\g_0^1),(x,\g_0^2)\in \Gamma_u$, and using Lemma \ref{geodesics align} we have  $$(\g_0^i,\g_s^i)\in \Gamma_u\quad\forall s\in[0,1],$$ and therefore $(x,\g_s^i)\in \Gamma_u$ for $i=1,2.$ By \eqref{affine u} and using $u(\g_0^1)=u(\g_0^2),$ and $u(\g_1^1)=u(\g_1^2)$ we get $$u(\g_s^1)=u(\g_s^2)\quad \forall s\in [0,1].$$
Finally, we see that $h(\g^1,\g^2)>0$ implies that $\g_s^1\neq \g_s^2$ for all $s\in [0,1]$, and since distinct points with equal $u$-value can not be related by $R_u$ we get $$(\g_s^1,\g_s^2)\not \in R_u\quad \forall s\in [0,1].$$ This proves the lemma. \end{proof}

\begin{remark}\label{causal reversal backwards branching synchronization}
    By causal reversal, the analogous statement holds for the backward branching set $A_{-,u}$  as well.
\end{remark}
The following lemma identifies a class of $\ell^p$-cyclically monotone sets for $p\in(0,1)$ within the $\ell$-cyclically monotone transport relation $\Gamma_u$. This will allow us to apply    $\ell^p$-optimal transport results to suitable subsets of the $L^1$ transport relation.
\begin{lemma} \label{p-cyclically monotone set}
	Let $\Delta\subset\Gamma_{u}$ be any set such that $(x_{0},y_{0}),(x_{1},y_{1})\in\Delta\implies((u(y_{1})-u(y_{0}))\cdot(u(x_{1})-u(x_{0}))\geq0.$
	Then $\Delta$ is $\ell^{p}$-cyclically monotone for all $p\in(0,1).$ 
\end{lemma}
\begin{proof}
	This follows from the works of Cavalletti-Mondino \cite[Prop. 4.13]{CavallettiMondino2024}.
\end{proof}

We will now show that the branching points $A_{+},A_{-}$ have $\m$-measure zero under the forward $p$-enb condition. This was first addressed in the metric case by Cavalletti \cite{Cavalletti2014}. In our setting, forward branching will be ruled out by assuming $\ell^p$-optimal couplings are concentrated on a graph of a map, which holds whenever the  metric spacetime is forward $p$-enb and satisfies $TMCP^+(K,N)$ by Theorem \ref{Brenier Map}.

\begin{theorem} \label{endpoints measure zero}
	Let $(X,d,\m,\ell)$ be a gh LLS, fix $p\in (0,1)$, and  assume that the metric spacetime is  forward $p$-enb and satisfies $TMCP^+(K,N)$. Then for any continuous reverse Lipschitz function
	$u:Z\rightarrow\R,$ we have 
	\[
	\m(A_{+,u})=0.
	\]
    
    \noindent Under analogous assumptions on the causally reversed spacetime, we get that $\m(A_{-,u})=0$.
\end{theorem}

\begin{proof}
	We focus on the forward branching set $A_{+,u}$, since the proof of $A_{-,u}$
	is identical under causal reversal of the spacetime and the hypotheses.
	
	\textbf{Step 1:} Suppose that $\m(A_{+,u})>0.$ By definition of $A_{+,u},$ and by Lemma \ref{measurable correspondence branching}, there exists an $\m$-measurable map $S:A_{+,u}\rightarrow G\times G$, $S(x)=(\g_x^1,\g_x^2)$ such that $$(x,\g_{x,s}^i)\in \Gamma_u,\quad u(\g_{x,s}^1)=u(\g_{x,s}^2),\quad (\g_{x,s}^1,\g_{x,s}^2)\not \in R_u\quad \forall s\in[0,1].$$
	By Lusin's theorem  applied to $\sigma$-compact $A_{+,u}$, we can select a subset of $A_{+,u}$ that is of positive $\m$-measure such that the 
	map $S$ is continuous and in particular, the functions 
	\[
	A_{+,u}\ni x\mapsto u(\gamma_{j}^{i})\in\R,\quad i=1,2,j=0,1
	\]
	are all continuous.
	
	Define $\alpha_{x}:=u(\gamma_{0}^{1}),$ and $\beta_{x}:=u(\gamma_{1}^{1})$
	and note that $\beta_{x}>\alpha_{x}$, where the inequality is strict as
	the geodesics are not constant. There exists a $B\subset A_{+,u}$
	with $\m(B)>0$ such that 
	\[
	\sup_{x\in B}\alpha_{x}<\inf_{x\in B}\beta_{x}.
	\]
	This set can be obtained by considering $A_{+,u}\cap B_{r}(x)$ for
	each $x\in A_{+,u},$ for a sufficiently small $r.$
	
	\textbf{Step 2}: Let $I=[c,d]$ be a non-trivial interval such that
	\[
	\sup_{x\in B}\alpha_{x}<c<d<\inf_{x\in B}\beta_{x}.
	\]
	By construction, we have that $I\subset\{u(\gamma_{s}^{i}):s\in[0,1]\},$
	for $i=1,2.$ 
	
	Fix any point inside $I,$ say $c.$ By Lemma \ref{branching geodesics}, for any $x\in B$
	there exists $s(x)$ such that 
	\[
	u(\gamma_{s(x)}^{1})=u(\gamma_{s(x)}^{2})=c.
	\]
	Since $\gamma^{1}\neq\gamma^{2},$ we can define on $B$ the two measurable transport
	maps $B\ni x\mapsto T_{i}(x):=\gamma_{s(x)}^{i},$ for $i=1,2.$ Then
	we define the transport plan 
	\[
	\pi:=\frac{1}{2}((Id,T_{1})_{\#}\m_{B}+(Id,T_{2})_{\#}\m_{B}),
	\]
	where $\m_{B}:=(\m(B))^{-1}\m\llcorner_{B}.$
Note that $(\mu_0,\mu_1) = ((P_1)_{\#}\pi, (P_2)_{\#} \pi )$ is a compact timelike $p$-dualisable pair.	

	\textbf{Step 3}: The support of $\pi$ is $\ell^{p}-$cyclically monotone
	for any $p\in(0,1).$ Indeed the measure $\pi$ is concentrated on
	the set 
	\[
	\Delta=\{(x,\gamma_{s(x)}^{1}):x\in B\}\cup\{(x,\gamma_{s(x)}^{2}):x\in B\}\subset\Gamma_{u}.
	\]
	
	If we take any two couples $(x_{0},y_{0}),(x_{1},y_{1})\in\Delta$,
	then by construction: 
	\[
	u(y_{1})-u(y_{0})=0.
	\]
	Therefore trivially $(u(y_{1})-u(y_{0}))\cdot(u(x_{1})-u(x_{0}))=0$
	and by Lemma \ref{p-cyclically monotone set}
	we get that $\Delta$ is $\ell^{p}$ cyclically monotone
	for any $p\in(0,1).$ We now have an $\ell^p$ optimal plan $\pi$ with $(P_{1})_{\#}\pi\ll\m$. However this plan can not be induced by a map, 
and this contradicts Theorem \ref{Brenier Map} where every $\ell^{p}$ optimal plan $\pi$ with an absolutely continuous initial measure
	is induced by a map.
\end{proof}

We remove branching points to obtain the \emph{non-branched transport set} $\T_u^{nb}$  and the \emph{non-branched transport relation} $R_u^{nb}$ :

\[
\T_{u}^{nb}:=\T_{u}\backslash(A_{+,u}\cup A_{-,u}),\quad R_{u}^{nb}:=R_{u}\cap\left(\T_{u}^{nb}\times \T_{u}^{nb}\right),
\]
and we note that $\T_u^{nb}$ still contains elements from endpoints $a_u,b_u$. This definition emphasizes the distinction between genuine branching $A_{+,u}, A_{-,u},$ where  geodesics match for a positive time interval and are otherwise distinct, and endpoint sets $a_u,b_u$ where it is possible that two distinct geodesics meet at either an initial point or a final point.  The latter observation on $a_u, b_u$ was known to Monge  \cite{Monge}.

After removing the branching structure, we obtain that the set $\T_u^{nb}$ is composed of equivalence classes defined by $R_u^{nb}.$ The following proposition is adapted from the metric result of Cavalletti \cite[Thm. 4.6]{Cavalletti2014}, as well as from the Lorentzian result of Cavalletti-Mondino \cite[Prop. 4.5]{CavallettiMondino2024}.
\begin{proposition}{\label{equivalence relation}}
	The non-branched transport relation $R_{u}^{nb}\subset X\times X$ is an equivalence
	relation on $\T_{u}^{nb}.$
\end{proposition}

\begin{proof}
	For all $x\in P_{1}(\Gamma_{u}),$ it follows that $(x,x)\in\Gamma_{u}$.
	Consider now $x,y\in\T_{u}^{nb}$ with $(x,y)\in R_{u}^{nb}.$ Then
	$(y,x)\in R_{u}^{nb}$ as well since either $(x,y)\in\Gamma$ or $(x,y)\in\Gamma^{-1}$. 
	
	We only have to show transitivity, that is, for every $(x,y),(y,z)\in R_{u}^{nb}$,
	it follows that $(x,z)\in R_{u}^{nb}$. We assume $x\neq y\neq z$,
	as otherwise the claim is trivial.
	
	Case 1: $(x,y),(y,z)\in\Gamma_{u}$. 
	
	We use the reverse triangle inequality
	\[
	u(z)-u(x)\geq\ell(x,z)\geq\ell(x,y)+\ell(y,z)=u(y)-u(x)+u(z)-u(y)=u(z)-u(x),
	\]
	which means that $u(z)-u(x)=\ell(x,z)$, and so $(x,z)\in\Gamma_{u}.$
	
	Case 2: $(x,y),(y,z)\in\Gamma_{u}^{-1}.$ Then $(z,y),(y,x)\in\Gamma_{u}$
	and so $(z,x)\in\Gamma_{u}$ which gives us $(x,z)\in\Gamma_{u}^{-1}.$
	
	Case 3: $(x,y)\in\Gamma_{u}$ and $(y,z)\in\Gamma_{u}^{-1}$.  Then by definition we have $(x,y),(z,y)\in \Gamma_u$. Since $y\not \in A_{-,u}$, it follows that $(x,z)\in R_u$. Since $x,z\in \T_u^{nb}$, we conclude that $(x,z)\in R_u$.

	Case 4: $(x,y)\in\Gu^{-1}$ and $(y,z)\in\Gu$. Then $(y,x),(y,z)\in \Gu$, and since $y\not \in A_{+,u}$, it follows that $(x,z)\in R_u$. Since $x,z\in \T_u^{nb}$, we again conclude that $(x,z)\in R_u$. 
\end{proof}
We will now show that each equivalence class of $R_u^{nb}$ is formed by a single timelike geodesic.

\begin{lemma}\label{same geodesic}
Fix any $x\in\mathcal{T}_{u}^{nb}$.
Then for any choice of $z,w\in R_{u}^{nb}(x)$, there exists $\gamma\in G$, where $G$ is defined in \eqref{Geodesics in transport order},
such that 
\[
\{x,z,w\}\subset\{\gamma_{s}:s\in[0,1]\}.
\]

If $\hat{\gamma}\in G$ enjoys the same property, then we have that $$( \{\hat{\gamma}_s: s\in [0,1]\} \cup \{\gamma_s: s\in [0,1]\}) \subset \{\bar{\gamma}_s: s\in [0,1]\}$$  for some $\bar{\gamma} \in G\subset TGeo^{\ell}$.
\end{lemma} 

\begin{proof}
This proof is similar to that of Proposition \ref{equivalence relation}. Assume that $x,z,w$ are all distinct points as otherwise the proof follows from Lemma \ref{geodesics align}. There are four cases to consider.

Case 1: $z\in \Gamma_u (x), w\in \Gamma_u^{-1}(x).$ By $\ell$-cyclical monotonicity we get:
$$\ell(w,z)\geq \ell(w,x)+\ell(x,z)=u(x)-u(w)+u(z)-u(x)\geq \ell(w,z),$$
which implies that $x,z,w$ all lie on a geodesic by  Lemma \ref{geodesics align}.

Case 2: $z,w\in \Gamma_u(x)$. Without loss of generality we can assume that $u(x)\leq u(w)\leq u(z)$. In fact these inequalities are strict due to the definition of $\T_{u}^{nb}=\T_{u}^e\backslash(A_{+,u}\cup A_{-,u})$, as seen by the proof of Lemma \ref{measurable correspondence branching}. Thus, we can assume that 
$$u(x)< u(w)< u(z).$$

Assume that there does not exist a geodesic $\gamma \in G$ such that $\gamma_0=x$, $\gamma_1 = z$ and $\gamma_s = w$ with $s\in (0,1).$ Consider $\eta\in G$ such that $(\eta_0,\eta_1)=(x,z)$, and for some $s\in (0,1)$ we have 
$$ u(\eta_s)=u(w), \quad \eta_s \in \Gamma_u(x), \quad \eta_s \neq w. $$ 

\noindent Since $u(\eta_s)=u(w)$ and $\eta_s\neq w$, we have $(\eta_s,w)\not\in R_u$. On the other hand,  both $(x,\eta_s)$ and $(x,w)$ are in $\Gu$ by Lemma \ref{geodesics align}, and consequently $x\in A_{+,u}$which is a contradiction.

The remaining two cases  are $z,w\in \Gamma_u^{-1}(x)$, and $z\in \Gamma_u^{-1}(x)$ with $w\in\Gamma_u(x)$, both follow since they are the reversed and symmetric cases considered above. The last part of the statement follows from the construction above.
\end{proof}

To summarize this subsection,  given a  transport interpolation set $Z\subset X$ and  a continuous reverse Lipschitz $u:Z\rightarrow \R$ we have constructed the non-branched transport set $\T_u^{nb}$ with the equivalence relation $R_u^{nb}$ on it, whose equivalence classes are timelike geodesics. In the next subsection, we will produce a disintegration of the measure $\m$ with respect to this equivalence relation.

\subsection{Regularity of the disintegration.}\label{S: regularity of disintegration}

Our goal is to obtain a \emph{strongly consistent disintegration} of the background measure $\m$ restricted to the \emph{non-branched transport set} $\T_u^{nb}$, where each equivalence class will be given by the \emph{non-branched transport relation} $R_u^{nb}$ from the previous section:

\[
\T_{u}^{nb}:=\T_{u}\backslash(A_{+,u}\cup A_{-,u}),\quad R_{u}^{nb}:=R_{u}\cap\left(\T_{u}^{nb}\times \T_{u}^{nb}\right).
\]

We follow the disintegration of measure paradigm as introduced by Bianchini-Cavalletti \cite{BianchiniCavalletti2013}. In order to invoke the disintegration theorem, we first denote each equivalence class by $[x]$ and index it by $\T_{u}^{nb}$, i.e.:

\[
\{[x]\}_{x\in\T_{u}^{nb}}:=\{y\in\T_{u}^{nb}:(x,y)\in R_{u}^{nb}\}_{x\in\T_{u}^{nb}}.
\]

\noindent and we need to construct a quotient
map $f$ associated to the equivalence relation $R_{u}^{nb}$:

$$f: \T_u^{nb} \rightarrow \T_u^{nb},\quad f(x)\in [x],\quad f(x)=f(y)\,\,\textrm{if } (x,y)\in R_u^{nb}.$$

Recall that a \emph{cross-section of an equivalence relation}
$E\subset X\times X$ is a set $D\subset X$ such that the intersection of $D$ with
each equivalence class is a singleton. A  \emph{section of an equivalence relation} is a map
$f:X\rightarrow X$ such that for any $x,y\in X,$ we have 
\[
(x,f(x))\in E, \: \mathrm{and} \: (x,y)\in E\implies f(x)=f(y).
\] Note that each section $f$ canonically produces a cross-section
$D=\{x\in X:x=f(x)\}.$

The following proposition was shown by Bianchini-Cavalletti  \cite[Section 4]{BianchiniCavalletti2013}, also see Braun-McCann \cite[Proposition 6.18]{BraunMcCann+}:
\begin{proposition}\label{quotient map}
There exists an $\m$-measurable quotient map $f:\T_{u}^{nb}\rightarrow\T_{u}^{nb}$
for the equivalence relation $R_{u}^{nb}.$
\end{proposition}

The image of the non branched transport set $D:=f(\T_u^{nb})$ satisfies $D=\{x\in \T_u^{nb}:d(x,f(x))=0\},$ and is thus $\m$-measurable. Consider now the \emph{quotient measure} $$ \q:= f_{\#} (\m\restrict_{\T_u^{nb}}),$$

\noindent which is a Borel measure concentrated on $D$. By Lusin's theorem and inner regularity of $\m$ there exists a $\sigma$-compact set $Q\subset D$ such that $\q(X\setminus Q)=0$. We can now apply the disintegration theorem to $\m$ to obtain a strongly consistent disintegration when restricted to $f^{-1}(Q)$:

$$ \m \restrict _{f^{-1}(Q)}=\int_Q \m_{\alpha}\q (d\alpha), \quad \m_{\alpha}(f^{-1}(\alpha))=\norm{\m_{\alpha}} \quad \q-\textrm{a.e.}\, \alpha \in Q,$$ and since $\q(X\setminus Q)=0$, this implies that $$ \m(\T_u^{nb}\setminus f^{-1}(Q))=0,$$ and our disintegration formula reduces to  \begin{equation}\label{disintegration}
 \m \restrict _{\T_u^{nb}}=\int_Q \m_{\alpha}\q (d\alpha), \quad \m_{\alpha}(f^{-1}(\alpha))=\norm{\m_{\alpha}} \quad \q-a.e.\, \alpha \in Q.
 \end{equation}

To summarize, we have constructed a quotient set $Q$ which we will identify with a subset of $\T_{u}^{nb}$ defined through the image of $f$. Moreover, each equivalence class is  formed by a timelike geodesic by Lemma \ref{same geodesic}.

For the ease of identification of the equivalence classes, we introduce a causal order preserving map $g$, that we call the \emph{ray map}, whose inverse will  embed the non-branched transport set $\T_u^{nb}$  in $Q\times \R.$ 

\begin{definition}\label{ray map definition}
	Define the \emph{ray map} $g:Dom(g)\subset Q\times \R \rightarrow \T_u^{nb}$ through the formula

	$$\begin{aligned}  \textrm{graph}(g) :=& \{(\alpha,t,x): \alpha\in Q, t\in [0,\infty), (\alpha,x)\in \Gamma_u, x\in \T_u^{nb},\ell(\alpha,x)=t\}\} \\
	& \cup  \{(\alpha,t,x): \alpha\in Q, t\in (-\infty,0], (x,\alpha)\in \Gamma_u, x\in \T_u^{nb},\ell(x,\alpha)=-t\}\}
 \end{aligned}$$
\end{definition}

The ray map $g$ associates to each $y \in Q$ and $t$ in its domain the unique element in $[y]$ at timelike distance $|t|$ from $y$ to its future if $t>0$, and to its past if $t<0$, and equal to $y$ if $t=0.$ By Proposition \ref{equivalence relation}, and Lemma \ref{same geodesic}, we see that the ray map is well defined and we can take $(\alpha,t)\in Dom(g)$ if and only if there exists $x\in \T_u^{nb}$ satisfying the conditions in Definition \ref{ray map definition}.

We summarize 
a few crucial properties of the ray map in the following Proposition.

\begin{proposition}\label{ray map properties}
	The ray map enjoys the following properties:
	\begin{enumerate}
		\item $graph(g)$ is an analytic subset of $Q\times \R\times X$, and the ray map $g$ is Borel map on its Borel domain; 
		\item The range of $g$, denoted by $\textrm{Ran}(g)$ satisfies $\textrm{Ran}(g)=f^{-1}(Q)$. Moreover,  $\m(\T_u^{nb}\setminus \textrm{Ran}(g))=0$;
		\item $t\mapsto g(\alpha,t)$ is an $\ell$-isometry, i.e. if $s,t\in Dom(g(\alpha,\cdot))$
with $s\leq t$ then $\ell(g(\alpha,s),g(\alpha,t))=t-s$;

		\item $Dom(g) \ni (y,t)\mapsto g(y,t)$ is bijective and its inverse for every $x\in f^{-1}(Q)$ is given by 
		$$  x\mapsto g^{-1}(x)=\begin{cases}
		    &(f(x), -\ell(x, f(x))) \quad \textrm{if } x\leq f(x) \\
            &
            (f(x),\ell(f(x), x) ) \quad \textrm{if } f(x)\leq x.
		\end{cases}$$
	\end{enumerate}
\end{proposition}

Using the quotient map $f$ and the ray map $g$, we identify each equivalence class from the disintegration of $\m$ in \eqref{disintegration}, for $\q$-a.e. $\alpha\in Q$ with $$(X_\alpha,\ell\restrict_{X_\alpha},\m_\alpha),$$ where $X_\alpha :=f^{-1}(\alpha), \m_\alpha$ is the conditional measure appearing in the disintegration formula \eqref{disintegration}, and the restriction $\ell\restrict_{X_\alpha}$ uses that $g$ is an $\ell$-isometry. We can further use the ray map to identify each one dimensional $(X_\alpha,\ell\restrict_{X_\alpha},\m_\alpha)$ with a one dimensional representative on $\R.$ In both representations, we will refer to each equivalence class $(X_\alpha,\ell\restrict_{X_\alpha},\m_\alpha),$ as a \emph{transport ray}. In the next subsection, we show that the $TMCP^+(K,N)$ condition on the measured gh LLS $(X,d,\m, \ell)$ can be localized to $\q$-a.e. transport ray $(X_\alpha,\ell\restrict_{X_\alpha},\m_\alpha)$.

    \subsection{Localization of \texorpdfstring{$TMCP^+(K,N)$}{TMCP+(K,N)} onto transport rays }

In this subsection, we will show that the $TMCP^+(K,N)$ condition on $(X,d,\m,\ell)$ can be localized onto each transport ray.
The $TMCP^+(K,N)$ condition will be reflected in the forward evolution of $\ell^p$-cyclically monotone subsets inside of $\T_u^{},$ which we will construct below. 

Starting with the transport set $\T_u$, consider an analytic subset  
$A\subset \T_u$ with $\m(A)>0,$ and note that Theorem \ref{endpoints measure zero} gives us that $\m(A_{+,u})=0.$ Define the analytic set $\Lambda_s$:  $$\Lambda_s = \{(x,y)\in (A\times(\T_u\setminus A_{+,u}))\cap \Gamma_u:\ell(x,y) = s \}.$$ By Lemma \ref{p-cyclically monotone set}, this set is $\ell^p$-cyclically monotone. Since $A$ is measurable, the \emph{forward evolution} $A_t$ of $A$, defined by \begin{equation}\label{Evolution}A_t:=P_2(\Lambda_t)\end{equation} is also analytic, and we will evaluate the Lebesgue measure of times corresponding to this evolution, i.e. $|\{ t\in [0,\infty):\m(A_t)>0\}|$.

The following Proposition is an adaptation of Proposition 4.14 from \cite{CavallettiMondino2024}, see also  \cite[Lemma 6.26]{BraunMcCann+}.

\begin{proposition}\label{evolution of subsets} Let $(X,d,\m,\ell)$ be a forward $p$-enb gh LLS that satisfies the $TMCP^+(K,N)$ condition, and let $u:Z\rightarrow \R$ be a continuous reverse-Lipschitz function. For any analytic set $A\subset \T_u\setminus b_u$ with $\m(A)>0$, there exists $s>0$ and a compact subset $B\subset A$ such that its forward evolution $B_t$ defined by \eqref{Evolution} satisfies\begin{equation}
    \bigcup_{t\in[0,s]}B_t\Subset X,\quad B_t\subset \T_u\setminus b_u,\quad and\quad \m(B_t)>0\quad \forall t\in[0,s).
\end{equation}

\noindent In particular, $|\{t\in [0,\infty):\m(A_t)>0 \}|>0.$
\end{proposition}

\begin{proof}
    Let $A\subset \T_u\setminus (b_u\cup A_{+,u})$ with $\m(A)>0,$ whereby removing the set $A_{+,u}$ is without loss of generality since $\m(A_{+,u})=0.$ Let $s\in[0,\infty)$ and consider the following $\ell^p$-cyclically monotone set $$\Lambda_s = \{(x,y)\in (A\times(\T_u\setminus A_{+,u}))\cap \Gamma_u:\ell(x,y) = s \}.$$ We claim that the projection of this set is monotone: $$0\leq s_1\leq s_2\implies P_1(\Lambda_{s_2})\subset P_1(\Lambda_{s_1})\subset A.$$ Indeed, for $x\in P_1(\Lambda_{s_2})$ there exists a $y$ with $\ell(x,y)=s_2$ and $(x,y)\in \Gamma_u$. Since $s_1\leq s_2,$ by Lemma \ref{geodesics align} we can take a point $y'$ on the same geodesic that is of distance $s_1$ from $x$, i.e. $\ell(x,y')=s_1.$ If $s_1<s_2$, then $y'\ll y$, while if $s_1=s_2$, we can take $y'=y$. By Lemma \ref{geodesics align} and the interpolation property of $Z$, wee have $(x,y')\in \Gu$ and $y'\in \T_u$. Moreover, $y'\not\in A_{+,u}$ as any forward branching at $y'$ would imply forward branching at $x$, contradicting that $x\not\in A_{+,u}$. Thus, $(x,y')\in \Lambda_{s_1}$ which implies $x\in P_1(\Lambda_{s_1}).$ 
Moreover, since $A\subset \T_u\setminus (b_u\cup A_{+,u}),$ this means that every $x\in A$ is not a final point, so there exists some  point $y$ with $\ell(x,y)>0,$ that is: there exists an $r>0$ such that $x\in P_1(\Lambda_r)$. It follows that $$\bigcup_{r>0} P_1(\Lambda_r)=A,$$ and by monotone convergence we can conclude that $$\lim_{s\downarrow 0}\m(P_1(\Lambda_s)) = \m(A)>0.$$ Take $r>0$ small enough such that $\m(P_1(\Lambda_r))>0$, and fix the set $B=P_1(\Lambda_r).$ By inner regularity of $\m$, there exists a compact subset of $B$ of positive measure, that we still denote by $B$, and by the measurable selection theorem, an $\m$-measurable map $T:B\rightarrow \T_u$ such that $(x,T(x))\in \Lambda_r$ for all $x\in B.$  Lusin's theorem allows us to find a compact subset of $B$ of positive measure, that we still denote by $B$, on which $T$ is continuous and thus $T(B)$ is compact.

Define the $\mu_0$ and $\mu_1$ as $$\mu_0:=\frac{1}{\m(B)}\m\restrict_B,\quad \mu_1:=T_\#\mu_0.$$ By construction, the coupling $\pi_T:=(Id,T)_\#\mu_0$ is $\ell^p$-cyclically monotone since $\spt \pi_T\subset \Lambda_r$ and so $\ell(x,T(x))=r$ for $\mu_0$-a.e. $x$, and additionally $$\int \ell(x,y)^p d\pi_T = r^p\in (0,\infty).$$

Since $\pi_T$ is $\ell^p$-cyclically monotone, and $\pi_T(\{ \ell>0\}=1,$ arguing as in Proposition 2.8 in \cite{CavallettiMondino2024} we conclude that $\pi_T$ is an $\ell^p$-optimal coupling, and Theorem \ref{Brenier Map} produces a unique $\ell^p$-geodesic $(\mu_t)_{t\in[0,1]}$ between $\mu_0$ and $\mu_1$ that satisfies $\mu_t\ll\m$ for all $t\in[0,1).$ Additionally, global hyperbolicity gives us that  $\spt \mu_t $ is compactly contained in $J^+(\spt \mu_0)\cap J^-(\spt \mu_1)$.

By construction, the map $T$ moves mass for a constant $\ell$-distance $r$ along the rays, and the geodesic $\mu_t$ is concentrated on $B_{tr}$, where $B_{tr}$ is the forward evolution of $B$ at fraction $t$ along the geodesics from $x$ to $T(x).$ Since $\mu_t\ll\m$, and $\mu_t$ is a probability measure with $\mu_t(B_{tr}) = 1,$ it follows that $\m(B_{tr})>0$ for all $t\in[0,1).$
Thus, setting $s=tr$, we have that $\m(B_s)>0$ for all $s\in[0,r)$. By global hyperbolicity, we get that $\cup_{s\in[0,r]}B_s$ is contained in a compact set, and $B_s \subset \T_u\setminus (b_u\cup A_{+,u}).$ Finally, since $B\subset A$, it follows that $B_s\subset A_s$, and hence $\m(A_s)\geq \m(B_s)>0$ for all $s\in [0,r),$ which proves that $|\{t\in[0,\infty):\m(A_t)>0] \}|>0.$
\end{proof}

We  now show that the forward evolution along the transport set preserves the $TMCP^+(K,N)$ condition. 
The following proof is based on  Proposition 6.4 from \cite{Cavalletti2014}

\begin{proposition}\label{TMCP along sets}
    Under the same assumptions as Proposition \ref{evolution of subsets}, consider a compact set $B\subset \T_u\setminus b_u$. Let $r>0$ and $(B_t)_{t\in [0,1)}$ be the fixed length forward evolution given by Proposition \ref{evolution of subsets}. Then for every fixed $t\in [0,1),$ we have $$\m(B_t)\geq \big(\tau_{K,N}^{(1-t)}(r)\big)^N\m(B).$$ Moreover, $B_t\subset \T_u\setminus b_u$   for every $t\in[0,1)$, and $\bigcup_{0\leq t<1}B_t$ is relatively compact. 
\end{proposition}
\begin{proof}
Since $\m(A_{+,u})=0,$ by inner regularity it is enough to prove the claim for a compact subset $C\subset B\setminus A_{+,u}$ and then let $\m(C)\uparrow \m(B).$ We therefore assume throughout the proof that $B$ is compact and $B\subset \T_u\setminus(A_{+,u}\cup b_u).$ Proposition \ref{evolution of subsets} then gives the fixed length relation $\Lambda_r^B$, and since $B\cap A_{+,u}=\emptyset$ we get that forward endpoint at distance $r$ is unique, and hence $\Lambda_r^B$ is the graph of a measurable map $T$.

Consider therefore the set $\Lambda_r^B$ from Proposition \ref{evolution of subsets} $$\Lambda_r^B:=\{(x,y)\in \Gamma_u:x\in B,\ell(x,y)=r\}.$$  Since $B\subset \T_u\setminus (A_{+,u}\cup b_u),$ and $r$ is the fixed length selected by Proposition \ref{evolution of subsets},  we have that for every $x\in B$ there exists a unique point $y$ such that $(x,y)\in \Gu$ and $\ell(x,y)=r$. We thus have that up to a null set we can write this relation as the graph of a measurable map $T$: $$(x,T(x))\in \Gamma_u,\quad \ell(x,T(x))=r.$$  

Let $A=T(B)$ and choose a dense sequence $\{y_i\}_{i\in\N}$ in the compact closure of $A$. The key point is that this relation $\ell^p$-cyclically monotone set associated to a Kantorovich potential, and we will approximate the fixed length relation by finite-target transports. First, define $\lambda_r:=pr^{p-1}$ for $p\in (0,1)$ whereby strict concavity gives us the inequality $s^p\leq \lambda_r s+(1-p)r^p.$ Additionally, using the reverse Lipschitz bound on $u$, namely $u(y)-u(x)\geq \ell(x,y)$  gives us the combined inequality  \begin{equation}\label{combined inequality}\ell(x,y)^p\leq \lambda_r(u(y)-u(x))+(1-p)r^p,\end{equation} for all causal pairs $(x,y)\in Z\times Z$, where we note that this is $\ell^p$-Kantorovich duality for the $p$-Kantorovich potential $\varphi_r(x):=-pr^{p-1}u(x)+(1-p)r^p$ associated with the fixed length $r$ transport together with its dual function  $\psi_r(y):=pr^{p-1}u(y)$, reducing the inequality to $\ell(x,y)^p\leq \varphi_r(x)+\psi_r(y).$ Equality  holds precisely when $(x,y)\in \Gamma_u$ and $\ell(x,y)=r$, and therefore $\ell^p$-cyclically monotone set of this $p$-Kantorovich dual pair is exactly $\{(x,y)\in \Gamma_u:\ell(x,y)=r\}.$ 

We now approximate this potential by finitely many target points. For $I\in \N,$ define $$F_I(x):=\max_{1\leq j\leq I}\ell(x,y_j)^p-\lambda_ru(y_j),$$ and for each $i\in 1,\dots I$ define $$E_{i,I}=\{ x\in B:\ell(x,y_i)^p-\lambda_ru(y_i)=F_I(x)\}.$$ Define $\Lambda_I:=\bigcup_{i=1}^IE_{i,I}\times \{y_i\},$ which is $\ell^p$-cyclically monotone by construction. Taking the limit $I\rightarrow \infty$ on $B$ we see $$F_I\rightarrow \varphi_r(x)\quad \textrm{uniformly on $B$},$$ which follows since $F_I(x)\leq \varphi_r(x),$ and for the points that satisfy $(x,T(x))\in \Gamma_u, \ell(x,T(x))=r$ we get  equality in  the combined inequality \eqref{combined inequality}  and therefore we have $\ell(x,T(x))^p-\lambda_ru(T(x)) = \varphi_r(x)$. Since the points $y_i$ are dense in the compact set $Y=T(B),$ we get the uniform convergence of the finite maxima $F_I\rightarrow \varphi_r$ on $B$. In particular, we have \begin{equation}\label{square brackets varphi}
\sup_{(x,y)\in \Lambda_I} \varphi_r(x)-(\ell(x,y)^p-\lambda_ru(y))=\sup_{(x,y)\in \Lambda_I}\lambda_r[u(y)-u(x)-\ell(x,y)]+[\lambda_r\ell(x,y)+(1-p)r^p-\ell(x,y)^p]\rightarrow 0.\end{equation}  The terms in the second equality in the square brackets are non-negative since $u$ is reverse Lipschitz, and $s^p$ is strictly concave, and hence we conclude that \begin{equation}\label{Lambda sup bound}\sup_{(x,y)\in \Lambda_I}|\ell(x,y)-r|\rightarrow 0.\end{equation} Furthermore, considering the pair such that $(x_I,y_I)\in \Lambda_I,$ and $(x_I,y_I)\rightarrow (x,y),$ we get that the non-negative  terms in \eqref{square brackets varphi} vanish in the limit, so we get $u(y)-u(x)=\ell(x,y)=r$, and thus we recover $(x,y)\in \Gamma_u,$ and $\ell(x,y)=r$. Since $x\not \in A_{+,u}$ for a.e. $x\in B,$ the forward point at distance $r$ is thus unique and we get $y=T(x)$ for a.e. $x\in B$.

We now apply the $TMCP^+(K,N)$ condition to the finite approximation. The sets $E_{i,I}$ form a finite partition of $B$ but may overlap. Replace them successively by $E_{i,I}\setminus (E_{1,I}\cup\dots \cup E_{i-1,I})$, and retaining the same notation  we obtain a disjoint measurable partition of $B$.  Set $T_I(x):=y_i$ for $x\in E_{i,I},$ for which the graph of $T_I$ is contained in $\Lambda_I,$ and for every $i$ with $\m(E_{i,I})>0$ define $$\mu_0^{i,I} = \frac{1}{\m(E_{i,I})}\m\restrict_{E_{i,I}},\quad \mu_1^{i,I}:=\delta_{y_i}.$$
By \eqref{Lambda sup bound}, for sufficiently large $I,$ all pairs we have $(\spt \mu_0^{i,I},\spt \mu_1^{i,I})$ are compact and uniformly strictly timelike. Since these fixed-target pairs are timelike $p$-dualisable,  we may apply the $TMCP^+(K,N)$ condition to every such pair to obtain the $\ell^p$-dynamical plan $\eta_{i,I}$  that satisfies the $TMCP^+(K,N)$ entropy inequality, whereby by  Theorem \ref{Brenier Map} this dynamical plan is unique and moreover   $\mu_t^{i,I}:=(e_t)_\#\eta_{i,I}\ll \m$ for all $t\in [0,1).$ Writing the density $\mu_t^{i,I}=\rho_{i,I,t}\m$, and $E_{i,I,t}:=\{ \rho_{i,I,t}>0\}$ the $TMCP^+(K,N)$ inequality gives us $$S_N(\mu_t^{i,I})\leq -\frac{1}{\m(E_{i,I})^{1-1/N}}\int_{E_{i,I}}\tau_{K,N}^{(1-t)}(\ell(x,y_i))d\m(x)\leq -\m(E_{i,I})^{1/N}\inf_{x\in E_{i,I}}\tau_{K,N}^{(1-t)}(\ell(x,y_i)),$$ where on the other hand, Holder's inequality gives us $$-S_N(\mu_t^{i,I})=\int_X\rho^{1-1/N}_{i,I,t}d\m \leq \m(E_{i,I,t})^{1/N}.$$ Combining the above two inequalities gives us \begin{equation}\label{summed cells}
\m(E_{i,I,t})\geq \inf_{x\in E_{i,I}}[\tau_{K,N}^{(1-t)}(\ell(x,y_i)]^N\m(E_{i,I}),
\end{equation}

We now sum these estimates. Define $\mu_0:=\frac{1}{\m(B)}\m\restrict_B,$ $\mu_1:=(T_I)_\#\mu_0,$ and let $\pi = (Id,T_I)_\#\mu_0$ be their $\ell^p$-optimal coupling (since the graph of $T_I$ lies in an $\ell^p$-cyclically monotone set $\Lambda_I$.) Setting $$\alpha_{i,I}=\frac{\m(E_{i,I})}{\m(B)},\quad \eta_I=\sum_{i=1}^I\alpha_{i,I}\eta_{i,I},$$ we see by Theorem \ref{Brenier Map} that $\eta_I$ is unique and $\mu_{I,t}:=(e_t)_\#\eta_I\ll \m$ for all $t\in [0,1).$ Reparametrizing and restricting $\eta_I$ to a subinterval from $t$ to 1 as in \cite{CavallettiMondino2024}, this still yields an $\ell^p$-optimal dynamical plan from $\mu_{I,t}$ to $\mu_{1,I}$. Since $\mu_{I,t}\ll \m$ and the restricted endpoint pair is compactly supported and timelike $p$-dualisable, we may apply Theorem \ref{Brenier Map} once more to conclude that the endpoint coupling is induced by a map. Therefore, an intermediate point can not continue to two different endpoints $y_i$ and $y_j$, and thus we have that $$\m(E_{i,I,t}\cap E_{j,I,t})=0\quad \textrm{for $i\neq j$}.$$ Therefore, we may sum \eqref{summed cells} to obtain $$\m(\cup_{i=1}^IE_{i,I,t})=\sum_{i=1}^I\m(E_{i,I,t})\geq \inf_{x\in B}[\tau_{K,N}^{(1-t)}(\ell(x,T_I(x))]^N\sum_{i=1}^I \m(E_{i,I})=\inf_{x\in B}[\tau_{K,N}^{(1-t)}(\ell(x,T_I(x))]^N\m(B).$$ Since $\ell(x,T_I(x))\rightarrow r$ uniformly on $B$ by \eqref{Lambda sup bound}, we get that $$\inf_{x\in B}[\tau_{K,N}^{(1-t)}(\ell(x,T_I(x))]^N\rightarrow [\tau_{K,N}^{(1-t)}(r)]^N.$$

We now show that the approximation recovers $B_t$ as we send $I\rightarrow \infty$. Let $B_{I,t}$ denote the set of $t$-intermediate points along maximizing timelike geodesics corresponding to pairs in $\Lambda_I$, i.e. $z\in B_{I,t}$ if there exists $(x,y)\in \Lambda_I$ such that $\ell(x,z)=t\ell(x,y)$ and $\ell(z,y)=(1-t)\ell(x,y).$ The dynamical plans $\eta_{i,I}$ are concentrated on such maximizing timelike geodesics, and therefore, we have up to an $\m$-null set $\cup_{i=1}^IE_{i,I,t}\subset B_{I,t}$ and consequentially $\m(B_{I,t})\geq C_{I,t}\m(B).$ By global hyperbolicity and Proposition \ref{evolution of subsets} the sets $B_{I,t}$ are contained in a common compact causal diamond, and after taking their closure and passing to a subsequence, Blaschke's Theorem \cite[Theorem 7.3.8]{MR1835418} gives us $$\overline B_{I_n,t}\rightarrow \Theta_t$$ in Hausdorff distance for some compact set $\Theta_t.$ We claim that $\Theta_t\subset B_t.$ Indeed, let $z_n\in B_{I_n,t}$ with $z_n\rightarrow z$ and choose $(x_n,y_n)\in \Lambda_{I_n}$ such that $$\ell(x_n,z_n)=t\ell(x_n,y_n),\quad \ell(z_n,y_n)=(1-t)\ell(x_n,y_n).$$ Passing to a subsequence we get $x_n\rightarrow x \in B$ and $y_n\rightarrow y$, and by convergence of the finite contact sets we get $$y=T(x),\quad (x,T(x))\in \Gamma_u,\quad \ell(x,T(x))=r,$$ so we get $$\ell(x_n,z_n)\rightarrow tr,\quad \ell(z_n,y_n)\rightarrow (1-t)r.$$ This implies that $u(T(x))-u(x)=r$, while the reverse Lipschitz conditions give $u(T(x))-u(z)\geq \ell(x,T(x))=(1-t)r$ and  $u(z)-u(x)\geq \ell(x,z)=tr$. Adding these gives $$r=u(T(x))-u(x)\geq tr +(1-t)r=r,$$ and hence quality holds throughout and we get that $(x,z),(z,T(x))\in \Gamma_u$ and since $\ell(x,z)=tr$ we have that $z\in B_t.$  Hence $\Theta_t\subset B_t$. Since $B_{I_n,t}\rightarrow \Theta_t$ in Hausdorff distance, upper semicontinuity of $\m$ on compact sets gives us $$\limsup_{n\rightarrow \infty}\m(B_{I_n,t})\leq \m(\Theta_t)\leq \m(B_t),$$ which lets us conclude that $$\m(B_t)\geq [\tau_{K,N}^{(1-t)}(r)]^N\m(B).$$

Finally, if $z\in B_t$ and $t<1,$ then $z$ has a nontrivial forward continuation to $T(x)$ with proper time $(1-t)r>0$. Hence $z\not \in b_u.$ Since also $z\in \T_u$, we therefore have $$B_t\subset \T_u\setminus b_u\quad \textrm{for every $t\in[0,1)$},$$ and moreover all of the slices $B_t$ lie in the compact set given by Proposition \ref{evolution of subsets}, and hence $\bigcup_{0\leq t<1}B_t$ is relatively compact. 
\end{proof}

\begin{remark}\label{tau sigma TMCP evolution}
    Equivalently, the bound of Proposition \ref{TMCP along sets} can be stated in terms of the distortion coefficients as $$\m(A_t)\geq\tau_{K,N}^{(1-t)}(r)^N\m(A)=(1-t)\sigma_{K,N-1}^{(1-t)}(r)^{N-1}\m(A),$$ where $r = \ell(x,T(x))$ for every $x\in A$ is provided by Proposition \ref{evolution of subsets} for which this infimum is constant. 
\end{remark}

We now combine the previous two propositions to show that forward $p$-enb spacetimes that satisfy the $TMCP^+(K,N)$ condition must have that their initial points $a_u$ are measure zero. Crucially, the proof of $\m(a_u)=0$ only requires that $\m(A_{+,u})=0$ which follows from only the forward $p$-enb and the $TMCP^+(K,N)$ assumptions.

\begin{proposition}\label{m(a)=0} Let $(X,d,\m,\ell)$ be a forward $p$-enb gh LLS satisfying the $TMCP^+(K,N)$ condition. Let $Z$ be a transport interpolation set, and $u:Z\rightarrow \R$ be a continuous reverse Lipschitz function. Then $$\m(a_u)=0.$$ By considering the causally reversed hypotheses $TMCP^-(K,N)$ with backwards $p$-enb, we get that $\m(b_u)=0.$
\end{proposition}
\begin{proof}
    Assume that $\m(a_u)>0$. By inner regularity of $\m,$ there is a compact set $\hat{A}\subset a_u$ with $\m(\hat{A})>0.$ We claim that $\hat{A}\subset \T_u^{nb}\setminus b_u$  $\m$-a.e. on $\hat{A}$. Indeed, this follows by noting that $a_u\cap b_u=\emptyset,$ and $a_u\cap A_{-,u}=\emptyset $ since otherwise if $x\in A_{-,u}$ then there exists a $z\neq x$ with $(z,x)\in \Gamma_u$ which contradicts that $x\in a_u.$ Recalling that forward $p$-enb and $TMCP^+(K,N)$ gives us that $\m(A_{+,u})=0$, we have that $\m(\hat{A}\cap A_{+,u})=0$. Thus, by inner regularity, we may replace $\hat A$ with a compact positive-measure subset of $\hat A\setminus A_{+,u}$, and retaining the same notation, we have that $$\hat{A}\subset \T_u\setminus (A_{+,u}\cup A_{-,u}\cup b_u)=\T_u^{nb}\setminus b_u.$$

We now apply the forward evolution to the set $\hat{A}.$ By Proposition \ref{evolution of subsets}, there exists an $r>0$ and a compact set $B\subset \hat{A}$ with $\m(B)>0$,  together with its fixed length forward evolution $(B_t)_{t\in[0,1)}$. By Proposition \ref{TMCP along sets} we have that $$\m(B_t)\geq [\tau_{K,N}^{(1-t)(r)}]^N\m(B)\quad \textrm{for all $t\in [0,1)$},$$ and $\cup_{0\leq t<1}B_t$ is relatively compact, and we note that $\lim_{t\rightarrow 0}\tau_{K,N}^{(1-t)}(r)^N=1.$ We also have that $B\cap B_t=\emptyset$ for every $t\in (0,1).$ Indeed, if $z\in B_t,$ with $t>0$ then there exists $x\in B\subset a_u$ such that $x\neq z$ and $(x,z)\in \Gamma_u$ which means that $z$ has a nontrivial predecessor, and hence $z\not \in a_u.$ Therefore $B\cap B_t=\emptyset$.

We now claim that $B_t\rightarrow B$ as $t\downarrow 0$, meaning that for every $\eps>0,$ $B_t\subset B^\eps$ for all sufficiently small $t>0$: here, $B^\eps:=\{z:d(z,B)<\eps\}.$ We argue this by contradiction: if it were not the case, then there would exist $t_n\downarrow 0$ and $z_n\in B_{t_n}$ with $d(z_n,B)\geq \eps.$ Choosing $x_n\in B$ and a maximizing timelike geodesic $\gamma_n$ such that $z_n=\gamma_n(t_n)$, by construction, the initial and final points of these geodesics lie in a fixed compact set. By global hyperbolicity, after passing to a subsequence, the geodesics converge to a limiting causal geodesic $\gamma$ with $x_n=\g_n(0)\rightarrow \g(0)=x\in B,$ and since $t_n\downarrow0$, we get $z_n=\g_n(t_n)\rightarrow \g(0)=x\in B,$ contradicting that $d(z_n,B)\geq \eps.$

Now choose a sequence $t_n\downarrow 0$ such that for $\eps = \frac{1}{n}$ we have $B_{t_n}\subset B^\eps = B^{1/n},$ and let $K$ be a compact set containing $B$ and all $B_{t_n}$. Since $B\cap B_{t_n}=\emptyset,$ we have $$\m(B_{1/n}\cap K)\geq \m(B)+\m(B_{t_n}),$$ and by applying Proposition \ref{TMCP along sets} we get $$\m(B^{1/n}\cap K)\geq \big(1+[\tau_{K,N}^{(1-t_n)}(r)]^N\big)\m(B).$$ Sending $n\rightarrow \infty$ gives that the right hand side converges to $2\m(B)$, whereas $B_{1/n}\cap K\downarrow B$ so by the left hand side converges to $\m(B),$ which gives a contradiction.

The proof of $\m(b_u)=0$ follows from causally reversing the hypotheses.

\end{proof}

With the above result, we now focus on the transport ray $(X_\alpha,\ell\restrict_{X_\alpha},\m_\alpha)$. As shown by Cavalletti-Mondino \cite{CavallettiMondino2024}, the transport rays are  absolutely continuous. The argument under the forward $p$-enb $TMCP^+(K,N)$ assumption is identical up to accounting for the one-sidedness of forward $p$-enb, but we crucially require that the transport set is non-branching, in particular that $\m(A_{+,u})=\m(A_{-,u})=0.$ Forward $p$-enb $TMCP^+(K,N)$ spacetimes imply that $\m(A_{+,u})=0$ by Theorem \ref{endpoints measure zero}, however it does not follow that the spacetime satisfies $\m(A_{-,u})=0,$ as this result requires the causal reversal of the hypotheses. We thus include $\m(A_{-,u})=0$  as an assumption.

\begin{theorem}\label{a.c}  Let $(X,d,\m,\ell)$ be a forward $p$-enb gh LLS satisfying the $TMCP^+(K,N)$ condition. Let $Z\subset X$ be a transport interpolation set and let $u:Z\rightarrow\R$ be a continuous reverse 1-Lipschitz function. Assume that $\m(A_{-,u})=0.$ Then the non-branched transport set $\T_u^{nb}$ admits a disintegration $$\m\restrict_{\T^{nb}_u}=\int_Q\m_\alpha \q(d\alpha)$$ where for $\q$-a.e. $\alpha$, the transport ray $(X_\alpha,\ell\restrict_{X_\alpha},\m_\alpha)$ satisfy \begin{equation}\label{endpoint restriction}\m_\alpha\restrict_{X_\alpha\setminus b_u} \ll\L^1,\end{equation} where $\L^1$ is the one-dimensional Lebesgue measure along the timelike geodesic $X_\alpha.$ Equivalently, we may write \eqref{endpoint restriction} as \begin{equation}\label{atomic conditional}
\m_\alpha = h(\alpha,\cdot)\L^1\restrict_{[a_\alpha,b_\alpha)}+c_\alpha\delta_{b_\alpha}
\end{equation} where $c_\alpha\geq 0$, and the final atomic term is omitted whenever $b_u$ is not a part of the ray as in \eqref{endpoint restriction}. If we additionally assume the casually reversed hypotheses, we get that $\m(b_u)=0$ and hence \eqref{endpoint restriction} becomes $$\m_\alpha\ll \L^1$$ on the entire ray $X_\alpha$.\hfill\(\square\)

\end{theorem}

We will now show that  $\q$-a.e. transport ray $(X_\alpha,\ell\restrict_{X_\alpha},\m_\alpha)$ inherits the $TMCP^+(K,N)$ condition. We first use the ray map to express absolute continuity of $\m_\alpha$ in ray coordinates. By Theorem \ref{a.c}, for $\q$-a.e. $\alpha\in Q$, there exists a nonnegative Borel function $h(\alpha,\cdot):[a_\alpha,b_\alpha)\rightarrow \R_+,$ such that, identifying $X_\alpha$ with its parameterization through $g(\alpha,\cdot)$ we get $$\m_\alpha \restrict_{[a_\alpha,b_\alpha)}=h(\alpha,\cdot)\L^1\restrict_{[a_\alpha,b_\alpha)}. $$  In terms of the disintegration of measure, we can write $$\m\restrict_{\T_u^{nb}\setminus b_u} =\int_Q g_\#(h(\alpha,\cdot)\L^1\restrict_{[a_\alpha,b_\alpha)})d\q(\alpha)=g_{\#}\big(h(\q \otimes\L^1)\restrict_{\{(\alpha,s):a_\alpha\leq s<b_\alpha \}}\big).$$  

We finally localize the estimate from Proposition \ref{TMCP along sets} to a pointwise $TMCP^+(K,N)$ bound onto each transport ray $(X_\alpha,\ell\restrict_{X_\alpha},\m_\alpha)$.

\begin{theorem}\label{Big TMCP+ theorem} Let $(X,d,\m,\ell)$ be a forward $p$-enb gh LLS satisfying the $TMCP^+(K,N)$ condition, and let $u:Z\rightarrow\R$ be a continuous reverse Lipschitz function. Assume that $\m(A_{-,u})=0.$ Then the density $h(\alpha,\cdot)$ on the proper interior part of the transport ray  $(X_\alpha,\ell\restrict_{X_\alpha},\m_\alpha)$ for $\q$-a.e. $\alpha \in Q$  satisfies \begin{equation}\label{TMCP+ density}\frac{h(\alpha,\tau)}{h(\alpha,s)}\geq\Bigg(\frac{s_{\frac{K}{N-1}}((\sigma_+-\tau)}{s_{\frac{K}{N-1}}((\sigma_+-s)}\Bigg) ^{N-1}\end{equation} for all $a_\alpha<s\leq \tau<\sigma_+$, and $\sigma_+\in(s, b_\alpha]$  expressed in ray coordinates along the ray $X_\alpha$ at points of the chosen representative where $h(\alpha,s)>0$.

In particular, positivity propagates forward along the ray: if $h(\alpha,s)>0,$ then $h(\alpha,\tau)>0$ for every future point $\tau>s.$

By causal reversal of the hypotheses, backwards $p$-enb spacetime satisfying $TMCP^-(K,N)$ condition with $\m(A_{+,u})=0$, we have that the density $h(\alpha,\cdot)$ of $(X_\alpha,\ell\restrict_{X_\alpha},\m_\alpha)$ for $\q$-a.e. $\alpha \in Q$ satisfies \begin{equation}\label{TMCP^- density} \frac{h(\alpha,\tau)}{h(\alpha,s)} \leq \Bigg(\frac{s_{\frac{K}{N-1}}((\tau-\sigma_-)}{s_{\frac{K}{N-1}}((s-\sigma_-)}\Bigg) ^{N-1}
\end{equation}
for all $\sigma_-<s\leq \tau<b_\alpha$ where $\sigma_-\in[a_\alpha,s)$ is any point in the past of $s$ on the transport ray expressed in ray coordinates, and the density is locally upper semi-continuous. 

Finally, for $p$-enb spacetimes that satisfy the $TMCP(K,N)$ condition we have that the density  $h(\alpha,\cdot)$ satisfies 

\begin{equation}\label{TMCP density}\Bigg(\frac{s_{\frac{K}{N-1}}((\sigma_+-\tau)}{s_{\frac{K}{N-1}}((\sigma_+-s)}\Bigg) ^{N-1}\leq\frac{h(\alpha,\tau)}{h(\alpha,s)} \leq \Bigg(\frac{s_{\frac{K}{N-1}}((\tau-\sigma_-)}{s_{\frac{K}{N-1}}((s-\sigma_-)}\Bigg) ^{N-1},\end{equation}
for all $a_\alpha<\sigma_-<s\leq \tau<\sigma_+<b_\alpha$ and the density $h(\alpha,\cdot)$ is locally Lipschitz and strictly positive on the ray interior.

\end{theorem}

\begin{proof}
By Proposition \ref{ray map properties}, the ray map $g(\alpha,\cdot):[a_\alpha,b_\alpha]\rightarrow X_\alpha$ is an $\ell$-isometry, and on the proper interior we have $\m_\alpha \restrict_{[a_\alpha,b_\alpha)}=h(\alpha,\cdot)\L^1\restrict_{[a_\alpha,b_\alpha)}.$

The disintegration in ray coordinates becomes $$\\m\restrict_{\T_u^{nb}\setminus b_u} = g_{\#}\big(h(\q \otimes\L^1)\restrict_{\{(\alpha,s):a_\alpha\leq s<b_\alpha \}}\big).$$ Because the measurable quotient construction can be decomposed into countably many pieces on which the representative lies on a fixed $u$-level set, it is enough to work on one such piece $$Q = \bigcup^\infty_{j=1}Q_j,\quad u(\alpha) = c_j\quad \forall \alpha \in Q_j.$$
We fix one such $Q_j$ and suppress $j$ below. Fix $\sigma_+\in \R$ and a bounded interval $I\Subset (a_\alpha,\sigma_+),$ and choose a measurable  $C\subset Q$ such that $I\cup\{ \sigma_+\}\subset \textrm{Dom}(g(\alpha,\cdot))$ for all $\alpha\in C$.  Define $A= g(C\times I).$ For every $x=g(a,s)\in A,$ set $T(x)=g(\alpha,\sigma_+)$ so that we have  $\ell(x,T(x))=\sigma_+-s.$ Since for every $x=g(\alpha,s)$ we have $u(T(x))=c_j+\sigma_+$,
the target $u$-value is constant, and hence by Lemma \ref{p-cyclically monotone set} we have that $\{ (x,T(x)): x\in A\}$ is an $\ell^p$-cyclically monotone set. Although $\ell(x,T(x))=\sigma_+-s$ varies with $s\in I,$ the same argument as in Proposition \ref{TMCP along sets} applies: indeed, the reverse Lipschitz inequality gives, for every causally admissible $y\in T(A)$ that $$\ell(x,y)\leq u(y)-u(x)=\sigma_+-s=\ell(x,T(x),$$ and thus $T(x)$ maximizes $\ell^p(x,\cdot)$ among these targets, and using the finite-target approximation and forward nonbranching argument from the proof of Proposition \ref{TMCP along sets}, we get $$\m(A_t)\geq \inf_{x\in A}\big[ \tau_{K,N}^{(1-t)}(\ell(x,T(x))\big]^N\m(A)=\inf_{s'\in I}\big[ \tau_{K,N}^{(1-t)}(\sigma_+-s')\big]^N\m(A).$$
 Moreover, for $x=g(\alpha,s)$, the $t$-intermediate point between $x$ and $T(x)$ in ray coordinates is $s_t = (1-t)s+t\sigma_+$, and thus, setting $I_t: = (1-t)I+t\sigma_+$ we have $$A_t = g(C\times I_t).$$
Using the disintegration in ray coordinates the above estimate becomes $$\int_C\int _{I_t}h(\alpha,\tau)d\tau \q(d\alpha) \geq \inf_{s'\in I}\big[ \tau_{K,N}^{(1-t)}(\sigma_+-s')\big]^N\int_C\int_Ih(\alpha,s')ds'\q(d\alpha).$$
Now using Remark \ref{tau sigma TMCP evolution}, we conclude that $$\int_C\int _{I_t}h(\alpha,\tau)d\tau \q(d\alpha) \geq (1-t)\inf_{s'\in I}\big[ \sigma_{K,N-1}^{(1-t)}(\sigma_+-s')\big]^{N-1}\int_C\int_Ih(\alpha,s')ds'\q(d\alpha),$$ and since $C\subset Q_j$ is arbitrary, we conclude that for $\q$-a.e. $\alpha \in Q_j,$ we have  $$\int_{I_t}h(\alpha,\tau)d\tau  \geq (1-t)\inf_{s'\in I}\big[ \sigma_{K,N-1}^{(1-t)}(\sigma_+-s')\big]^{N-1}\int_Ih(\alpha,s')ds'.$$
Now letting $I\downarrow\{s\}$ where $s$ is a Lebesgue point of $h(\alpha,\cdot)$, the interval $I_t$ shrinks to $\tau = (1-t)s+t\sigma_+,$ and $|I_t|=(1-t)|I|.$ Dividing the above inequality by $(1-t)|I|$ and applying the Lebesgue differentiation theorem gives us $$h(\alpha,\tau)\geq \big[ \sigma_{K,N-1}^{(1-t)}(\sigma_+-s)\big]^{N-1}h(\alpha,s).$$ Finally, since $1-t = \frac{\sigma_+-\tau}{\sigma_+-s}$,  the curvature coefficients $\sigma_{K,N-1}^{(1-t)}$ reduce to $$\frac{h(\alpha,\tau)}{h(\alpha,s)}\geq \bigg(\frac{s_{K/(N-1)}(\sigma_+-\tau)}{s_{K/(N-1)}(\sigma_+-s)}\bigg)^{N-1},$$ which holds for $a_\alpha<s\leq \tau<\sigma_+\leq b_\alpha$ with $h(\alpha,s)>0,$ and in particular $h(\alpha,s)>0$ implies that $h(\alpha,\tau)>0$ for every future interior point $\tau>s$.

Under causal reversal, the same argument gives the backward estimate \eqref{TMCP^- density}, and combining the two yields the full two sided $TMCP(K,N)$ estimate \eqref{TMCP density}.
\end{proof}

\begin{remark}
    The $TMCP^+(K,N)$ bound \eqref{TMCP+ density} on $h_\alpha$ can be written more compactly as \begin{equation}\label{TMCP+ compact} \frac{h(\alpha,\tau)}{h(\alpha,s)}\geq \bigg[\sigma_{K,N-1}^{\big(\frac{\sigma_+-\tau}{\sigma_+-s}\big)}(\sigma_+-s)\bigg]^{N-1},
    \end{equation}  and the $TMCP^-(K,N)$ bound \eqref{TMCP^- density} as \begin{equation}\label{TMCP- compact} \frac{h(\alpha,\tau)}{h(\alpha,s)}\leq \bigg[\sigma_{K,N-1}^{\big(\frac{s-\sigma_-}{\tau-\sigma_-}\big)}(\tau-\sigma_-)\bigg]^{-(N-1)},\end{equation} 
    where $a_\alpha<\sigma_-<s\leq \tau<\sigma_+<b_\alpha$.
\end{remark}

\section{Solution to the Lorentzian Monge problem}

In this section, we prove Theorem \ref{Solution to Monge} by constructing an optimal map for the Lorentzian Monge problem.  The results of this section are based on the metric measure space construction first presented by Cavalletti \cite{Cavalletti2014}. The strategy is to use structure of the non-branched transport set to perform a reduction of the full problem onto timelike geodesics. The ray map $g$ will the allow us to treat each timelike geodesic as an  interval in $\R$, where the optimal transport map can be constructed explicitly. This construction thus provides the solution to the original problem by providing a Lorentzian Monge map along each ray.

We will assume throughout this section that  $(X,d,\m,\ell)$ is a gh LLS that is forward $p$-enb, and satisfies the $TMCP^+(K,N)$ condition with fixed $p\in(0,1)$. $K\in\R,$ $N\in(1,\infty).$ 

Starting with a  timelike 1-dualisable pair $(\mu_0,\mu_1)\in \P_c(X)^2$ with $\mu_0\ll \m,$ let $\eta \in \Pi_\ll^{1\textrm{-opt}}(\mu_0,\mu_1)$, and assume that there exists a continuous reverse 1-Lipschitz Kantorovich potential $u:X\rightarrow \R$ such that $\eta(\Gu)=1$. We will use the results of Section \ref{Disintegration section} to obtain the non-branched transport set $$\T_u^{nb}=\T_u\setminus (A_{+,u}\cup A_{-,u}).$$ The forward $p$-enb and $TMCP^+(K,N)$ conditions give $\m(A_{+,u})=0$ by Theorem \ref{endpoints measure zero}. To get a clean disintegration of $\m$ onto each transport ray $(X_\alpha,\ell\restrict_{X_\alpha},\m_\alpha)$, we additionally require that $\m(A_{-,u})=0$, that is, that the backwards branching happens on a set of measure zero. Since this does not follow from the forward $p$-enb and the $TMCP^+(K,N)$ assumption, we will explicitly assume that \begin{equation}\label{backwards branching measure zero} \m(A_{-,u})=0.
\end{equation}
At the end of this section, we will provide sufficient conditions for which this condition is satisfied.

With the above assumptions, Theorem \ref{a.c}   gives us a disintegration of $\m\restrict_{\T_u^{nb}}$, whereby each equivalence class is a transport ray $(X_\alpha,\ell\restrict_{X_\alpha},\m_\alpha)$: 

 $$\m\restrict_{\T^{nb}_u}=\int_Q\m_\alpha \q(d\alpha)$$ where for $\q$-a.e. $\alpha$, the transport ray $(X_\alpha,\ell\restrict_{X_\alpha},\m_\alpha)$ satisfies $$\m_\alpha \restrict_{[a_\alpha,b_\alpha)}=h(\alpha,\cdot)\L^1\restrict_{[a_\alpha,b_\alpha)}. $$ where $\L^1$ is the one-dimensional Lebesgue measure along the timelike geodesic $X_\alpha.$ 
 
The following Lemma shows that we can restrict the optimal plan $\eta$ for the timelike 1-dualisable pair $(\mu_0,\mu_1)$ to the transport set $\T_u$.

\begin{lemma}[Restriction to the transport set]\label{Restriction to the transport set} Let $\eta\in \Pi_{\ll}^{1\textrm{-opt}}(\mu_0,\mu_1)$, and let $u$ be a continuous reverse 1-Lipschitz potential from Theorem \ref{Kantorovich Potential} satisfying $\eta(\Gu)=1$. Then  $$\eta ((\T_u\times \T_u))=1.$$
\end{lemma}
\begin{proof}
    Let $(z,w)\in \Gamma_u,$ with $z\neq w.$ This pair $(z,w)$ is in $R_u\setminus \{ x=y\}$, and by the definition of the transport relation $R_u = \Gamma_u\cup \Gamma_u^{-1}$ we get that $z\in P_1(R_u\setminus \{ x=y\}=\T_u.$  Additionally, $w\in \T_u$ as $(w,z)\in \Gamma_u^{-1}.$ Thus:
    $$(z,w)\in \T_u\times \T_u,$$ which implies that $\Gamma_u\setminus \Delta_X\subset \T_u\times \T_u$. Since $\eta$ is supported on $\Gamma_u$, it follows that $$\eta(\T_u\times \T_u\cup \Delta_X)\geq \eta(\Gamma_u)=1.$$ Since $(\mu_0,\mu_1)$ is timelike 1-dualisable, $\eta$ is concentrated on strictly timelike pairs. Hence $\eta(\Delta_X)=0,$ and we conclude that $\eta(\T_u\times \T_u)=1.$
\end{proof}

From this lemma, since $\eta(\T_u\times \T_u)=1$ we have $\mu_0(\T_u)=\mu_1(\T_u)=1.$ Moreover, we have that $\mu_0(b_u)=0$, since if $x\in b_u$ and $(x,y)\in\Gamma_u$ is timelike, then $y\neq x$ would be a nontrivial forward continuation of $x$, contradicting the definition of $b_u$. Analogously, we have $\mu_1(a_u)=0.$ Since we are assuming that $\mu_0\ll\m,$ and since $\m(a_u)=0$ by Proposition \ref{m(a)=0}, we have that $\mu_0(a_u)=0.$  Since $\m(A_{+,u})=0$ by Theorem \ref{endpoints measure zero} and $\m(A_{-,u})=0$ by our standing assumption \eqref{backwards branching measure zero},
we get $\mu_0(A_{+,u}\cup A_{-,u})=0$, and in particular $\mu_0(\T_u^{nb}) =1$. We can therefore use the disintegration of $\m\restrict_{\T_u^{nb}}$ to obtain a disintegration of  $\mu_0$: \begin{equation}\label{mu0 disintegration}\mu_0=\rho_0\m\restrict_{\T^{nb}_u}=\rho_0\m\restrict_{\T_u}=\int_Q\rho_0\m_{\alpha}\q(d\alpha)=\int_Q \mu_{0,\alpha}\q_{\mu_0}(d\alpha),
\end{equation}
\noindent where $\mu_{0,\alpha},=c(\alpha)\rho_0\m_{\alpha}$ and $c(\alpha)$ is a normalizing constant chosen to that $\q_{\mu_0}=\frac{\q}{c(\alpha)}.$  Since $\mu_0(b_u)=0$, disintegration gives us $\mu_{0,\alpha}(\{ b_u\})=0$ for $\q_{\mu_0}$-a.e. $\alpha$, and combining this with the absolute continuity of $\m_\alpha$ on the proper ray interior yields $\mu_{0,\alpha}\ll \L^1$ for $\q_{\mu_0}$-a.e. $\alpha$.

Note that it is possible that $\mu_1(b_u)>0$, since we are only assuming that $\mu_0\ll\m$. In particular, the target measure $\mu_1$ may charge terminal endpoints of transport rays, and therefore need not be concentrated on $\T_u^{nb}$.  We will however obtain a disintegration of $\mu_1$ that is compatible with the disintegration of $\mu_0$ by disintegrating the optimal coupling $\eta\in \Pi_{\ll}^{1\rm{-opt}}(\mu_0,\mu_1)$ in the following Lemma.

\begin{lemma}\label{joint coupling}
    For $\eta\in \Pi_{\ll}^{1\textrm{-opt}}(\mu_0,\mu_1)$, we have the following disintegration:

    $$\eta =\int_Q\eta_\alpha \q_{\mu_0}(d\alpha),\quad \eta_\alpha\in \P_c((R_u^{nb}(\alpha)\cap \T_u)\times R_u(\alpha)),$$ with $(P_1)_\#\eta_\alpha=\mu_{0,\alpha}.$
\end{lemma}
\begin{proof}

By Lemma \ref{Restriction to the transport set} and the preceding discussion, we in fact have that $\eta$ is concentrated on $\T_u^{nb}\times \T_u$. We obtain a partition of  $(\T_u^{nb}\times X)\cap \Gamma_u$ through the partition of $\T_u^{nb}$ given by the equivalence classes $\{R_u^{nb}(\alpha)\}_{\alpha\in Q}:$

    $$(\T_u^{nb}\times X)\cap \Gamma_u = \bigcup_{\alpha\in Q}(R_u^{nb}(\alpha)\times X)\cap\Gamma_u,$$
    and the disintegration theorem  gives us the following disintegration of $\eta:$
    $$\eta=\int_Q \eta_\alpha \q_{\eta}(d\alpha),\quad \eta_\alpha\in \P_c(((R_u^{nb}(\alpha)\times X)\cap\Gamma_u).$$ Since $(R_u^{nb}(\alpha)\times X)\cap\Gamma_u\subset (R_u^{nb}(\alpha)\cap \T_u)\times R_u(\alpha),$ all we are left to show is that $$\q_{\mu_0}=\q_{\eta}.$$ Indeed, if we consider any Borel $A\subset Q,$ we see that $$\mu_0(f^{-1}(A))=\eta(f^{-1}(A)\times X)=\eta(\{(x,y)\in\Gamma_u:x\in \T_u^{nb}, f(x)\in A\}),$$ where $f:\T_u^{nb}\rightarrow Q$ is the quotient map from \ref{quotient map} associated with the disintegration of $\T_u^{nb}$. The right hand side is exactly $\q_\eta(A)$ and the left hand side is $\q_{\mu_0}(A)$, and thus $\q_{\eta}=\q_{\mu_0}$ which provides the desired disintegration.
\end{proof}

This lemma gives us the following disintegration of $\mu_1:$ \begin{equation}\label{mu1 disintegration}
\mu_1 = (P_2)_{\#}\eta = \int_Q (P_2)_{\#}\eta_\alpha \q_{\mu_0}(d\alpha)=:\int_Q \mu_{1,\alpha}\q_{\mu_0}(d\alpha)\end{equation}

\noindent and we have that the coupling $\eta_\alpha\in \Pi_{\ll}^{1\textrm{-opt}}(\mu_{0,\alpha},\mu_{1,\alpha})$ is timelike 1-dualizable, and hence optimal for $\q_{\mu_0}$-a.e. $\alpha\in Q.$ Since $\mu_0(\T_u^{nb})=1,$ it suffices to construct a Borel map $T$ on $\T_u^{nb}$ whose values are allowed to lie at future endpoints of the corresponding completed rays such that $T_\#\mu_0=\mu_1.$ Once this map $T$ is defined $\mu_0$-a.e. on $\T_u^{nb}$, it may be extended arbitrarily to $X\setminus \T_u^{nb},$ for example by setting it to the identity map on $X\setminus \T_u^{nb}.$

To solve the Lorentzian Monge problem, we construct a map which moves mass monotonically forward in time along each transport ray. Recall that for each equivalence class $\alpha\in Q$, the transport ray $(X_,\ell\restrict_{X_\alpha}, \m_\alpha)$, so we have $\ell(g(\alpha,s),g(\alpha,t))=t-s$ whenever $s\leq t$ belong to $\textrm{Dom}(g(\alpha,\cdot))$. We can thus identify the measures $\mu_{0,\alpha}$ and $\mu_{1,\alpha}$ on the transport ray with their cumulative distribution functions on  $\R$: 

$$H(\alpha,t):= \mu_{0,\alpha}((-\infty ,t)),\quad F(\alpha,t):=\mu_{1,\alpha}((-\infty,t)).$$ Since $\mu_{0,\alpha}\ll \L^1,$ the map $t\mapsto H(\alpha,t)$ is continuous for $\q_{\mu_0}$-a.e. $\alpha$. The one dimensional solution to the Monge problem is then given by the following monotone map $T$ \cite[Theorem 2.18]{VillaniTopics}: \begin{equation}\label{monotone map}
T(\alpha,s):= \big(\alpha,\,\sup \{t\in\textrm{Dom}(g(\alpha,\cdot)): F(\alpha,t)\leq H(\alpha,s) \} \big). 
\end{equation}
We need to ensure that $\mu_0$ and $\mu_1$ assign the same amount of mass to each transport ray. To that end, we first require the following definition adapted from \cite{CaffarelliFeldmanMcCann2002}:

\begin{definition} A set $A\subset \T_u^{nb}$ is called \emph{ray-saturated} if it is a union of complete transport rays. Equivalently, there exists a Borel set $Y\subset Q$ such that $A = f^{-1}(Y).$ A Borel set $A^+\subset \T_u^{nb}$ is called a \emph{positive end} if whenever $z\in A^+$ and $(z,w)\in \Gu$, then $w\in A^+.$ In ray coordinates, the intersection of a positive end with each transport ray is an interval of the form $[s_0(\alpha),b_{u(\alpha})]$ or $(s_0(\alpha),b_{u}(\alpha)]$, where  $s_0 \in [a_u(\alpha), b_u(\alpha)]$.

\end{definition}

The following Lemma is an extension  of Caffarelli-Feldman-McCann \cite[Lemma 27]{CaffarelliFeldmanMcCann2002} that provides the required mass balance of rays, and additionally shows that the map $T(\alpha,s)$ from \eqref{monotone map} satisfies $T(\alpha,s)\geq s.$ The result $T(\alpha,s)\geq s$ is crucial, as it shows that along each transport ray the Lorentzian Monge map moves mass monotonically forward in time towards the future endpoint $b_\alpha$, consistent with the Lorentzian cost $\ell\restrict_{X_\alpha}(g(\alpha,s),g(\alpha,t))=t-s$ for $t\geq s.$

\begin{lemma}\label{mass balance} Let $\mu_0,\mu_1\in \P_c(X)$ be a timelike 1-dualisable  pair of measures. If $A=f^{-1}(Y)$ is ray-saturated, then $$\mu_0(A)=\q_{\mu_0}(Y)=\int_Y \mu_{1,\alpha}(X)\q_{\mu_0}(d\alpha).$$ More generally, if a Borel set $A^+\subset X$ forms a positive end, then $$\mu_0(A^+)\leq \mu_1(A^+).$$

    Additionally, for $\q_{\mu_0}$-a.e. $\alpha$ transport ray, we have $$F(\alpha,s)\leq H(\alpha,s),$$ and consequently the map $T$ defined by \eqref{monotone map} satisfies $T(\alpha,s)\geq s$ for $\mu_{0,\alpha}$-a.e. $s.$
\end{lemma}

\begin{proof}

If $x\in A^+,$ and $(x,y)\in \Gamma_u$, then $y$ lies in the future of $x$ along the same transport ray. By definition of a positive end, we have $y\in A^+.$ Since $\eta$ is concentrated on $\Gamma_u,$ we have $$\mu_0(A^+)=\eta(A^+\times X)=\eta(A^+\times A^+)\leq \eta(X\times A^+)=\mu_1(A^+).$$ 

We now show the CDF inequality. For $\q_{\mu_0}$-a.e. $\alpha,$ the coupling $\eta_\alpha$ is timelike 1-dualisable, so in ray coordinates $s\leq t$  $\eta_\alpha$-a.e. Therefore $\{ t<r\}\subset \{ s< r\}$, and hence we get $$F(\alpha, r)=\eta_\alpha (\{ t<r\})\leq \eta_\alpha(\{ s<r\})=H(\alpha,r).$$  Since $F(\alpha,s)\leq H(\alpha,s)$,  the point $t=s$ belongs to $\{t:F(\alpha,t)\leq H(\alpha,s)\}$, and by recalling the definition of \eqref{monotone map} we conclude that $T(\alpha,s)\geq s$.

Finally, if $A=f^{-1}(Y)$ is ray saturated, then $$\mu_0(A)=\int_Y\mu_{0,\alpha}(X_\alpha)\q_{\mu_0}(d\alpha)=\q_{\mu_0}(Y),$$ and since $\mu_{1,\alpha}$ is a probability measure for $\q_{\mu_0}$-a.e. $\alpha,$ we also have $$\int_Y\mu_{1,\alpha}(X)\q_{\mu_0}(d\alpha)=\q_{\mu_0}(Y).$$ Combining the above two equalities, we get $$\mu_0(A)=\q_{\mu_0}(Y)=\int_Y\mu_{1,\alpha}(X)\q_{\mu_0}(d\alpha).$$
\end{proof}

We now construct an optimal map for the Lorentzian Monge problem.

\begin{theorem}\label{Lorentzian Monge}
    Let $(X,d,\m,\ell)$ be a gh LLS, that is forward $p$-enb and satisfies the $TMCP^+(K,N)$ for $K\in \R, N\in(1,\infty), p\in (0,1)$, and assume that $(\mu_0,\mu_1)\in \P_c(X)^2$ with $\mu_0\ll \m$ is a timelike 1-dualizable pair, with $\eta \in \Pi_{\ll}^{1\textrm{-opt}}(\mu_{0},\mu_{1})$, and let $u$ be a continuous reverse Lipschitz Kantorovich potential satisfying $\eta(\Gu)=1$. Assume that $\m(A_{-,u})=0$. Then there exists a Borel map $T:X\rightarrow X$ such that $T_{\#}\mu_0=\mu_1$, and $$\int_X \ell(x,T(x))d\mu_0(x)=\int_{X\times X} \ell(x,y)d\eta(x,y).$$
\end{theorem}

\begin{proof}

By Lemma \ref{Restriction to the transport set}, $\eta(\T_u\times \T_u)=1,$  so $\mu_0(\T_u)=\mu_1(\T_u)=1,$ and since $\mu_0\ll \m$ and $\m(a_u)=\m(A_{+,u})=\m(A_{-,u})=0,$ we have that $\mu_0(\T_u^{nb})=1.$ It therefore suffices to construct an optimal map on $\T_u^{nb}$ and extend it by identity to all of $X.$

By Theorem \ref{a.c}, the reference measure $\m$ disintegrates to $\T_u^{nb}$ into transport rays $(X_\alpha,\ell\restrict_{X_\alpha},\m_\alpha)$ with $\m_\alpha\restrict_{[a_\alpha,b_\alpha)}\ll\L^1,$ with at most a possible atom at the endpoint $b_\alpha$. Since we have established that $\mu_0(b_u)=0,$ we have that $\mu_{0,\alpha}\ll \L^1$ for $\q_{\mu_0}$-a.e. $\alpha$. By \eqref{mu0 disintegration} and \eqref{mu1 disintegration}, we get the disintegrations $$\mu_0 = \int_Q\mu_{0,\alpha}\q_{\mu_0}(d\alpha),\quad \mu_1 = \int_Q\mu_{1,\alpha}\q_{\mu_0}(d\alpha),$$ where each $\mu_{0,\alpha}\ll \L^1$ on each ray transport ray $X_\alpha$, and the optimal coupling $\eta$ also disintegrates as $\eta = \int_Q\eta_\alpha \q_{\mu_0}(d\alpha)$ with $$\eta_\alpha \in \Pi^{1\textrm{-opt}}_\ll(\mu_{0,\alpha},\mu_{1,\alpha})\quad \textrm{for $\q_{\mu_0}$-a.e. $\alpha$}. $$ 

We now use the Lorentzian cost on each transport ray. In particular if $s\leq t$, then the $\ell$-isometry via the ray map gives us that \begin{equation}\label{cost}\ell(g(\alpha,s),g(\alpha,t))=t-s.
\end{equation}

The disintegrations of $\mu_0$ and $\mu_1$ give us that the maps $\alpha \mapsto \mu_{0,\alpha}, \mu_{1,\alpha}$ are measurable, and consequently the following maps are Borel $$(\alpha,t)\mapsto H(\alpha,t):=\mu_{0,\alpha}((-\infty,t)),\quad (\alpha,t)\mapsto F(\alpha,t):=\mu_{1,\alpha}((-\infty,t)),$$  In particular these maps are increasing in $t$, and $H(\alpha,\cdot)$ is also continuous in $t$ for $\q_{\mu_0}$-a.e. $\alpha\in Q$ due to the absolute continuity of $\mu_{0,\alpha}$. 

Define the map $T(\alpha,s)$ from \eqref{monotone map} $$T(\alpha,s):=\bigg(\alpha,\sup\{t:F(\alpha,t)\leq H(\alpha,s)\}\bigg),$$
This map is Borel, and crucially satisfies $T(\alpha,s)\geq s$ for $\mu_{0,\alpha}$-a.e. $s$ by Lemma \ref{mass balance}. Moreover, for every Borel $A\subset Q$ and $t\in \R$, we have $$T^{-1}(A\times [t,\infty))=\{(\alpha,s):\alpha\in A, H(\alpha,s)\geq F(\alpha,t)\}\in \B(Q,\R).$$

\noindent Since $\mu_{0,\alpha}$ has no atoms for $\q_{\mu_0}$-a.e. $\alpha\in \T_u^{nb},$ the map $T$ is  an optimal transport map between the measures $\mu_{0,\alpha}$ and $\mu_{1,\alpha}$ with the cost $\ell$ given by \eqref{cost}. Since the measure $\eta_\alpha$ is also the optimal coupling between $\mu_{0,\alpha}$ and $\mu_{1,\alpha},$ it follows by Lemma \ref{joint coupling} that

\begin{align*}
\int_X \ell(x,T(x))d\mu_0(x)= & \int_Q\int_{\R} (T(\alpha,s)-s) d\mu_{0,\alpha}(s)d\q_{\mu_0}(\alpha)\\
=&\int_Q\int_{X\times X} \ell(x,y)d\eta_\alpha(x,y)d\q_{\mu_0}(\alpha)\\
=& \int_{X\times X}\ell(x,y) d\eta(x,y), 
\end{align*}

\noindent where we have used Lemma \ref{mass balance} to ensure $T(\alpha,s)\geq s$ in the first equality.

\end{proof}

We now extend the above construction to general compactly supported measures given by Theorem \ref{decomposing pi}, where the compactly supported initial and final measures $\mu_0,\mu_1$ admit an $\ell^1$-optimal coupling $\pi$ such that $\ell(x,y)>0$ for $\pi$-almost every $(x,y).$ Such a coupling can be decomposed into countably many compactly supported pieces. However, rather than solving the Monge problem separately on these pieces, whose source marginals need not be mutually singular, we work directly with an $\ell^1$-cyclically monotone relation generating this coupling This extension is inspired by the works of Kell-Suhr \cite{Kell-Suhr}.

We first extend the definitions of a transport relation, and branching sets. Let   $\pi\in \Pi_\leq ^{1\textrm{-opt}}(\mu_0,\mu_1)$ satisfy $\ell>0$ $\pi$-a.e. By optimality, we may fix an $\ell^1$-cyclically monotone Borel set $\Gamma_0$ that satisfies $$\Gamma_0\subset \{ \ell>0\},\quad \pi(\Gamma_0)=1.$$ 
By inner regularity, we may additionally assume that $\Gamma_0$ is $\sigma$-compact. For any $E\subset X_\ll^2$, define the \emph{geodesic subdivision and geodesic concatenation operations respectively} as  $S(E)$ and $C(E)$: $$S(E):=\{ (a,b)\in X_\ll^2:\exists (x,y)\in E,\,x\leq a\ll b\leq y,\, \ell(x,y)=\ell(x,a)+\ell(a,b)+\ell(b,y)\},$$ $$C(E):= \{ (x,z)\in X_\ll^2:\exists y\in X,\,(x,y),(y,z)\in E,\,\ell(x,z)=\ell(x,y)+\ell(y,z)\}.$$
Here, $S(E)$ includes all timelike subsegments of maximizing geodesics joining pairs in $E$, and $C(E)$ includes pairs $(x,z)$ obtained by concatenating two pairs $(x,y), (y,z)\in E$, provided that $\ell(x,z)=\ell(x,y)+\ell(y,z)$ so that the concatenation remains maximizing. We iteratively define $$\Gamma^{(0)}:=\Gamma_0,\quad \Gamma^{(n+1)}:=\Gamma^{(n)}\cup S(\Gamma^{(n)})\cup C(\Gamma^{(n)}),$$ and finally $$\Gamma:=\bigcup_{n=0}^\infty \Gamma^{(n)}.$$ The reverse triangle inequality and $\ell^1$-cyclical monotonicity imply that each $\Gamma^{(n)}$ is $\ell^1$-cyclically monotone, and global hyperbolicity and continuity of $\ell_+$ imply that $S$ and $C$ preserve $\sigma$-compactness. Consequently $\Gamma$ is a $\sigma$-compact $\ell^1$-cyclically monotone relation satisfying $\Gamma_0\subset \Gamma$, and $S(\Gamma)=\Gamma=C(\Gamma)$, and thus $\Gamma$ is saturated under maximizing geodesic subdivision and concatenation. Thus, take $\Gamma \supset \Gamma_0$ as such a saturated $\ell$-cyclically monotone relation constructed above and define $$\Gamma_\ll:=\{ (x,y)\in \Gamma:\ell(x,y)>0\},\quad R:=\Gamma_\ll \cup \Gamma^{-1}_\ll \cup \Delta_X.$$ Additionally define the \emph{forward and backward branching sets} associated to $\Gamma$ respectively as  $$A_+:=\{ x:\exists y_1,y_2\neq x\textrm{ with $(x,y_i)\in \Gamma_\ll, (y_1,y_2)\not\in R$} \},$$
$$A_-:=\{ y:\exists x_1,x_2\neq y\textrm{ with $(x_i,y)\in \Gamma_\ll, (x_1,x_2)\not\in R$} \}.$$ We can now state the theorem.

\begin{theorem} Let $(X,d,\m,\ell)$ be a gh LLS which is forward $p$-enb and satisfies the $TMCP^+(K,N)$ condition for some $p\in(0,1).$ Let $\mu_0,\mu_1$ be compactly supported probability measures with $\mu_0\ll \m$, and let $\pi\in \Pi_\leq ^{1\textrm{-opt}}(\mu_0,\mu_1)$ satisfy $\pi[\ell(x,y)>0]=1$. Assume that $\m(A_-)=0$. Then there exists a Borel map $T:X\rightarrow X$ such that $T_\#\mu_0=\mu_1,$ and $$\int_X\ell(x,T(x))d\mu_0(x)=\int_{X\times X}\ell(x,y)d\pi(x,y).$$ In particular, $(Id,T)_\#\mu_0$ is an $\ell_1$-optimal coupling.
    
\end{theorem}
\begin{proof}

    By construction, the relation $\Gamma$ is $\sigma$-compact and saturated under maximizing geodesic subdivision and concatenation. Thus, it has the same saturation properties as the relations $\Gamma_u$ in the preceding theorem. Indeed, if $(x,y)\in \Gamma_\ll$, and $z$ is on a maximizing timelike geodesic from $x$ to $y,$ then  maximal $\ell^1$-cyclical monotonicity and the reverse triangle inequality give $(x,z),(z,y)\in \Gamma.$ Similarly, if $(x,y),(y,z)\in \Gamma_\ll$, then $(x,z)\in \Gamma$ and $\ell(x,z)=\ell(x,y)+\ell(y,z).$ 

    We first argue that $\m(A_+)=0.$ This argument is identical to the one given by Theorem \ref{endpoints measure zero}, where the role of the reverse 1-Lipschitz potential $u$ is replaced by $\ell^1$-cyclical monotonicity of $\Gamma$: indeed, if $\m(A_+)>0,$ then after measurable selection and restriction to a compact subset, the two branches can be shortened to a common timelike length $r>0$. The resulting fixed length relation is $\ell^p$-cyclically monotone, and the forward $p$-enb argument used in \ref{endpoints measure zero} then excludes such a branching configuration, where the contradiction comes from $\ell^p$ transport plans being induced by maps. Hence, we get $\m(A_+)=0.$ 

    Since by assumption, $\m(A_-)=0,$ we therefore get $\m(A_+\cup A_-)=0$, and since $\mu_0\ll \m$, we get that $\mu_0(A_+\cup A_-)=0.$ Define the nonbranched transport set $$\T^{nb}:=P_1(R\setminus \Delta_X)\setminus (A_+\cup A_-).$$ The construction of Sections \ref{S: transport set and transport relation} and \ref{S: regularity of disintegration}  extend to the present relation $\Gamma$: indeed, the proofs use only the reverse triangle inequality, cyclical monotonicity, and saturation under geodesic subdivision and concatenation which we have established for $\Gamma$. Consequently, after removing $A_+\cup A_-$, the equivalence classes of $R$ are totally ordered timelike transport rays.

    To obtain the measurable quotient, we follow the construction of Bianchini-Cavalletti as in Proposition \ref{quotient map}. Fix a continuous time function  $\vartheta:X\rightarrow \R$ from Proposition 2.12 of Beran-Harvey-Napper-Rott \cite{Beran-Harvey-Napper-Rott} and enumerate $\Q:=\{ q_j\}_{j\in \N}$. Since $\vartheta$ is continuous and strictly increasing along each timelike transport ray,  every nontrivial  transport ray intersects some level set $\vartheta^{-1}(q_j)$ exactly at one point. Selecting the first such level and taking singleton classes as their own representative we get a measurable quotient map into the quotient space $Q$ via $f:\T^{nb}\rightarrow Q$, and we can construct the corresponding ray map $g$ as in Definition \ref{ray map definition} using the signed $\ell$-distance from the selected representative. Thus the measurable quotient, ray map, and strongly consistent disintegration from  Sections \ref{S: transport set and transport relation} and \ref{S: regularity of disintegration} extend to $\Gamma.$ In particular, we have $$\ell(g(\alpha,s),g(\alpha,t))=t-s,\quad s\leq t.$$ 
    We now disintegrate the source measure $\mu_0$ with respect to the quotient map $f$: $$\mu_0=\int_Q\mu_{0,\alpha}\q(d\alpha),$$ and we argue that $\mu_{0,\alpha}$ is nonatomic for $\q$-a.e. $\alpha.$ Suppose not, i.e. that $\mu_{0,\alpha}$ has an atom on  $\bar{Q}\subset Q$, which has positive $\q$-measure. Then there exists some $n\in \N$ such that the set of rays carrying an atom of mass at least $1/n$ has positive $\q$-measure. By measurable selection, we can choose one such atom $x_\alpha$ on each of these rays. Define $S:=\{x_\alpha\}_{\alpha \in\bar{Q}},$ and note that we can write $$\mu_0(S)=\int_{\bar{Q}} \mu_{0,\alpha}(\{x_\alpha\})\q(d\alpha)>0,$$ where moreover $\mu_0\ll \m$ implies that $\m(S)>0$. We now extract a uniform positive transport length. Disintegrating $\pi$  with respect to its first marginal as $$\pi = \int \pi_x\mu_0(dx),$$ and since $\ell(x,y)>0$ for $\pi$-a.e. $(x,y),$ we have that for $\mu_0$-a.e. $x\in S$ there exists a point in the transport relation with $\ell(x,y)>0$. For $k\in \N,$ let $$S_k=\big\{x\in S:\exists y\in \spt \mu_1,\,(x,y)\in \Gamma_0,\,\ell(x,y)\geq \frac{1}{k} \big\}.$$
    Then we have that up to a $\mu_0$-null set that $S=\cup_{k=1}^\infty S_k$, and hence for some $k$ we have $\mu_0(S_k)>0$ and $\m(S_k)>0$.
    Fix such a $k$, and choose $r$ such that $0<r<\frac{1}{k}.$ By inner regularity, we can choose a compact $B\subset S_k$ of positive measure. For each $x\in B$, choose a point $y_x$ such that $$(x,y_x)\in \Gamma_0,\quad \ell(x,y_x)\geq \frac{1}{k}>r.$$ Taking a maximizing timelike geodesic from $x$ to $y_x$, and letting $z_x$ be a point of timelike distance $r$ from $x$, we get that for the maximal relation $\Gamma_\ll$: $$(x,z_x)\in \Gamma_\ll, \quad \ell(x,z_x)=r.$$ We can thus define $$\Lambda_r=\{ (x,z_x)\in \Gamma_\ll:x\in B\}$$ as in Proposition \ref{evolution of subsets}, and this $\Lambda_r$ is automatically $\ell^p$-cyclically monotone. Following the proof of Theorem \ref{evolution of subsets}, we denote by $B_t$ the evolution of $B$ obtained by moving each $x\in B$ a fraction $t$ of the distance $r$ along its corresponding transport ray. Then noting the bound in Proposition \ref{TMCP along sets} $\m(B_t)\geq [\tau _{K,N}^{(1-t)}(r)]^N\m(B),$ since $\tau _{K,N}^{(1-t)}(r)\rightarrow 1$ as $t\downarrow 0,$ there exists a $t_0$ and a $c>0$ such that $$\m(B_t)\geq c\m(B)>0\quad \textrm{for every $0<t<t_0$,}$$ and arguing as in Theorem \ref{evolution of subsets}  we get that $\m(B_s\cap B_t)=0.$

    Now choose an infinitely many distinct times $t_j\in (0,t_0)$. Since all of the sets $B_{t_j}$ are contained in a single compact causal diamond $B_{t_j}\subset J^+(B)\cap J^-(\spt \mu_1)$, we get that $$\m(J^+(B)\cap J^-(\spt \mu_1))<\infty.$$ However, the pairwise $\m$-disjointness of the $B_{t_j}$ gives $$\m(J^+(B)\cap J^-(\spt \mu_1))\geq \m(\cup_{j=1}^\infty B_{t_j})=\sum_{j=1}^\infty\m(B_{t_j})\geq \sum_{j=1}^\infty c \m(B)=+\infty,$$ which is a contradiction. Therefore $\mu_{0,\alpha}$ is nonatomic for $\q$-a.e. $\alpha.$

    From here, the proof proceeds analogously as in the previous Theorem. We disintegrate $\pi$ via the quotient map $$\pi =\int_Q\pi_\alpha \q(d\alpha),\quad (P_1)_\#\pi_\alpha =\mu_{0,\alpha},$$ and define $\mu_{1,\alpha}:=(P_2)_\#\pi_\alpha.$ We thus have the disintegration $$\mu_1 =\int _Q {\mu_{1,\alpha}}\q(d\alpha),$$ and since $\pi$ is concentrated on $\Gamma_\ll,$ each $\pi_\alpha$ transports mass monotonically forward along the corresponding transport ray. We moreover have that the conditional CDF's satisfy $$F(\alpha,t)\leq H(\alpha,t),$$ and thus the one-dimensional monotone rearrangement  \eqref{monotone map} defines the map $T(\alpha,s)$ such that $(T_\alpha)_\#\mu_{0,\alpha}=\mu_{1,\alpha},$ and $T(\alpha,s)\geq s.$ From this point, the proof follows verbatim the previous Theorem \ref{Lorentzian Monge}.
\end{proof}

We now discuss the assumption $\m(A_{-,u})=0$ appearing in the hypothesis of Theorem \ref{Lorentzian Monge}. If we assume that the spacetime satisfies both forward and backwards $p$-enb condition together with the full $TMCP(K,N)=TCMP^+(K,N)\cap TMCP^-(K,N)$, then it follows by Theorem \ref{endpoints measure zero} that $\m(A_{+,u})=\m(A_{-,u})=0$ and the full proof of Theorem \ref{Lorentzian Monge} follows. We note that by the proof of Theorem \ref{forward p-enb implies backward p-enb} in the Appendix, the backward $p$-enb assumption is redundant in the $TCD_p(K,N)$ setting, since  we have proved that forward $p$-enb $TCD_p(K,N)$  spacetimes are automatically backward $p$-enb.

In the forward $p$-enb $TMCP^+(K,N)$ setting, we only get that $\m(A_{+,u})=0$ by Theorem \ref{endpoints measure zero}, that $\m(a_u)=0$ by Proposition \ref{m(a)=0}, as well as  the forward evolution estimates from Theorem \ref{evolution of subsets} and Proposition \ref{TMCP along sets}. In particular, the $\m(A_{-,u})=0$ condition is only required to ensure that we have a well-defined disintegration of $\T_u^{nb}$ into transport rays $(X_\alpha,\ell\restrict_{X_\alpha},\m_\alpha)$ and crucially that $\m_\alpha\restrict_{[a_\alpha,b_\alpha)}\ll \L^1,$ so the proof utilizes the $TMCP^+(K,N)$ condition with the forward $p$-enb assumption only to ensure the existence and uniqueness of forward $\ell^p$ optimal transport maps.

\section{Applications of $TMCP^+(K,N)$ needle decomposition}

In this section, we will show the following applications of localization of $TMCP^+(K,N)$ condition.  We first show that under the assumption of $\m(A_-,u)=0$ for every reverse Lipschitz $u$, the  $TMCP^+(K,N)$ condition is equivalent to the $TMCP^+_{rLip}(K,N)$ condition, for which curvature is defined in terms of gradient flow curves of reverse-Lipschitz $u:Z\rightarrow \R$, where $Z\Subset X$ is a compact timelike interpolation set. Using this equivalence, we provide an alternative interpretation $TMCP^+_{rLip}(K,N)$ condition using Raychaudhuri's ODE. 

We will also consider model forward $p$-enb $TMCP^+(K,N)$ spacetimes, in which the $TMCP^+(K,N)$ inequality is saturated. These model spacetimes correspond to the synthetic analogues of forward Minkowski,  and anti-de Sitter spacetimes. Remarkably, forward rigidity propagates backwards: if the $TMCP^+(K,N)$ bound is saturated, the spacetime is automatically non-branching in both directions--- $\m(A_{-,u})=0$ for every continuous reverse-Lipschitz potential $u.$

The decoupling of $TMCP^+(K,N)$ and $TMCP^-(K,N)$ conditions allows us to define \emph{blended $TMCP$ spacetimes}, where we consider different curvature parameters $K_1\neq K_2$ for forward and backward causal structures. We classify the Bonnet-Myers timelike diameter bounds for such blended spacetimes, and provide an improvement on the synthetic Hawking singularity theorem in this setting. 

Finally, we will provide an equivalence of the corresponding timelike curvature dimension bounds arising from the Lorentzian $L^1$ optimal transport problem. We show that the Localized $TCD^1_{rLip}(K,N)$ condition is equivalent to the Limiting $TCD_1(K,N)$ condition, where limiting is interpreted  as the $TCD_p(K,N)$ condition with $p\rightarrow1^-$.

\subsection{Equivalence of $TMCP^+(K,N)$ conditions}

Throughout this subsection, by a reverse-Lipschitz function, we mean a continuous reverse 1-Lipschitz $u:Z\rightarrow \R$ defined on a \emph{compact timelike interpolation set} $Z\Subset X.$ The following definition is a natural adaption of Theorem \ref{Big TMCP+ theorem}.

\begin{definition}\label{TMCPrLip}{($TMCP^+_{rLip}(K,N)$ spacetimes).}
 Let $(X,d,\m,\ell)$ be a gh LLS. Let  $\Gamma\subset X_\ll^2$ be compact and define the timelike interpolation set $Z\Subset X$. Let $u:Z\rightarrow\R$ be a continuous reverse 1-Lipschitz function, and define $$\Gamma_u:=\{(x,y)\in Z\times Z:u(y)-u(x)=\ell(x,y)>0\}\cup \Delta_Z.$$ We say that the metric spacetime satisfies the $TMCP^+_{u}(K,N)$ condition if there exists a family $\{(X_\alpha,\ell\restrict_{X_\alpha},\m_\alpha) \}_{\alpha \in Q}$ such that:

 \begin{enumerate}
     \item There exists a disintegration $$\m\restrict_{\T_u} = \int_Q \m_\alpha \q(d\alpha),\quad \m_\alpha(X_\alpha)=1, \textrm{for $\q$-a.e. $\alpha\in Q$}.$$
     \item For $\q$-a.e. $\alpha \in Q$, $(X_\alpha,\ell\restrict_{X_\alpha},\m_\alpha) $ is a transport ray.
     \item For $\q$-a.e. $\alpha \in Q$, $(X_\alpha,\ell\restrict_{X_\alpha},\m_\alpha) $ verifies $TMCP^+(K,N)$;  i.e. the density $\m_\alpha\restrict_{[a_\alpha,b_\alpha)]} = h(\alpha,\cdot) d\L^1\restrict_{[a_\alpha,b_\alpha)}$ satisfies \eqref{TMCP+ compact}\begin{equation}\frac{h(\alpha,\tau)}{h(\alpha,s)}\geq \sigma_{K,N-1}^{\big(\frac{\sigma_+-\tau}{\sigma_+-s}\big)}(\sigma_+-s)^{N-1},\end{equation} for all $a_\alpha<s\leq \tau<\sigma_+\leq b_\alpha$, at points of the chosen representative where $h(\alpha,s)>0$ whenever the distortion coefficients are defined.
 \end{enumerate}

 In the particular case $\Gamma = C\times \{ o\}$, where $C\Subset I^-(o)$ and $u(z)=-\ell(z,o),$ we additionally require that $o\in\overline X_\alpha$ for $\q$-a.e. $\alpha$ with $\m_\alpha(C)>0$.

 Finally, we say that the spacetime verifies the $TMCP^+_{rLip}(K,N)$ condition if  $TMCP^+_{u}(K,N)$ holds for every compact $\Gamma\subset X_\ll^2$ defining its timelike interpolation set $Z$, and every  continuous reverse 1-Lipschitz $u:Z\rightarrow \R.$
 \end{definition}

We prove that $TMCP^+_{rLip}(K,N)$ implies $TMCP^+(K,N)$. This proof is a direct Lorentzian adaptation of Proposition 8.9  from \cite{CavallettiMilman2021}. This proof does not require any non-branching assumptions.
\begin{proposition}
    Let $(X,d,\m,\ell)$ be a gh LLS satisfying the $TMCP^{+}_{rLip}(K,N)$ condition with $K\in \R, N\in (1,\infty).$ Then it satisfies $TMCP^+(K,N)$ condition.
\end{proposition}

\begin{proof}

Fix any $o\in X$ and consider any $\mu_0\in \P_c(X)$ such that $\mu_0\ll \m$ and $\spt\mu_0\subset I^-(o).$ Set $$\Gamma:=\spt \mu_0\times \{o\}\subset X_\ll^2,\quad Z:=\cup_{s\in [0,1]}Z_s(\Gamma),$$ where since $\Gamma$ is compact, we have $Z\Subset X$ by global hyperbolicity. Define the following continuous and reverse Lipschitz function $u:Z\rightarrow \R$ as $u(z):=-\ell(z,o).$ From the $TMCP^+_{rLip}(K,N)$ condition, we deduce a disintegration of $\m$ on $\T_u$ along transport rays $\{ X_\alpha\}_{\alpha \in Q}:$ $$\m\restrict_{\T_u} = \int_Q \m_\alpha \q(d\alpha),\quad \m_\alpha(X_\alpha) = 1, \textrm{for $\q$-a.e. $\alpha \in Q,$ }$$ where each $X_\alpha$ is a transport ray for $\Gamma_u$, and $(X_\alpha,\ell\restrict_{X_\alpha},\m_\alpha)$ verifies the $TMCP^+(K,N)$ condition. In particular, we have that $(x,o)\in \Gamma_u$ for every $x\in \spt \mu_0$, and additionally we have $o\in \overline X_\alpha$ for $\q$-a.e. $\alpha$ with $\m_\alpha(\spt \mu_0)>0.$
Now consider any $\mu_0\in\P_c(X)$ such that $\mu_0\ll\m,$ and $\spt \mu_0\subset I^-(o).$ Then the pair $(\mu_0,\delta_o)$ is timelike $p$-dualisable for some $p\in (0,1)$ that we fix. Writing $\mu_0 = \rho_0 d\m$, the measurability of the disintegration gives us that the function $Q\ni \alpha\mapsto z_\alpha:=\int\rho_0(x) d\m_\alpha(dx)$ is $\q$-measurable, and hence the set $\bar{Q}:=\{ \alpha \in Q:z_\alpha\in (0,\infty)\}$ is $\q$-measurable. Moreover, since $\mu_0(\T_u)=1$, we have $\int_Qz_\alpha \q(d\alpha)=\int_{\T_u}\rho_0d\m=1$, which implies  $z_\alpha<\infty $ for $\q$-a.e. $\alpha \in Q$.  We thus have that $$\int_{\bar{Q}} z_\alpha \q(d\alpha) =\int _Q z_\alpha \q(d\alpha)=1.$$

Define $\mu_0^{\alpha}:=\frac{1}{z_{\alpha}}\rho_0\m_\alpha\in \P(X_{\alpha})$ for every $\alpha \in \bar{Q},$ and consider the $\ell^p$-optimal transport problem on $(X_\alpha,\ell\restrict_{X_\alpha},\m_\alpha)$ between $\mu_0^{\alpha}$ and $\delta_o.$ By Definition of $TMCP_{rLip}^+(K,N),$ we have that $o\in \overline X_\alpha$ for $\q$-a.e. $\alpha\in \bar{Q},$ and thus the ray parameterization extends continuously to $o$ with $$g(\alpha,b_\alpha)=0,\quad b_\alpha-s=\ell(g(\alpha,s),o).$$ Consequently, for $x=g(\alpha,s),$ we can define the constant-speed maximizing timelike geodesic $$\g_{\alpha,s}(t):=\begin{cases}
    g(\alpha,(1-t)s+tb_\alpha),&\quad 0\leq t<1,\\
    o,&\quad t=1,
\end{cases}$$
whereby setting $\nu^\alpha:=(s\mapsto\g_{a,s})_\#(g(\alpha,\cdot)^{-1})_\#\mu_0^\alpha$ gives us $\nu^\alpha\in OptTGeo^\ell_p(\mu_0^\alpha,\delta_o).$
Define $\nu: = \int_{\bar{Q}}\nu^{\alpha} z_\alpha\q(d\alpha),$ which satisfies $(e_0)_\#\nu = \mu_0$ and $(e_1)_\#\nu = \delta_o.$ We now show that $\nu \in OptTGeo_p^{\ell}(\mu_0,\delta_o)$ by showing that $t\mapsto (e_t)_\# \nu:=\mu_t$ is in fact a $\ell^p$-geodesic. Indeed, for any $0\leq s\leq t\leq 1$, consider $(e_s,e_t)_\#\nu$ which is a timelike coupling between $\mu_s$ and $\mu_t$, which gives us  

\begin{equation} 
\begin{split}
\ell_p(\mu_s,\mu_t)^p & \geq \int_{\bar{Q}}\int_{X_\alpha \times X_\alpha} \ell^p(x,y)(e_s,e_t)_\#\nu^{\alpha}(dxdy)z_\alpha \q(d\alpha) \\
 & = \int_{\bar{Q}}\int_{X_\alpha\times X_\alpha}(t-s)^p\ell^p(x,y)(e_0,e_1)_\#\nu^\alpha (dxdy)z_\alpha \q (d\alpha)\\
& = (t-s)^p\int_Q\int_{X_\alpha} \ell^p(x,o)\mu_0^{\alpha}(dx)z_\alpha \q (d\alpha)\\
&=(t-s)^p\int_Q\int_{X_\alpha} \ell^p(x,o)\rho_0(x)\m_\alpha(dx)\q(d\alpha) \\ 
& = (t-s)^p\int_X \ell^p(x,o)\rho_0\m(dx) \\ 
& = (t-s)^p \ell_p(\mu_0,\delta_o)^p.
\end{split}
\end{equation}

The reverse triangle inequality gives us the reverse inequality, and therefore $\nu\in OptTGeo^\ell_p(\mu_0,\delta_o).$ The fact that $\mu_t = (e_t)_\#\nu$ verifies the $TMCP^+(K,N)$ condition follows verbatim by following the remainder of the proof of Proposition 8.9 from \cite{CavallettiMilman2021}. 
\end{proof}

We now argue the converse, that forward $p$-enb $TMCP^+(K,N)$ spacetimes that satisfy $\m(A_{-,u})=0$ for every continuous reverse 1-Lipschitz $u:Z\rightarrow \R$ are in fact $TMCP^+_{rLip}(K,N)$. The raywise density estimate was already established in Theorem \ref{Big TMCP+ theorem}, but we provide an alternative proof based on the contraction estimates on the R\'enyi entropy $S_N:\P(X)\rightarrow [-\infty,0].$   We will require this approach in the equivalence of $TCD(K,N)$ conditions in the last subsection.

\begin{proposition}
    Let $(X,d,\m,\ell)$ be a gh LLS forward $p$-enb $TMCP^+(K,N)$  spacetime with $K\in \R, N\in (1,\infty)$ and $p\in(0,1).$  Assume that  $\m(A_{-,u})=0$ for every continuous reverse-Lipschitz $u:Z\rightarrow \R$. Then it satisfies the $TMCP^{+}_{rLip}(K,N)$ condition.
\end{proposition}

\begin{proof}
    Fix a continuous reverse-Lipschitz $u:Z\rightarrow \R$, such that $$Z=\bigcup_{s\in[0,1]}Z_s(\Gamma),\quad \Gamma\subset X_\ll^2.$$ The $TMCP^+(K,N)$ condition together with the forward $p$-enb and the $\m(A_{-,u})=0$ assumptions allows us to invoke the disintegration of $\T_u^{nb}$ from Theorem \ref{a.c}, whereby  $$\m\restrict_{\T_u}=\m\restrict_{\T^{nb}_u}=\int_Q\m_\alpha \q(d\alpha),$$ where for $\q$-a.e. $\alpha$, the transport ray $(X_\alpha,\ell\restrict_{X_\alpha},\m_\alpha)$ satisfy $$\m_\alpha\restrict_{[a_\alpha,b_\alpha)} =h(\alpha,\cdot)d\L^1\restrict_{[a_\alpha,b_\alpha)}.$$  Our goal is to use the $TMCP^+(K,N)$ evolution estimates to show that $h(\alpha,\cdot)$ is in fact a $TMCP^+(K,N)$ density. We first utilize the ray coordinates $g$ to obtain a uniform estimate across transport rays.

Fix a Borel set $\bar{Q}\subset Q$ with $\q(\bar{Q})>0$, and choose $R_0,s>0$ and $\eps>0$ such that $[R_0,R_0+s+\eps]$ lies in the domains of the rays in $\bar{Q}$. Define the initial strip $$B_\eps:=\cup_{\alpha \in \bar{Q}}g(\alpha,[R_0,R_0+\eps])\implies \m(B_\eps)=\int_{\bar{Q}}\int_{R_0}^{R_0+\eps}h(\alpha,r)dr \q(d\alpha).$$ If we now evolve every point in the strip $B_\eps$ forward along their transport rays towards the target $R_0+s,$ then for the initial point $r\in [R_0,R_0+\eps]$, its remaining distance to the target is $R_0+s-r\in [s-\eps,s],$ and therefore the interpolation time $t$ has the coordinate $(1-t)r+t(R_0+s).$ Hence, the evolved strip $(B_\eps)_t$ and its measure are $$(B_\eps)_t = \cup_{\alpha \in \bar{Q}}g(\alpha,[R_0+t_s,R_0+ts+(1-t)\eps])\implies m((B_\eps)_t)=\int_{\bar{Q}}\int_{R_0+ts}^{R_0+ts+(1-t)\eps}h(\alpha,r)dr \q(d\alpha).$$ The set evolution estimate from Proposition \ref{TMCP along sets} gives us that $\m((B_\eps)_t)\geq \inf_{r\in[s-\eps,s]}[\tau_{K,N}^{(1-t)}(r)]^N\m(B_\eps),$ which we can rewrite as \begin{equation}\label{key inequality TMCPrlip} \int_{\bar{Q}}\int_{R_0+ts}^{R_0+ts+(1-t)\eps}h(\alpha,r)dr \q(d\alpha) \geq \inf_{r\in[s-\eps,s]}[\tau_{K,N}^{(1-t)}(r)]^N \int_{\bar{Q}}\int_{R_0}^{R_0+\eps}h(\alpha,r)dr \q(d\alpha).
\end{equation}

Dividing the left hand side by $\eps$ and utilizing the Lebesgue differentiation theorem (since the interval is of length $(1-t)\eps$, we get $$\frac{1}{\eps} \int_{\bar{Q}}\int_{R_0+ts}^{R_0+ts+(1-t)\eps}h(\alpha,r)dr \q(d\alpha)\rightarrow (1-t)\int_{\bar{Q}}h(\alpha,R_0+ts)\q(d\alpha).$$ Similarly, the right hand side becomes $$\frac{1}{\eps} \int_{\bar{Q}}\int_{R_0}^{R_0+\eps}h(\alpha,r)dr \q(d\alpha)\rightarrow \int_{\bar{Q}}h(\alpha,R_0)\q(d\alpha).$$ The continuity of the distortion coefficients gives us that $$\inf_{r\in[s-\eps,s]}[\tau_{K,N}^{(1-t)}(r)]^N \rightarrow \tau_{K,N}^{(1-t)}(s)^N,$$ and consequently, inequality \eqref{key inequality TMCPrlip} gives us $$(1-t)\int_{\bar{Q}}h(\alpha,R_0+ts)\q(d\alpha)\geq \tau_{K,N}^{(1-t)}(s)^N \int_{\bar{Q}}h(\alpha,R_0)\q(d\alpha).$$ Since this holds for every Borel $\bar{Q}\subset Q,$ we get $$(1-t)h(\alpha,R_0+ts)\geq \tau_{K,N}^{(1-t)}(s)^Nh(\alpha,R_0)\quad \textrm{for $\q$-a.e. $\alpha$}.$$ Using the rescaling $\tau_{K,N}^{(1-t)}(s)^N=(1-t)\sigma_{K,N-1}^{(1-t)}(s)^{N-1}$, we get $$h(\alpha,R_0+ts)\geq \sigma_{K,N-1}^{(1-t)}(s)^{N-1}h(\alpha,R_0),$$ which is precisely the one dimensional $TMCP^+(K,N)$ density inequality.

It remains to verify that when $\Gamma = C\times \{ o\}$ and $u(z)=-\ell(z,o)$, we have $o\in \overline X_\alpha$ for $\q$-a.e. $\alpha$ with $\m_\alpha(C)>0$. If $\m(C)=0$, this follows immediately by disintegration. Otherwise, setting $\mu_0:=\frac{\m\restrict_C}{\m(C)}$, we see that by Theorem \ref{Brenier Map}, there exists an $\ell^p$-optimal dynamical plan $\eta$ from $\mu_0$ to $\delta_0$ with $(e_t)_\# \eta\ll \m$ for every $t<1$. Since $\m(A_{+,u}\cup A_{-,u})=0$, we have that for $\eta$-a.e. geodesic $\g$, all rational time points $\g_t$ with $t\in [0,1)\cap \Q$ belong to $\T_u^{nb}.$ Since $(\g_s,\g_t)\in\Gu$ whenever $s<t<1$, these points belong to some transport ray $X_\alpha$. Sending $t\uparrow1$, we get $o\in \overline X_\alpha$, and the conclusion follows by disintegrating $\mu_0$ along the transport rays.
\end{proof}

We summarize the result of this subsection in the following Theorem.

\begin{theorem}
        Let $(X,d,\m,\ell)$ be a gh LLS satisfying the $TMCP^{+}_{rLip}(K,N)$ condition with $K\in \R, N\in (1,\infty).$ Then it satisfies $TMCP^+(K,N)$ condition, without any non-branching assumptions.
        
Conversely, assume that the spacetime satisfies the forward $p$-enb $TMCP^+(K,N)$ condition with $K\in \R, N\in (1,\infty)$ and $p\in(0,1).$  Assume that  $\m(A_{-,u})=0$ for every admissible continuous reverse-Lipschitz $u:Z\rightarrow \R$. Then it satisfies $TMCP^{+}_{rLip}(K,N)$ condition.

By causal reversal of the hypotheses, we get the equivalence of $TMCP^-(K,N)$ and $TMCP^-_{rLip}(K,N).$
\end{theorem}

\subsection{Localization Theorem via Raychaudhuri ODE}
We will give a natural interpretation of the $TMCP^+(K,N)$ condition in terms of Raychaudhuri's equation along gradient flow curves associated to reverse 1-Lipschitz functions. We first recall the derivation of Raychaudhuri's equation. The formulation below follows the presentation from Wald \cite{Wald1984}.

Let $(M,g)$ be a smooth Lorentzian manifold of dimension $n=N,$ fix a timelike geodesic $\gamma(t)$ parameterized by proper time $t$, and let $\dot{\gamma}$ be the unit timelike tangent vector. We now consider a timelike geodesic congruence around $\gamma$, that is, consider a one-parameter family of geodesics $\gamma_s(t)$ , where $s =(s^1,\dots s^{N-1})$ with $\gamma_0=\gamma.$ The variation vector field $J_{(j)}(t)=\frac{\partial}{\partial s^j}\gamma_s(t)\big|_{s=0}$  with $j=1,\dots N-1$, are Jacobi fields along $\gamma,$ and satisfies the Jacobi equation $$\nabla_{\dot{\gamma}}\nabla_{\dot{\gamma}} J_{(j)}+R(J_{(j)},\dot{\gamma})\dot{\gamma}=0$$ where $R$ is the Riemann curvature tensor. The components of the Jacobi fields parallel to $\dot{\gamma}$ do not describe transverse geodesic deviation, so we focus on the remaining $N-1$ linearly independent normal directions, i.e. those that satisfy $g(J_{(j)},\dot{\gamma})=0$. The normal bundle $\dot \gamma^\perp$ has dimension $N-1$, so we can choose an orthonormal basis  $\{ E_1(t),\dots,E_{N-1}(t)\}$ of spacelike vectors orthogonal along $\dot{\gamma},$ i.e.: $$g(E_i,E_j)=\delta_{ij},\quad g(E_i,\dot{\gamma})=0,\quad \nabla_{\dot{\gamma}}E_i = 0.$$ Now, denoting by $J_{(1)},\dots J_{(N-1)}$ the $N-1$ linearly independent normal Jacobi fields, we can express each one in the parallel basis above as $$J_{(j)}(t) = \sum_{i=1}^{N-1}A^i_j(t)E_i(t),$$ where the $(N-1)\times (N-1)$ matrix $A(t)$ satisfies the matrix Jacobi equation $$\ddot{A}(t)+\mathcal{R}(t)A(t)=0,$$ where $\mathcal{R}(t)$ is known as the \emph{tidal tensor} $\mathcal{R}_{ij}(t)=g(R(E_i,\dot{\gamma})\dot{\gamma},E_j)$, symmetric by the symmetries of the Riemann tensor.

The Jacobi fields $J_{(j)}$ describe how an infinitesimal spacelike volume evolves along $\gamma.$ The transverse volume density is defined as $h(t):=|\det(A(t))|$ on an interval where $A(t)$ is nonsingular, where geometrically, the evolution of  an infinitesimal spacelike parallelepiped at $t_0$ spanned by $J_{(1)}(t_0),\dots J_{(N-1)}(t_0)$ with volume $V_0$ will have volume at time $t$ given by $V(t) =V_0 \frac{h(t)}{h(t_0)}.$ Defining the \emph{expansion tensor} $B(t)$ as  $$B(t) = \dot{A}(t)A(t)^{-1},$$ and the \emph{expansion scalar} $\theta(t)$ as its trace $$\theta(t) = \tr(B(t))=\tr(\dot{A}(t)A(t)^{-1}),$$ we  use the Jacobi determinant formula  $\frac{d}{dt}\det(A(t)) = \det(A(t))\tr(\dot{A(t)}A(t)^{-1})$ to obtain that $$\dot{h}(t)=h(t)\theta(t)\implies \theta(t)=(\log h(t))' = \frac{\dot{h}(t)}{h(t)}.$$

Decomposing the expansion tensor $B(t)$  into its trace, symmetric, and anti-symmetric parts as $$B=\frac{\theta}{N-1}I +\sigma+\omega,$$ with $$\sigma = \frac{1}{2}(B+B^T)-\frac{\theta}{N-1}I,\quad \tr(\sigma)=0,\quad \sigma^T=\sigma,$$ $$\omega =\frac{1}{2}(B-B^T),\quad \omega^T=-\omega$$ and these quantities are known respectively as the shear $\sigma$ and vorticity $\omega$.

The evolution of the expansion tensor $B(t)=\dot{A}A^{-1}$ obeys the Riccati equation: $$\dot{B} = \ddot{A}A^{-1}-\dot{A}A^{-1}\dot{A}A^{-1} = \ddot{A}A^{-1}-B^2 =-\mathcal{R}-B^2, $$ where we have used the Jacobi equation $\ddot{A}=-\mathcal{R}A.$ Now taking the trace throughout the equation, we get $$\tr(\dot{B})+\tr(B^2) = -\tr(\mathcal{R}),$$ where we have that $\tr{\dot{B}} = \dot{\theta}$, $\tr(\mathcal{R}) = \Ric(\dot{\gamma},\dot{\gamma})$, and the decomposition of $B^2$ into its trace, shear and vorticity parts after expansion of the trace $$\tr(B^2) = \frac{\theta^2}{N-1}+\tr(\sigma^2)+\tr(\omega^2) +\tr(\textrm{cross terms}),$$ where the trace of the cross terms of the type $\tr(\sigma\omega)=0$ vanish since $\tr(\sigma)=\tr(\omega)=0$ with $\sigma$  symmetric and $\omega$ asymmetric. Furthermore, $\tr(\sigma^2) = \| \sigma\|^2\geq0, \tr(\omega^2)=-\|\omega\|^2\leq 0$, and we get the full Raychaudhuri equation \begin{equation}\label{Raychaudhri}
    \dot{\theta} +\frac{\theta^2}{N-1}=-\Ric(\dot{\gamma},\dot{\gamma})-\| \sigma\|^2+\|\omega \|^2.
\end{equation}

We now make the key assumption that the geodesic congruence is vorticity free: \begin{equation} \omega = 0,\end{equation}
which holds whenever the geodesic congruence is  hypersurface orthogonal. In the smooth transport setting relevant here, the geodesic congruence is locally generated by a gradient of a sufficiently regular potential, and is therefore hypersurface orthogonal, and hence its vorticity vanishes. Since the shear $\| \sigma\|^2$ is manifestly non-negative, the Raychaudhuri under the Ricci curvature bound $\Ric(\dot{\gamma},\dot{\gamma})\geq K$ reduces to $$\dot{\theta} +\frac{\theta^2}{N-1}\leq -K.$$

The Model ODE with equality is exactly the Riccati ODE $$\dot{\theta}_K+ \frac{\theta_K^2}{N-1}=-K,$$ whioch admits the solution $\theta_K(t)=\log'(s_K(t/\sqrt{N-1})^{N-1}).$ Recalling that our distortion profile functions $s_K$ from \eqref{profile function} satisfy $s''_K+Ks_K=0$, for a congruence focusing to a future point $b$ so that $A(b)=0$ with $A(t)$ nonsingular for $t<b$. The corresponding future model solution is therefore$$\theta_K^+:=\frac{d}{dt}\log\bigg[s_K\bigg(\frac{b-t}{\sqrt{N-1}}\bigg)^{N-1}\bigg] = -\sqrt{N-1}\frac{\dot{s}_K\big(\frac{b-t}{\sqrt{N-1}}\big)}{s_K\big(\frac{b-t}{\sqrt{N-1}}\big)}.$$ Both $\theta$ and $\theta_K^+$ diverge to $-\infty$ at the future endpoint $b.$ Under the above focusing assumptions, the Riccati comparison principle then yields $$\theta(t)\geq \theta^+_K(t),$$ and since $\theta = (\log h)'$ this is exactly $$(\log h(t))'\geq \frac{d}{dt}\log\bigg[s_K\bigg(\frac{b-t}{\sqrt{N-1}}\bigg)^{N-1}\bigg].$$ Integrating from $s$ to $\tau$ gives us exactly $$\frac{h(\tau)}{h(s)}\geq \bigg[\frac{s_K\big(\frac{b-\tau}{\sqrt{N-1}}\big)}{s_K\big(\frac{b-s}{\sqrt{N-1}}\big)}\bigg]^{N-1}=\sigma_{K,N-1}^{\big(\frac{b-\tau}{b-s}\big)}(b-s)^{N-1}$$ which is precisely the localized $TMCP^+(K,N)$ density bound from \ref{TMCP+ compact}. Analogously, the $TMCP^-(K,N)$ density bound follows by applying the same argument after reversing the time orientation. With the Dirac at the past endpoint $a,$  the Model solution being $$\theta_K^-:=\log'(s_K(t-a)\sqrt{K/(N-1)})^{N-1}),$$ which diverges to $+\infty$ at $a,$ and the Riccati comparison principle gives us $\theta\leq \theta^-_K,$ which is precisely \eqref{TMCP- compact}

The preceding smooth computation was carried out for the integer dimension $N=\dim(M).$ Since the resulting one-dimensional density inequality can be applied for every synthetic dimension parameter $N\in (1,\infty),$ and recalling the equivalence of $TMCP^+_{rLip}(K,N)$ with $TMCP^+(K,N)$ from the previous subsection, we obtain the following Raychaudhuri characterization along almost every transport ray $(X_\alpha,\ell\restrict_{X_\alpha},\m_\alpha).$ where volume density $h(\alpha,\cdot)$ satisfies the logarithmic derivative bound in the distributional sense: $$(\log h(\alpha,t))'\geq -\sqrt{N-1}\frac{\dot{s}_K\big(\frac{b_\alpha-t}{\sqrt{N-1}}\big)}{s_K\big(\frac{b_\alpha-t}{\sqrt{N-1}}\big)}.$$

\begin{subsection}{Model $TMCP^+(K,N)$ spacetimes}

We now turn to model spacetimes. We call a spacetime a \emph{forward Model $TMCP^+(K,N)$ spacetime} if the entropy inequality \eqref{Renyi Entropy Inequality} defining $TMCP^+(K,N)$ is attained as an equality for $N'=N$,  for every admissible contraction $\mu_t$ of a compactly supported absolutely continuous measure $\mu_0$ to a future Dirac point $\delta_x$, at every time $t\in[0,1)$. These correspond to the synthetic analogues of spacetimes whose forward behaviour is Minkowski $(K=0),$ scaled de Sitter $(K<0)$ and scaled anti-de Sitter $(K>0)$, \cite[Remark 5.11]{CavallettiMondino2024}.  

Throughout this subsection, to relate the global and localized $TMCP^+(K,N)$ conditions, we assume both forward $p$-enb and that $\m(A_{-,u})=0$ for every admissible continuous reverse-Lipschitz $u:Z\rightarrow \R$. This in particular implies that $\m(A_{+,u})=\m(A_{-,u})=0$, and from Theorem \ref{endpoints measure zero}, the transport relation $R_u$ is an equivalence relation on a set of full measure. The disintegration into transport rays is then well defined for every continuous reverse-Lipschitz $u:Z\rightarrow \R$. Thus, the $L^1$ transport relation associated to every admissible reverse-Lipschitz potential is nonbranching in both forward and backward in time outside an $\m$-null set. 
We can thus apply the equivalence of Model $TMCP^+(K,N)$ and $TMCP^+_{rLip}(K,N)$ condition at the level of localization.

To study the model cases, we additionally assume that the localized forward $TMCP^+(K,N)$ comparison is attained as an equality along the transport rays $(X_\alpha, \ell\restrict_{X_\alpha},\m_\alpha)$ with respect to their finite future endpoints $b_\alpha$, namely, for $\q$-a.e. $\alpha$, $$\frac{h(\alpha,\tau)}{h(\alpha,s)}=\bigg[\frac{s_K\big(\frac{b_\alpha-\tau}{\sqrt{N-1}}\big)}{s_K\big(\frac{b_\alpha-\tau}{\sqrt{N-1}}\big)}\bigg]^{N-1},\quad a_\alpha<s<\tau<b_\alpha.$$ Equivalently, we can write this as 
$$h(\alpha,t) = C_\alpha\cdot s_K\bigg(\frac{b_\alpha - t}{\sqrt{N-1}}\bigg)^{N-1},$$  where $C_\alpha>0$ is a normalization constant.

The Bonnet-Myers type consequences depend on the sign of $K.$

\begin{enumerate}
    \item \textbf{Case $K>0$}. Set $c= \sqrt{\frac{K}{N-1}},$ and write $$h(\alpha,t)=C_\alpha\sin(c(b_\alpha-t))^{N-1}.$$ Since the density is positive in the interior of the transport ray, the sine function can not pass through its first zero before reaching the past endpoint. Consequently, $0<c(b_\alpha-a_\alpha)\leq \pi$,  and  we get the corresponding Bonnet-Myers type bound $$L_\alpha:=b_\alpha - a_\alpha \leq \pi \sqrt{\frac{N-1}{K}},$$ and if equality holds, the $a_\alpha$ lies exactly at the next zero of the model density, and therefore$$\lim_{t\downarrow a_\alpha}h_\alpha(t)=0.$$  The saturation of the Bonnet-Myers bound is precisely the two-endpoint focusing case. In the smooth Jacobi-field interpretation, if  $h_\alpha(t)=|\det(A(t)|$ arises from a congruence focused at $b_\alpha$, then the vanishing of the determinant at $a_\alpha$  means that the two endpoints are conjugate along the corresponding timelike geodesic.

We now examine the additional constraints imposed by the backward  $TMCP^-(K,N)$ condition. For every $a_\alpha<s<\tau<b_\alpha$, the forward Model $TMCP^+(K,N)$ equality and the backward $TMCP^-(K,N)$ inequality give:
$$\bigg[\frac{\sin(c(b_\alpha-\tau))}{\sin(c(b_\alpha-s))}\bigg]^{N-1}=  \frac{h(\alpha,\tau)}{h(\alpha,s)}\leq\bigg[\frac{\sin(c(\tau-a_\alpha))}{\sin(c(s-a_\alpha))}\bigg]^{N-1}.$$ 
Since all of the sine terms are positive in the interior of rays, these conditions are compatible precisely when
$$\sin(c(b_\alpha-s))\sin(c(\tau-a_\alpha))-\sin(c(b_\alpha-\tau))\sin(c(s-a_\alpha))=\sin(cL_\alpha)\sin(c(\tau-s))\geq 0.$$ 
The forward Bonnet-Myers bounds already gives $0<L_\alpha\leq \frac{\pi}{c},$ and since $0<\tau-s<L_\alpha$, we see that both factors in the last equation are nonnegative. 
Thus, the non-strict backward $TMCP^-(K,N)$ condition is  compatible with every forward Model ray, and gives no sharper length bound. However, if we now consider the backward $TMCP^-(K,N)$ as a strict inequality, then saturation of the Bonnet-Myers bounds becomes impossible. More precisely if the backward inequality is  strict for some, equivalently every, $a_\alpha<s<\tau<b_\alpha$, then  
$$\sin(cL_\alpha)\sin(c(\tau-s))>0,$$ 
and we must have that $\sin(cL_\alpha)>0$. The inequality $0<cL_\alpha\leq \pi$ thus forces $cL_\alpha<\pi,$ and hence $$L_\alpha <\pi \sqrt\frac{N-1}{K}.$$ 
Moreover, the density does not vanish at the past endpoint since $$\lim_{t\downarrow a_\alpha}h(\alpha,t)=C_\alpha\sin(cL_\alpha)^{N-1}>0,$$ and in the smooth Jacobi-field interpretation, the corresponding congruence does not focus at $a_\alpha$, and therefore $a_\alpha$ is not conjugate to $b_\alpha$ throughout this Jacobi family. 

Conversely, suppose that the backward $TMCP^-(K,N)$ holds with equality for some $a_\alpha<s<\tau<b_\alpha$. Then   $$\sin(cL_\alpha)\sin(c(\tau-s))=0\implies \sin(cL_\alpha)=0,$$ and together with the Bonnet-Myers bound, this forces $$L_\alpha = \pi\sqrt\frac{N-1}{K}.$$ The density is then symmetric with respect to its two endpoints $$h(\alpha,t)=C_\alpha \sin(c(b_\alpha-t))^{N-1}=C_\alpha\sin(c(t-a_\alpha))^{N-1},$$ which vanishes at both endpoints. Thus, simultaneous forward and backward model equality occurs precisely at the maximal Bonnet-Myers length. In the smooth Jacobi field interpretation, the congruence focuses at both endpoints, and the two  endpoints form a conjugate pair. This is the complete  anti-de Sitter model along transport rays.

We can strengthen this conclusion by removing the forward Model assumption in the maximal-length case. Indeed, suppose that the transport ray satisfies both the forward and backward inequalities, and that in addition $$L_\alpha = \pi \sqrt{\frac{N-1}{K}}.$$ Then, for every $a_\alpha<s<\tau<b_\alpha$, the two inequalities give 

$$\bigg[\frac{\sin(c(b_\alpha-\tau))}{\sin(c(b_\alpha-s))}\bigg]^{N-1}\leq  \frac{h(\alpha,\tau)}{h(\alpha,s)}\leq\bigg[\frac{\sin(c(\tau-a_\alpha))}{\sin(c(s-a_\alpha))}\bigg]^{N-1}.$$ Since $cL_\alpha = \pi$, we have $\sin(c(b_\alpha-t))=\sin(c(t-a_\alpha))$. Thus, the lower and upper bounds therefore coincide, giving $$ \frac{h(\alpha,\tau)}{h(\alpha,s)}=\bigg[\frac{\sin(c(\tau-a_\alpha))}{\sin(c(s-a_\alpha))}\bigg]^{N-1}\implies h(\alpha,t)=C_\alpha\sin(c(t-a_\alpha))^{N-1}.$$ We have thus  the following maximal-length rigidity along rays: for $K>0$, that $TMCP^+(K,N)$ and $TMCP^-(K,N)$ conditions, together with $L_\alpha = \pi \sqrt{\frac{N-1}{K}}$ implies that $$h(\alpha,t)=C_\alpha\sin\bigg(\sqrt{\frac{K}{N-1}} (t-a_\alpha)\bigg)^{N-1},$$ which shows that the saturation of the Bonnet-Myers bound forces simultaneous forward and backward model equality along the ray, even when neither equality is assumed beforehand. This is the Lorentzian counterpart of the maximal-length rigidity for the one-dimensional $MCP(K,N)$ densities of Cavalletti-Mondino \cite[Remark 2.14]{CavallettiMondino2020}.

\item \textbf{Case $K=0$}. The model profile function is $s_0(r)=r$, and hence the forward Model density is 
$$h(\alpha,t)=C_\alpha(b_\alpha-t)^{N-1}.$$ 

Unlike the $K>0$ case, there is no Bonnet-Myers type upper bound on the length of the transport ray.

We now examine the additional constraints imposed by the backward $TMCP^-(0,N)$ condition. For every $a_\alpha<s<\tau<b_\alpha$, the forward Model $TMCP^+(0,N)$ equality and the backward $TMCP^-(0,N)$ inequality give: $$\bigg[\frac{b_\alpha-\tau}{b_\alpha-s}\bigg]^{N-1} = \frac{h(\alpha,\tau)}{h(\alpha,s)}\leq\bigg[\frac{\tau - a_\alpha}{s-a_\alpha}\bigg]^{N-1}$$Since all the factors are positive, these conditions are compatible precisely when $$(b_\alpha-s)(\tau-a_\alpha)-(b_\alpha-\tau)(s-a_\alpha)=L_\alpha(\tau-s)>0.$$ Thus, the backward $TMCP^-(0,N)$ comparison is automatically \textit{strict for every pair of distinct interior points} of a finite nontrivial forward Model ray. In particular, backward $TMCP^-(0,N)$ imposes no additional restrictions on the ray length. Moreover, the forward model density remains positive at the past endpoint: $$\lim_{t\downarrow a_\alpha} h(\alpha,t)=C_\alpha L_\alpha^{N-1}>0.$$ In the smooth Jacobi-field interpretation, a congruence focused at the future endpoint $b_\alpha$ therefore does not focus at $a_\alpha$.

We can quantify the strictness of the backward comparison: first define the forward and backward model expansions respectively by $$\theta_+(t):=-\frac{N-1}{b_\alpha-t},\quad \theta_-(t)=\frac{M-1}{t-a_\alpha}.$$ Their difference satisfies $$\theta_-(t)-\theta_+(t)=(N-1)\bigg(\frac{1}{t-a_\alpha}+\frac{1}{b_\alpha-t}\bigg)=\frac{(N-1)L_\alpha}{(t-a_\alpha)(b_\alpha-t)}>0.$$ Since $(t-a_\alpha)(b_\alpha-t)\leq \frac{L_\alpha^2}{4}$, we obtain the quantitative separation $$\theta_-(t)-\theta_+(t)\geq \frac{4(N-1)}{L_\alpha}.$$ Equality in this last estimate occurs at the midpoint of the ray. Thus, unlike in the $K>0$ case, the two model expansions remain strictly separated on every finite nontrivial transport ray. In particular, simultaneous forward and backward Model equality is impossible on any finite nontrivial ray.

We can strengthen the comparison by removing the forward Model assumption. Indeed, suppose that a transport ray satisfies both the inequalities from  $TMCP^+(0,N)$ and $TMCP^-(0,N)$ conditions. Then

$$\bigg[\frac{b_\alpha-\tau}{b_\alpha-s}\bigg]^{N-1} = \frac{h(\alpha,\tau)}{h(\alpha,s)}=\bigg[\frac{\tau - a_\alpha}{s-a_\alpha}\bigg]^{N-1},$$ and taking their logarithmic derivatives, we obtain for almost every $t\in (a_\alpha,b_\alpha)$ that $$-\frac{N-1}{b_\alpha-t}\leq \frac{d}{dt}\log h(\alpha,t)\leq \frac{N-1}{t-a_\alpha}.$$ These inequalities give a two-sided bound on the expansion of the conditional measure, without requiring either model equality, and in particular, these two inequalities do not impose an upper bound on $L_\alpha$.

Finally, we consider the limiting model profiles when one or both endpoints are finite. These cases lie outside of the compact admissible $rLip$ interpolation setting, but arise naturally if one has global timelike rays.

If $b_\alpha =+\infty$, the forward Model coefficient is obtained by taking the limit $$\lim_{b\rightarrow \infty}\bigg(\frac{b-\tau}{b-s}\bigg)^{N-1}=1,$$ and thus the forward Model equality along such a ray  becomes $\frac{h(\alpha,\tau)}{h(\alpha,s)}=1,$ and therefore $$h(\alpha,t)=C_\alpha.$$ Similarly, if $a_\alpha = -\infty$, backward  Model equality forces to be $h(\alpha,t)=C_\alpha$.  We therefore obtain a classification of the simultaneous forward and backward model equalities in the $K=0$ case. On a finite nontrivial ray, simultaneous equalities are impossible. The same incompatibility persists when exactly one endpoint is finite, since the model profile with a finite endpoint is nonconstant, whereas the model corresponding to one infinite endpoint is constant.  The only nontrivial possibility is consequently $$a_\alpha =-\infty,\quad b_\alpha = +\infty,\quad h(\alpha,t)=C_\alpha.$$ This recovers the one-dimensional Minkowski-type model associated with a complete timelike line and a constant density. For such a bi-infinity ray, the measure $C_\alpha dt$ is necessarily infinite. Thus, this last classification is about a global, potentially $\sigma$-finite ray disintegration, rather than the normalized conditional probabilities arising from the compact $rLip$ localization.

\item \textbf{Case $K<0$}. Set $c= \sqrt{\frac{-K}{N-1}},$ and write the forward Model density as $$h(\alpha,t)=C_\alpha\sinh(c(b_\alpha-t))^{N-1}.$$  Since the hyperbolic sine is strictly positive on $(0,\infty)$ and has no further zeros, there is no Bonnet-Myers type upper bound on the length $L_\alpha=b_\alpha-a_\alpha$ of the transport ray.

We now examine the additional constraints imposed by the backward $TMCP^-(K,N)$ condition. For every $a_\alpha<s<\tau<b_\alpha$, the forward Model equality and backward $TMCP^-(K,N)$ inequality give $$\bigg[\frac{\sinh(c(b_\alpha-\tau))}{\sinh(c(b_\alpha-s))}\bigg]^{N-1}=  \frac{h(\alpha,\tau)}{h(\alpha,s)}\leq\bigg[\frac{\sinh(c(\tau-a_\alpha))}{\sinh(c(s-a_\alpha))}\bigg]^{N-1}.$$

Since all the hyperbolic sine terms are positive in the interior of the ray, the above equations are compatible precisely when $$\sinh(c(b_\alpha-s))\sinh(c(\tau-a_\alpha))-\sinh(c(b_\alpha-\tau))\sinh(c(s-a_\alpha))=\sinh(cL_\alpha)\sinh(c(\tau-s))>0.$$ Thus, the backward $TMCP^-(K,N)$ comparison is automatically \textit{strict for every pair of distinct interior points} of a finite nontrivial forward Model ray. In particular, the backward $TMCP^-(K,N)$ condition imposes no additional restrictions on the ray length. Moreover, the forward model density remains positive at the past endpoint: $$\lim_{t \downarrow a_\alpha}h(\alpha,t)=C_\alpha\sinh(cL_\alpha)^{N-1}>0.$$ In the smooth Jacobi-field interpretation, a congruence that is focused at the future endpoint $b_\alpha$ does not focus at $a_\alpha$. We can quantify the strictness of the backward comparison. Define the forward and backward model expansions respectively $$\theta_+(t)=-(N-1)c \coth(c(b_\alpha-t)),\quad \theta_-(t)=(N-1)c \coth (c(t-a_\alpha)).$$ Their difference satisfies $$\theta_-(t)-\theta_+(t)=(N-1)c\frac{\sinh(cL_\alpha)}{\sinh(c(t-a_\alpha))\sinh(c(b_\alpha-t))}.$$ Since the denominator is bounded above by $\sinh\bigg(\frac{cL_\alpha}{2}\bigg)^2$, we obtain the quantitative separation $$\theta_-(t)-\theta_+(t)\geq 2(N-1)c \coth\bigg( \frac{cL_\alpha}{2}\bigg)>2(N-1)c.$$ Equality in the first inequality once again occurs at the midpoint of the ray. Thus, unlike the positive curvature case, the forward and backward model expansions remain uniformly separated, independently of how large the ray length becomes. In particular, simultaneous forward and backward Model equality is impossible on any finite nontrivial ray.

We can strengthen the comparison by removing the forward Model assumption. Indeed, suppose that a transport ray satisfies both the inequalities from the $TMCP^+(K,N)$ and the $TMCP^-(K,N)$ conditions. Then $$\bigg[\frac{\sinh(c(b_\alpha-\tau))}{\sinh(c(b_\alpha-s))}\bigg]^{N-1}\leq \frac{h(\alpha,\tau)}{h(\alpha,s)}\leq\bigg[\frac{\sinh(c(\tau-a_\alpha))}{\sinh(c(s-a_\alpha))}\bigg]^{N-1},$$ and taking the logarithmic derivatives, we obtain at almost every $t\in (a_\alpha,b_\alpha)$: $$-(N-1)c\coth(c(b_\alpha-t))\leq \frac{d}{dt}\log h(\alpha,t)\leq (N-1)c\coth (c(t-a_\alpha)).$$ These inequalities give two sided control on the expansion of the conditional measures without requiring either model equality. As both bounds remain compatible for arbitrary finite $L_\alpha>0$, they do not impose a Bonnet-Myers type upper bound.

Finally, we consider the limiting profiles when one or both endpoints are infinite. As for $K=0$, these cases lie outside of the compact admissible $rLip$ interpolation setting, but arise naturally if one has global timelike rays.

If $b_\alpha =+\infty$, the forward Model coefficient is obtained from the limit $$\lim_{b\rightarrow +\infty}\frac{\sinh(c(b-\tau))}{\sinh(c(b-s))}=e^{-c(\tau-s)}.$$ Hence, forward Model equality becomes $$ \frac{h(\alpha,\tau)}{h(\alpha,s)}=e^{-(N-1)c(\tau-s)}\implies h(\alpha,\tau)=C_\alpha e^{-(N-1)ct}.$$ Similarly, sending the past endpoint $a_\alpha \rightarrow -\infty$, the backward Model equality gives $$h(\alpha,t)=C_\alpha e^{(N-1)ct}.$$

We can now classify the possibilities for simultaneous forward and backward Model equality when the endpoints are finite or infinite --- every forward $K<0$ model density is strictly decreasing in $t$, whether its future endpoint is finite or infinite, and conversely, every backward model density is strictly increasing in $t$, whether its past endpoint is finite or infinite. Consequently, no positive density can simultaneously satisfy both forward and backward Model equalities on a nontrivial interval, regardless of the endpoint configuration. 

For a bi-infinite ray, however, the ordinary $TMCP^+(K,N)$ and $TMCP^+(K,N)$ inequalities remain compatible. Taking the limits $b_\alpha\rightarrow+\infty$ and $a_\alpha \rightarrow-\infty$ in the corresponding comparisons gives $$  e^{-(N-1)c(\tau-s)}\leq \frac{h(\alpha,\tau)}{h(\alpha,s)}\leq e^{(N-1)c(\tau-s)},$$ which can be equivalently written via the logarithmic expansion, for almost every $t\in \R$, as: $$ -(N-1)c\leq\frac{d}{dt}\log h(\alpha,t)\leq (N-1)c.$$ Thus, in contrast with the $K>0$ case, the two-sided $TMCP(K,N)$ condition does not give a diameter bound, but provides a uniform bound on the exponential rate of the volume expansion or contraction. The two extremal Model solutions are $$h_+(t)=Ce^{-(N-1)ct},\quad h_-(t)=Ce^{N-1}ct,$$ with the natural intermediate example being the de Sitter-type density $$h_{dS}(t):=C\cosh(ct)^{N-1}.$$ Indeed, we have that $\frac{d}{dt}\log h_{dS}(t)=(N-1)c\tanh (ct)$, and therefore, for every finite $t$, we have $$ -(N-1)c<\frac{d}{dt}\log h_{dS}(t)< (N-1)c.$$ Consequently, the de Sitter-type profile density satisfies both forward and backward $TMCP(K,N)$ comparisons strictly on every nontrivial finite interval, while also approaching the two external model expansions asymptotically: $$\lim_{t\rightarrow -\infty}\frac{d}{dt}\log h_{dS}(t)=-(N-1)c,\quad \lim_{t\rightarrow +\infty}\log h_{dS}(t)=(N-1)c.$$

We therefore obtain a complete classification of simultaneous forward and backward model equalities in the $K<0$ case: equalities are incompatible on every nontrivial transport ray, including rays with infinite endpoints. Nevertheless, the ordinary two-sided $TMCP(K,N)$ inequalities admit nontrivial densities whose expansion interpolates between the two extreme exponential rates. As in the $K=0$ case, bi-infinite model profiles describe a global, potentially $\sigma$-finite ray disintegration, rather than the normalized conditional probabilities arising from the compact $rLip$ localization.

\end{enumerate}
\end{subsection}
\begin{subsection}{Blended $TMCP$ conditions}
  
THe disentanglement of the forward and backward $TMCP$ condition in the preceding sections allows us to impose different curvature-dimension bounds in the two timelike directions. This motivates the following blended $TMCP$ condition, where the forward and backward conditions hold simultaneously with different curvature and dimension parameters. \begin{definition}
 Given $K_1,K_2\in \R$ and $N_1,N_2\in(1,\infty),$  assume that the spacetime is both forward and backwards $p$-enb. We say that the spacetime satisfies the \emph{blended $TMCP$ condition}  provided it satisfies both $TMCP^+(K_1,N_1)$ and $TMCP^-(K_2,N_2)$. 
 
 \end{definition}
 
 For such spacetimes, we consider an admissible continuous reverse-Lipschitz function $u:Z\rightarrow \R$ where by Theorem \ref{endpoints measure zero}, we get a disintegration of $\m\restrict_{\T_u^{nb}}$. Theorem \ref{Big TMCP+ theorem} gives us that $\q$-almost every transport ray $(X_\alpha,\ell\restrict_{X_\alpha},\m_\alpha)$ satisfies the blended $TMCP$ bound \begin{equation}\label{blended TMCP bound}\bigg[\frac{s_{K_1}\big(\frac{b_\alpha -\tau}{\sqrt{N_1-1}}\big)}{s_{K_1}\big(\frac{b_\alpha -s}{\sqrt{N_1-1}}\big)}\bigg] ^{N_1-1}\leq\frac{h(\alpha,\tau)}{h(\alpha,s)}\leq \bigg[\frac{s_{K_2}\big(\frac{\tau - a_\alpha }{\sqrt{N_2-1}}\big)}{s_{K_2}\big(\frac{s-a_\alpha }{\sqrt{N_2-1}}\big)}\bigg] ^{N_2-1},\end{equation}
for all $a_\alpha<s< \tau<b_\alpha$ and the density $h(\alpha,\cdot)$ is locally Lipschitz and strictly positive in the interior of the ray.

The  Bonnet-Myers type bounds depend on the signs of $K_1$ and $K_2$, and unlike the preceding model spacetime classification, we do not assume that either equality in the blended $TMCP$ bound \eqref{blended TMCP bound} is saturated. An important feature of blended $TMCP$ spacetimes is that the two directional estimates constrain the same density, and so they must be mutually compatible. To express this compatibility, we set $$\m_i:=N_i-1,\quad S_i(t):=s_{K_i}\big(\frac{t}{\sqrt m_i}\big), \quad L_\alpha:=b_\alpha-a_\alpha,$$ so that for finite endpoints, we obtain by differentiating the logarithmic form of \eqref{blended TMCP bound} $$-m_1\frac{S_1'(L_\alpha-t)}{S_1(L_\alpha-t)}\leq \frac{d}{dt}\log h(\alpha,a_\alpha+t)\leq m_2\frac{S_2'(t)}{S_2(t)}$$ for almost every $t\in (0,L_\alpha)$. Consequently, we get the following compatibility condition \begin{equation}\label{compatibility bound} m_1\frac{S_1'(L_\alpha-t)}{S_1(L_\alpha-t)}+m_2\frac{S_2'(t)}{S_2(t)}\geq 0,\quad 0<t<L_\alpha.\end{equation} The resulting Bonnet-Myers type bounds now depend on the signs of $K_1$ and $K_2$, which we classify below.

\begin{enumerate}
  \item \textbf{Case $K_1,K_2>0$.} Set $c_i=\sqrt{\frac{K_i}{N_i-1}}$ for $i=1,2$. In this case, each directional comparison gives a Bonnet-Myers bound. Therefore,$$L_\alpha = b_\alpha -a_\alpha \leq \min\bigg\{ \pi\sqrt{\frac{N_1-1}{K_1}},\pi\sqrt{\frac{N_2-1}{K_2}}\bigg\}.$$ Equivalently, 
  $$L_\alpha\cdot\max\bigg\{ \sqrt{\frac{K_1}{N_1-1}},\sqrt{\frac{K_2}{N_2-1}}\bigg\}\leq\pi,$$ and in particular, both endpoints of the transport ray are finite, and the shorter of the two focusing scales controls the total length of the transport ray. This however is not the sharpest bound, since the compatibility condition \eqref{compatibility bound} may impose a strictly smaller upper bound on $L_\alpha$. In particular,  if we ask that $$\frac{K_1}{N_1-1} =\frac{K_2}{N_2-1}=c^2>0,$$ then \eqref{compatibility bound} becomes $$ m_1\cot (c(L_\alpha-t))+m_2\cot (ct)\geq0,$$ which is equivalent to $(m_1+m_2)\sin(cL_\alpha)+(m_1-m_2)\sin(c(2t-L_\alpha))\geq0.$ Consequently, we get $$L_\alpha \leq \frac{1}{c}\bigg[\pi - \arcsin \bigg(\frac{|N_1-N_2|}{N_1+N_2-2}\bigg)\bigg],$$ and for the case $N_1\neq N_2$ the blended inequalities force $$L_\alpha<\frac{\pi}{c},$$ even though both separate Bonnet-Myers bounds permit equality. In particular, the saturation of the common focusing length requires $N_1=N_2$, and hence $K_1=K_2$.
  \item \textbf{Case $K_1>0$, and $K_2\leq 0$.} In this case, only the forward curvature parameter gives a finite Bonnet-Myers bound, and hence both endpoints are finite and $$L_\alpha \leq \pi\sqrt{\frac{N_1-1}{K_1}}.$$ However, the backward comparison can strengthen this bound even though $K_2\leq 0$  does not give a finite diameter bound. Indeed, in the case $N_1>N_2$,  the compatibility bound \eqref{compatibility bound} prevents the attainment of the forward Bonnet-Myers bound, and correspondingly we get  $L_\alpha<\pi\sqrt{\frac{N_1-1}{K_1}},$ whereas the case $N_2\geq N_1$ allows both the compatibility of \eqref{blended TMCP bound}, and the possibility of attainment of the Bonnet-Myers bound. 

  To see this, we set $$ c=\sqrt{\frac{K-1}{N_1-1}},\quad  \Phi_2(t):=\frac{S_2'(t)}{S_2(t)}=\begin{cases} \frac{1}{t},\quad &K_2=0,\\ \omega\coth(\omega t),\quad &K_2<0,\end{cases}\quad \omega = \sqrt\frac{-K_2}{N_2-1}.$$ In the case $N_1>N_2$, the saturation of the forward Bonnet-Myers bound would imply by \eqref{blended TMCP bound} that $$(N_2-1)\Phi_2(t)-(N_1-1)c \cot (ct)\geq 0,$$ whereby multiplying by $t$ and sending $t\downarrow 0$ gives $N_2\geq N_1$ which is a contradiction, and hence $L_\alpha <\pi\sqrt{\frac{N_1-1}{K_1}}.$  Conversely, if $N_2\geq N_1$, then since $K_2\leq 0$ and $0<ct<\pi$, the compatibility bound is verified for every $0<t<\frac{\pi}{c}$ since $$(N_2-1)\Phi_2(t)-(N_1-1)c\cot (ct)\geq \frac{N_2-1}{t} - \frac{N_1-1}{t}\geq 0.$$ In fact, the forward model density $h_\alpha(t)=C_\alpha \sin (ct)^{N_1-1}$ realizes the forward comparison with equality, and satisfies the backward comparison inequality, and attains $L_\alpha = \pi \sqrt{\frac{N_1-1}{K_1}}$.

  \item \textbf{$K_1\leq 0, K_2>0$.} This is exactly the time reversal of Case 2, where the relevant bounds are with respect to the backward Bonnet-Myers bound $$L_\alpha \leq \pi\sqrt{\frac{N_2-1}{K_2}}.$$ Consequently, if $N_2>N-1$, saturation of the Bonnet-Myers bound is impossible, and the case $N_1\geq N_2$  gives that the compatibility inequality \eqref{compatibility bound} is satisfied. Furthermore, the compatibility inequality holds even when $L_\alpha = \pi\sqrt{\frac{N_2-1}{K_2}},$ which can be seen by considering the density $$h_\alpha (t)=C_\alpha\sin\bigg(\sqrt{\frac{K_2}{N_2-1}}(t-a_\alpha)\bigg)^{N_2-1},\quad a_\alpha<t<a_\alpha +\pi\sqrt{\frac{N_2-1}{K_2}},$$ which satisfies both comparisons, and equality in the backward comparison.

  \item \textbf{Case $K_1,K_2\leq 0$.} In this case, there is no finite Bonnet-Myers bound. Indeed, when both endpoints are finite, the compatibility inequality becomes $$(N_1-1)\Phi_1(L_\alpha-t)+(N_2-1)\Phi_2(t)\geq 0,$$ where $$\Phi_i(t)=\begin{cases}
    \frac{1}{t},\quad &K_i=0 \\
    \omega_i\coth (\omega_i t),\quad &K_i<0,
  \end{cases}\quad \omega_i:=\sqrt\frac{-K_i}{N_i-1}.$$ Since both $\Phi_i$ are strictly positive, the compatibility holds for every $L_\alpha>0$. The constant density $h_\alpha(t)=C_\alpha$ satisfies both directional comparisons on any finite interval independently of $N_1,N_2$, and unlike in the Cases 2 and 3, the dimensional mismatch can not force an upper bound on the ray length. 

  For the bi-infinite rays, assuming that the timelike rays can be extended globally, we get the bound $$-\lambda_1\leq \frac{d}{dt}\log h(\alpha,t)\leq \lambda_2,\quad \lambda_i:=\sqrt{-(N_i-1)K_i}.$$ In particular, if $K_1=0$, then $h_\alpha$ is nondecreasing and if $K_2=0$ then $h_\alpha$ is nonincreasing, and when $K_1=K_2=0$ we get that $h_\alpha$ is constant.

\end{enumerate}
The improved Bonnet-Myers bounds for the blended $TMCP$ spacetimes motivate the following refinement of the synthetic Hawking singularity theorem of Cavalletti-Mondino\cite{CavallettiMondino2024}. In their proof, the focusing estimate that follows from the forward $TMCP^+(K_1,N_1)$ comparison. We combine their focusing estimate  with the  backward $TMCP^-(K_2,N_2)$ comparison to provide an additional constraint on the lengths of the same transport rays, whereby the compatibility of the two comparisons from \eqref{blended TMCP bound}  yield a strictly sharper focusing bound in certain blended curvature and dimension parameter regimes. 

We first state the following definitions  of Cavalletti-Mondino \cite{CavallettiMondino2024} needed for the theorem. A Borel set $V\subset X$ is called \emph{achronal} if $x\not\ll y$ for every $x,y\in V$. For such an achronal say, we say that it is \emph{future timelike complete (FTC)} if for every $x\in I^+(V)$, the set $J^-(x)\cap V$ has compact closure in $V$. For an achronal FTC set, we define the \emph{future time-separation function} $\ell_V(x)$ as $$\ell_V(x):=\sup_{v\in V}\ell(v,x),\quad x\in I^+(V).$$ Following \cite[Definition 5.2]{CavallettiMondino2024}, let $\T_V$ be the  transport set associated with $u=\ell_V$, and let $Q$ be the quotient set of equivalence classes of the transport relation $R_{\ell_V}^{nb}$. Each equivalence class is a timelike transport ray $X_\alpha$, indexed by $\alpha \in Q$, and we let let $f:Q\rightarrow V$  be the quotient map and $\q$ the corresponding quotient measure, defined in the finite-measure disintegration of Section \ref{S: transport set and transport relation}. Define the initial evaluation map of its transport rays through the ray map $g$: $$f_0:Q\rightarrow V,\quad f_0(\alpha):=\lim_{t\downarrow 0}g(\alpha,t),$$  and let $a:\T_V\rightarrow V$ be the \emph{initial point projection map}. Define the \emph{initial sectional measure} as $$\H_0:=(f_0)_\#(h(\alpha,0)\q(d\alpha)),$$ where $\q$ is the corresponding quotient measure. We say that $V$ has \emph{forward synthetic mean curvature bounded above by $H_0\in \R$} if $\H_0$ is a nonnegative Radon measure with $(f_0)_\#\q\ll \H_0$, and such that for any normal variation $$V_{t,\phi}:=\{x\in \T_V:0\leq \ell_V(x)\leq t\phi(a(x))\},$$ the following inequality holds for any bounded Borel function $\phi:V\rightarrow [0,\infty)$ with bounded support: $$ \liminf_{t\downarrow 0}\frac{\m(V_{t,\phi})-t\int_V\phi d\H_0}{t^2/2}\leq H_0\int_V\phi^2d\H_0.$$ We additionally define the \emph{Hawking focusing constant} $D_{H_0,K_1,N_1}$ as follows:

$$ D_{H_0,K_1,N_1}:=\begin{cases}
\frac{\pi}{2}\sqrt{\frac{N_1-1}{K_1}},\quad &K_1>0, H_0=0,\\
\sqrt{\frac{N_1-1}{K_1}}\cot^{-1}\bigg(\frac{-H_0}{\sqrt{K_1(N_1-1)}}\bigg),\quad  &K_1>0,H_0\neq 0,\\
-\frac{N_-1}{H_0},\quad &K_1=0, H_0<0,\\
\sqrt{\frac{N_1-1}{-K_1}}\coth^{-1}\bigg(\frac{-H_0}{\sqrt{-K_1(N_1-1)}}\bigg),\quad &K_1<0, H_0<-\sqrt{-K_1(N_1-1)},\\
+\infty &K_1=0,H_0\geq 0, \textrm{ or } K_1<0, H_0\geq-\sqrt{-K_1(N_1-1)}.
\end{cases}
$$
Define the directional Bonnet-Myers lengths by $$L_i^{BM}:=\begin{cases}\pi \sqrt{\frac{N_i-1}{
K_i}},\quad &K_i>0\\
+\infty \quad &K_i\leq 0
\end{cases}\quad i=1,2.$$ 
Finally, we define the \emph{blended Bonnet-Myers constant} $L_{blend}(K_1,N_1;K_2,N_2)$ as the supremum  of ray length admissible by the compatibility bound \eqref{compatibility bound}, whereby $$L_{blend}:=\sup\bigg\{L\in(0,\min\{L_1^{BM},L_2^{BM}\}\,:m_1\frac{S_1'(L-t)}{S_1(L-t)}+m_2\frac{S_2'(t)}{S_2(t)}\geq 0\,\, \forall t\in (0,L)]\bigg\},$$ where the endpoint restriction ensures that each curvature comparison function remains on its first positive branch.  In particular, $L_{blend}=+\infty$ when $K_1,K_2\leq 0$.  

\begin{theorem}[Blended Hawking singularity theorem] Let $(X,d,\m,\ell)$ be a timelike nonbranching gh LLS satisfying the blended $TMCP$ conditions $TMCP^+(K_1,N_1)$ and $TMCP^-(K_2,N_2)$ for $K_1,K_2\in \R$ and $N_1,N_2\in (1,\infty).$ Let $V\subset X$ be a Borel achronal FTC set having its forward synthetic mean curvature bounded above by $H_0\in \R$. Then for every $x\in I^+(V)$, we have $$\ell_V(x)\leq \min\{ D_{H_0,K_1,N_1},L_{blend}(K_1,N_1;K_2,N_2)\}.$$ In particular, for every timelike geodesic $\g\in TGeo^\ell$ with $\g_0\in V$, the maximal (on the right) domain of definition of $\g$ is contained in $[0,\min\{ D_{H_0,K_1,N_1},L_{blend}(K_1,N_1;K_2,N_2)\}].$
\end{theorem}
\begin{proof}
The proof follows from Cavalletti-Mondino \cite[Theorem 5.6]{CavallettiMondino2024}, where we consider the disintegration of $\m\restrict_{I^+(V)}$ along the transport rays $X_\alpha$ associated with $\ell_V$: $$\m\restrict_{I^+(V)}=\int_Q {g_\alpha}_\#\big(h(\alpha,\cdot))\L^1\restrict_{(0,L_\alpha)}\big)\q(d\alpha),$$ where $L_\alpha = b_\alpha -a_\alpha$ is the length of the transport ray parametrized by proper time from $V$, so that $a_\alpha = 0$, and $g$ is the ray-map from \ref{ray map definition}. The forward $TMCP^+(K_1,N_1)$ comparison gives 
$$h(\alpha,t)\geq h(\alpha,0)\bigg[\frac{S_1(s-t)}{S_1(s)}\bigg]^{N_1-1},\quad 0\leq t\leq s<L_\alpha,$$ and combining this with the synthetic mean-curvature upper bound on $V$, and following the second-order normal variation and Fatou argument of Cavalletti-Mondino gives us $$L_\alpha \leq D_{H_0,K_1,N_1}$$ for $\q$-a.e. transport ray $X_\alpha$ whenever $D_{H_0,K_1,N_1}<\infty.$

On the other hand, the backward $TMCP^-(K_2,N_2)$ comparison combined with the forward comparison through the compatibility bound \eqref{compatibility bound} gives $$ (N_1-1)\frac{S_1'(L_\alpha-t)}{S_1(L_\alpha-t)}+(N_2-1)\frac{S_2'(t)}{S_2(t)}\geq 0,\quad 0<t<L_\alpha.$$ By definition of the blended Bonnet-Myers constant, we get $$L_\alpha \leq L_{blend}(K_1,N_1;K_2,N_2).$$ Combining the above two inequalities for $L_\alpha$ yields $$L_\alpha\leq \min\{ D_{H_0,K_1,N_1},L_{blend}(K_1,N_1;K_2,N_2)\}$$ for $\q$-almost every ray, where the case of infinite upper bound is understood as imposing no restriction on $L_\alpha$. The disintegration formula, lower semicontinuity of $\ell_V$, and the full support of $\m$ extends this estimate to every $x\in I^+(V)$.

Finally, if $\g\in TGeo^\ell$ starts from $V$ and is parametrized by proper time, then $$t=\ell(\g(0),\g(t))\leq \ell_V(\g(t)),$$ and thus its maximal future proper-time parameter interval is bounded by the same constant.
\end{proof}

\begin{remark}[Extension to $p$-enb spacetimes] The timelike nonbranching assumption can be weakened to timelike $p$-enb by using the $TMCP$ localization  result of Braun-McCann \cite{BraunMcCann+}[Theorems A.5], under their corresponding hypotheses. Adapting the argument of Section 7.3 in \cite{BraunMcCann+} with the $TMCP^+(K_1,N_1)$ density, we recover the forward Hawking focusing bound of Cavalletti-Mondino, while the additional backward $TMCP^-(K_2,N_2)$ comparison through \eqref{compatibility bound} yields the same refined focusing estimate as in the preceding theorem.
\end{remark}

An important feature of the above theorem is that the upper bound $L_{blend}$ is independent of the initial mean curvature bound $H_0$. In contrast to Hawking's classical focusing argument, which is a combination of timelike Ricci curvature lower bound with an upper bound on the mean curvature of the initial hypersurface, the blended bound arises from the compatibility of forward and backward curvature-dimension comparisons along the same transport rays. This consequently provides an independent geometric restriction on the maximum timelike length of these transport rays,  governed by the compatibility of future and past volume expansions rather than by the initial mean curvature alone. In particular, the blended theorem yields a strictly sharper estimate whenever $$L_{blend}(K_1,N_1;K_2,N_2)<D_{H_0,K_1,N_1}.$$
This distinction becomes particularly relevant when the classical Hawking focusing time is infinite, since then the blended curvature conditions may yield a finite upper bound on timelike length. Indeed, taking $K_1=0, H_0\geq0, K_2>0$, so that no negative upper bound  of the initial hypersurface mean curvature is required, gives $$D_{H_0,0,N_1}=+\infty,\quad L_{blend}(0,N_1;K_2,N_2)\leq\pi\sqrt\frac{N_2-1}{K_2}<+\infty,$$ and thus the blended Hawking theorem gives $\ell_V(x)\leq L_{blend}(0,N_1;K_2,N_2)<+\infty$, despite the absence of a finite forward Hawking focusing time.

If we consider the case $K_1>0,K_2=0$ and $N_1>N_2$, the backward curvature comparison alone provides no finite Bonnet-Myers bound i.e. $L_2^{BM}=+\infty$, yet the compatibility of the two directional comparisons \eqref{compatibility bound} gives $$L_{blend}(K_1,N_1;0,N_2)<\pi\sqrt\frac{N_1-1}{K_1}=L_1^{BM}.$$ Thus, although the backward curvature condition alone gives no finite upper bound on timelike length, its compatibility with the forward curvature condition yields a strictly stronger restriction than the forward Bonnet-Myers bound. This improvement cannot be obtained by simply taking the minimum of the two individual Bonnet-Myers bounds, but arises from the interaction of the forward and backward volume-density comparisons.

More generally, a strict improvement of the Hawking focusing bound is possible when $N_1=N_2=N$. For instance, taking $K_1>0, K_2>4K_1$ and $H_0=0$, we get

$$L_{blend}(K_1,N;K_2,N)=\pi\sqrt\frac{N-1}{K_2}<\frac{\pi}{2}\sqrt\frac{N-1}{K_1}=D_{0,K_1,N},$$ which shows that the blended theorem can provide a strictly sharper focusing bound without requiring different synthetic dimension parameters.

The curvature parameters appearing in the blended Hawking focusing theorem also have a natural interpretation in terms of the cosmological constant $\Lambda$ in Einstein's field equations. In four-dimensional general relativity, the strong energy condition gives the timelike Ricci lower bound $\Ric_g(u,u)\geq -\Lambda$ for every unit timelike vector $u$. Hawking's classical singularity theorem relies on timelike convergence condition $\Ric_g(u,u)\geq0$, guaranteed by the strong energy condition through Einstein's equations with $\Lambda= 0$ in his original formulation \cite{Hawking1966, HawkingEllis1973}. More generally, the same condition is guaranteed when $\Lambda\leq 0$, as explicitly considered in the subsequent Hawking-Penrose singularity theorem \cite{HawkingPenrose1970}. The extension to positive $\Lambda$ where timelike convergence condition is no longer guaranteed was studied by Borde \cite[Theorem 4]{Borde} and Andersson-Galloway \cite[Proposition 3.3]{Andersson-Galloway}. Their results establish past timelike geodesic incompleteness under a sufficiently large positive mean curvature of the initial hypersurface. Reversing the time orientation yields the corresponding future focusing estimate under the upper mean-curvature bound $H_0\leq -\sqrt{3\Lambda}$. Setting our synthetic curvature parameter as $K_1=-\Lambda$, the future focusing time is given by $$D_{H_0,-\Lambda,4}=\sqrt{\frac{
3}{\Lambda}}\coth^{-1}\bigg(\frac{-H_0}{\sqrt{3\Lambda}}\bigg)<+\infty.$$
The blended curvature conditions provide an additional perspective on this classical focusing mechanism and its positive-$\Lambda$ extension. In ordinary smooth Lorentzian geometry, the timelike Ricci curvature is invariant under time reversal, and hence the parameters $K_1,K_2$ should not be interpreted as two distinct physical cosmological constants. In the synthetic setting, however, the absence of an underlying symmetric Ricci tensor allows the forward and backward curvature-dimension comparison to encode genuinely different geometric behaviour. Such asymmetry is particular natural in nonsmooth and nonreversible spactimes, where the compatibility of the two directional curvature bounds may yield stronger conclusions than the classical one-sided arguments. The interpretation of  distinct forward and backward synthetic curvature bounds as different effective cosmological constants for synthetic spacetimes is developed in our forthcoming work.

\end{subsection}
\subsection{Equivalence of $TCD_p(K,N)$ conditions for $p=1$ via needle decomposition}

In this subsection, we will define the two different notions of the Timelike Curvature Dimension condition corresponding to the limit exponent $p=1$, namely $TCD_1(K,N)$ and $TCD_{rLip}^1(K,N)$, and prove their equivalence by means of needle decomposition.
These two notions are based on the $L^1$ optimal transportation
problem. Throughout this section, we impose the assumption $\m(X)=1$, $\spt(\m)=X$, which is without loss of generality by the works of Cavalletti-Mondino \cite{CavallettiMondino2017} which allows to extend the constructions
to measures $\m$ that are $\sigma$-finite.
 We also assume that $X$ is proper.

Given a measured gh LLS $(X,d,\m, \ell)$, together with the R\'enyi entropy functional $S_N(\cdot|\m):\P(X)\rightarrow [-\infty,0]$ from Definition \eqref{Renyi},
where $K\in \R, N\in [1,\infty)$ correspond to the curvature and the dimension parameters respectively and are fixed. We first recall the $TCD_p(K,N)$ condition for $p\in(0,1)$ that appeared in Braun \cite{BraunTCD}, see also \cite{McCann2020,CavallettiMondino2024,Cavalletti-Mondino2025+}.
\begin{definition}($TCD_p (K,N)$) condition) Let $p\in(0,1)$.
 A gh LLS $(X,d,\m,\ell)$ spacetime satisfies the $TCD_p (K,N)$ condition if for any absolutely continuous timelike $p$-dualisable  pair $(\mu_0,\mu_1)=(\rho_0\m,\rho_1\m)\in \P_c(X)^2$, there exists an  $\ell_p$-optimal dynamical plan $\nu\in OptTGeo_p(\mu_0,\mu_1)$ and a coupling  $\pi\in \Pi_{\ll}^{p\rm{-opt}}(\mu_0,\mu_1)$, such that setting   $\mu_t:=(e_t)_{\#}\nu$, the following entropy inequality holds:
 $$S_{N'}(\mu_t| \m) \leq -\int_{X\times X} \big{(}  \tau_{K,N'}^{(1-t)}(\ell(x_0,x_1)) \rho_0^{-1/N'}(x_0)+\tau_{K,N'}^{(t)}(\ell(x_0,x_1)) \rho_1^{-1/N'}(x_1) ) \\ \big{)} d\pi(x_0,x_1)$$  for all $N'\geq N,$ and $t\in[0,1].$ 
\end{definition}

The following definition of $TCD_1(K,N)$ also uses displacement convexity, and is interpreted as the limiting case where $p\rightarrow1^-.$
\begin{definition} (Limiting  $TCD_1(K,N)$){\label{Limiting TCD definition}}
Let $(X,d,\m,\ell)$ be a 
gh LLS. The
spacetime satisfies the $TCD_{1}(K,N)$ $condition$ if for
every absolutely continuous pair $(\mu_{0},\mu_{1})=(\rho_0\m,\rho_1\m)\in\mathcal{P}_{c}(X)^2$ that is timelike-1-dualisable,
there exists a Borel probability measure $\eta\in\mathcal{P}(C([0,1],X))$
concentrated on constant speed timelike geodesics such that   $\int \ell(\gamma_{0},\gamma_{1})d\eta(\gamma)=\ell_{1}(\mu_{0},\mu_{1})$
and such that 

\begin{equation}{\label{Entropic Inequality}}
S_{N'}(\mu_{t}|\m)\leq-\int_{TGeo^{\ell}} \tau_{K,N'}^{(1-t)}(\ell(\gamma_{0},\gamma_{1}))\rho_{0}^{-1/N'}(\gamma_{0})+\tau_{K,N'}^{(t)}(\ell(\gamma_{0},\gamma_{1}))\rho_{1}^{-1/N'}(\gamma_{1})d\eta(\gamma)
\end{equation}

holds for every $t\in[0,1]$ and all $N'\geq N$, where  $\mu_t:=(e_t)_\# \eta\ll\m$ and $(e_0,e_1)_\# \eta=(\mu_0,\mu_1).$
\end{definition}

The additional requirement at $p=1$ is that the optimal transport is dynamically realized by a plan concentrated on the constant speed timelike geodesics. Unlike the case $0<p<1$, the cost $\ell$ is linear, rather than strictly concave, so the $\ell^1$-optimaltiy alone does not select the same dynamical structure.

We additionally define the Localized $TCD^1_{rLip}(K,N)$ condition which is based on integral curves of reverse-Lipschitz functions  following the works of \cite{CavallettiGigliSantarcangelo2021, CavallettiMondino2024,CavallettiMilman2021} 
in the Riemannian setting, and \cite{CavallettiMondino2024,BraunMcCann+} in the Lorentzian signature. As in Section \ref{Disintegration section}, we restrict to the admissible class of reverse 1-Lipschitz functions $u,$ which are continuous reverse 1-Lipschitz functions $u:Z\rightarrow \R$, where $Z\subset X$ is a compact timelike interpolation set such that every maximizing timelike geodesic joining a par $x,y\in Z$ satisfying $u(y)-u(x)=\ell(x,y)>0$ is contained in $Z$.

\begin{definition}
(Localized $TCD^1_{rLip}(K,N)$)
Let $(X,d,\m,\ell)$ be a gh LLS. Let
$u:Z\rightarrow \R$ be a continuous reverse-Lipschitz function, where $Z$ is a timelike interpolation set in the sense of the Section \ref{Disintegration section}. The spacetime satisfies
the $TCD_{u}^{1}(K,N)$ condition if there exists a family $\{X_{\alpha}\}_{\alpha\in Q}\subset X$
such that:

\begin{enumerate}
\item There is a strongly consistent disintegration of $\m\llcorner{}_{\T_{u}}=\int_{Q}\m_{\alpha}\q(d\alpha)$, where $\m_\alpha(X_\alpha)=1$ for $\q$-a.e. $\alpha\in Q$, and $\m(\T_u\setminus \T_u^{nb})=0$,
\item for $\q$-a.e. $\alpha\in Q,$ $f^{-1}(\alpha) =:X_{\alpha}$ is a transport ray for
$\Gamma_{u}$, where $f$ is from Proposition \ref{quotient map},
\item for $\q$-a.e. $\alpha\in Q,$ we have $\m_{\alpha}=h(\alpha,\cdot)\L^{1}\llcorner_{X_{\alpha}}$,
and metric spacetime $(X_\alpha, \ell\restrict_{X_\alpha},\m_\alpha)$  satisfies the $CD(K,N)$ density condition for all $r_0<r_1$ and $t\in[0,1]$: $$h(\alpha,(1-t)r_0+tr_1)^{\frac{1}{N-1}}\geq \sigma_{K,N-1}^{(1-t)}(r_1-r_0)h(\alpha,r_0)^{\frac{1}{N-1}}+\sigma_{K,N-1}^{(t)}(r_1-r_0)h(\alpha,r_1)^{\frac{1}{N-1}}.$$
\end{enumerate}

Furthermore, we say that the space satisfies the $TCD^{1}_{rLip}(K,N)$ $condition$
if it satisfies the $TCD_{u}^{1}(K,N)$ condition for every admissible 
reverse 1-Lipschitz function $u:Z\rightarrow\R$.

\end{definition}

We now state the equivalence of the Localized and the Limiting $TCD$ conditions.
\begin{theorem}
Let $(X,d,\m,\ell)$ be an $p$-enb
gh LLS for some $p\in(0,1)$, and let $K\in \R, N\in(1,\infty)$. Assume that for every absolutely continuous timelike dualisable pair 
$(\mu_0,\mu_1)$ there is a $\pi\in \Pi_\ll^{1\textrm{-opt}}(\mu_0,\mu_1)$   and an admissible continuous reverse 1-Lipschitz 
$u:Z\rightarrow \R$ such that $\pi(\Gamma_u)=1$. Then the spacetime satisfies
the $TCD_{1}(K,N)$ condition if and only it satisfies the $TCD^{1}_{rLip}(K,N)$
condition.
\end{theorem}

We will separately present the two implications in the following two subsections.

\subsubsection{$TCD_1(K,N)\implies TCD^1_{rLip}(K,N):$ ``Limiting" implies ``Localized".}\label{limiting implies localized}

Fix $u:X\rightarrow\R$, an admissible continuous reverse-Lipschitz function,
and assume that the spacetime satisfies the $TCD_{1}(K,N)$ condition. Additionally, assume that the spacetime is forward and backwards $p$-enb.

\textbf{Step 1:} We first show that the $TCD_{1}(K,N)$ condition implies both the $TMCP^+(K,N)$ and the 
 $TMCP^-(K,N)$ conditions.

\begin{lemma}
 Let $(X,d,\m,\ell)$ be gh LLS $p$-enb spacetime satisfying $TCD_{1}(K,N)$ condition. Then for
any $\mu_{0}\in\P_{c}(X)$ with $\mu_{0}\ll\m$ and $\spt \mu_0\subset I^-(x_1), $
 there exists a dynamical plan $\eta$ with $\mu_{t}:=(e_t)_\#\eta$ which is an $\ell_{p}$ geodesic for
any $p\in(0,1)$ such that $\mu_{t}=\rho_{t}\m+\mu_{t}^{\rm{sing}}$ for all
$t\in[0,1)$, and for all $t\in[0,1)$ and $N'\geq N$ we have
\begin{equation} \label{TMCP inequality}
\int_{X}\rho_{t}^{-1/N'}d\mu_{t}\geq\int_{X}\tau_{K,N'}^{(1-t)}(\ell(x,x_1))\rho_{0}^{-1/N'}(x)d\mu_{0}(x).
\end{equation}

\end{lemma}

\begin{proof}
Given an absolutely continuous $\mu_{0}\in\P_{c}(X)$ and an $x_1$ such that $\mu_0(I^-(x_1))=1$,
since $\spt(\m)=X$ we can consider $\mu_{1,\epsilon}:=c_{\epsilon}\m\llcorner_{B_{\epsilon}(x_{1})}$
with $c_{\epsilon}>0$ a normalization constant. By the definition
of $TCD_{1}(K,N)$ there exists a dynamical plan $\nu_{\epsilon}$
between $\mu_{0},\mu_{1,\epsilon}$, and we let $\mu_{t,\epsilon}:=(e_{t})_{\#}\nu_{\epsilon}$.
Since $x_{1}\in I^{+}(\text{spt}(\mu_{0}))$,  it follows that there exists a small $\epsilon>0$ such that $\spt\mu_0\times \spt\mu_{1,\epsilon} \Subset \{ \ell>0\}$ and so the pair $(\mu_{0},\mu_{1,\epsilon})$  is timelike
1-dualisable.

Since the space is globally hyperbolic, define the following set $Z$, where we recall \eqref{Midpoints}:  $$Z:=Z(\spt \mu_0, \spt \mu_1)=\cup_{s\in [0,1]}Z_s(\spt \mu_0, \spt \mu_1).$$ Since $Z\subset J^{+}(\spt(\mu_{0}))\cap J^{-}(\spt(\mu_{1,\epsilon}))$, we conclude that $Z$ is compact. Therefore the family $\nu_{\epsilon}$ is pre-compact
by Prokhorov's theorem \cite[Theorem 2.54]{BraunMcCann+}, and we can obtain a limit dynamical plan $\nu$
that is concentrated on constant speed timelike geodesics with marginals
at times $t=0,1$ given by $\mu_{0}$ and $\mu_{1}=\delta_{x_{1}}$
respectively. 

Let $\mu_{t}:=(e_{t})_{\#}\nu$, we show that $\mu_{t}$ is a $\ell_{p}$-geodesic for all $p\in(0,1).$
This follows by nothing that the coupling $(e_s,e_t)_\#\nu$ is admissible between $\mu_s$ and $\mu_t$, and then directly computing 

\begin{align*}
\ell_{p}(\mu_{s},\mu_{t})^p &\geq \left(\int_{TGeo^{\ell}}\text{\ensuremath{\ell^{p}}}(\gamma_{s},\gamma_{t})d(e_{s},e_{t})_{\#}\nu\right)\\
&=(t-s)^p\left(\int\ell^{p}(\gamma_{0},\gamma_{1}) d\nu(\gamma)\right) \\
&=(t-s)^p\ell_{p}(\mu_{0},\mu_{1})^p,
\end{align*}

\noindent where in the first inequality we have used that $(e_{t},e_{s})_{\#}\nu$
is an admissible plan from $\mu_{s}$ to $\mu_{t}$, in the second
line we have used that $\gamma$ are geodesics, and in the last step
we have used that $(e_{0},e_{1})_{\#}\nu$ is the only transport plan
from $\mu_{0}$ to $\delta_{x_{1}}$, and thus optimal for every $p\in(0,1).$ The reverse inequality follows by the reverse triangle inequality of $\ell_p$. Hence, we conclude that $\mu_t$ is an $\ell^p$-geodesic for every $p\in(0,1).$

The inequality \eqref{TMCP inequality} then follows by the lower semicontinuity of
the entropy functional $S_N$.
\end{proof}

In the above Lemma, the $TMCP^-(K,N)$ condition follows by causal reversal of the hypothesis. Thus, we get that the spacetime satisfies the full $TMCP(K,N)$ condition, which together with the assumption of $p$-enb allows us to invoke
the results of Section \ref{Disintegration section} to obtain a strongly consistent disintegration formula

\[
\m\restrict_{\T_{u}^{nb}}=\int_{Q}\m_{\alpha}\q(d\alpha),
\]

\noindent where the first equality holds by Theorem \ref{endpoints measure zero}. Moreover, we have $\m(\T_u\setminus \T_u^{nb})=1$. Hence we are only left to show that ($X_{\alpha},\ell\restrict_{X_\alpha},\m_{\alpha})$
satisfies the $CD(K,N)$ condition.

\vspace{1 mm}

\textbf{Step 2:}  Since the spacetime is gh LLS and $p$-enb for some $p\in(0,1),$ and satisfies the $TMCP(K,N)$ condition, then for $\q$-a.e. $\alpha$ we have $\m_{\alpha}=h(\alpha,\cdot)\L^{1}\restrict_{X_{\alpha}}$
and the one dimensional metric measure space ($X_{\alpha},\ell\restrict_{X_\alpha},\m_{\alpha})$
satisfies the $TMCP(K,N)$ condition. In particular, $h(\alpha,\cdot)$
is strictly positive in the relative interior of $X_{\alpha}$, and
it is locally Lipschitz by Theorem \ref{Big TMCP+ theorem}. 
Recall the \emph{ray map} $g:Dom(g)\subset Q\times\R\rightarrow\T_{u}^{nb}$ from Definition \ref{ray map definition}, and its properties from Proposition \ref{ray map properties}.  The following Lemma is the adaptation of the construction by Cavalletti-Gigli-Santarcangelo \cite[Lemma 3.3]{CavallettiGigliSantarcangelo2021}. The only
thing that needs to be checked is that the constructed measures $\mu_{0},\mu_{1}$
are timelike 1-dualisable.
\begin{lemma}\label{TMCP sup inf inequality}
For any $\bar{Q}\subset Q$ Borel sets with positive $\q$-measure
and $R_{0},R_{1},L_{0},L_{1}\in\R$ with $R_{0}<R_{1},~L_{0},L_{1}>0$
and $[R_{0},R_{1}+L_{1}]$ belonging to the domain of $h(\alpha,\cdot)$ for $\q$-a.e. $\alpha \in \bar Q,$
we have

\[
(L_{t})^{1/N}\sup_{\bar{Q}}h(\alpha,R_t)^{1/N}\geq(L_{0})^{1/N}\tau_{K,N}^{(1-t)}(R_1-R_0)\inf_{\bar{Q}}h(\alpha,R_0)^{1/N}
\]

\[
+(L_{1})^{1/N}\tau_{K,N}^{(t)}(R_1-R_0)\inf_{\bar{Q}}h(\alpha,R_1)^{1/N}
\] for every $t\in[0,1],$ where $R_{t}=(1-t)R_{0}+tR_{1},$ and $L_{t}=(1-t)L_{0}+tL_{1}.$
\end{lemma}

\begin{proof}
\emph{Step 1}:  Fix $\bar{Q}\subset Q$ Borel with positive $\q$-measure and consider
$R_{0},R_{1},L_{0},L_{1}$ given in the hypothesis. Let $\epsilon>0$
and for $i=0,1$ define the following probability measures:

\begin{equation}\label{ray measures}
\mu_{i}:=\frac{1}{\q(\bar{Q})}\int_{\bar{Q}}g(\alpha,\cdot)_{\#}\big(\frac{1}{\epsilon L_{i}}\mathcal{L}^{1}\restrict_{[R_{i},R_{i}+\epsilon L_{i}]}\big)\q(d\alpha),
\end{equation} and choose $\epsilon>0$ sufficiently small so that $R_0+\epsilon L_0<R_1$. Thus, every source point on a ray lies strictly in the chronological past of every target point on the same ray, and for such a pair of measures, the transport has to be performed in
the timelike direction along the rays $\{X_{\alpha}\}_{\alpha\in\bar{Q}}$,
due to the disintegration associated to $\T_{u}^{nb}$. 

Therefore, the measures $(\mu_0,\mu_1)$
are  timelike 1-dualisable, as the
transport plan $\pi$ which rearranges the mass monotonically along
each ray is optimal. Hence $\spt(\pi)\subset\Gamma_{u}$ and so this
plan is $\ell$-cyclically monotone and thus $\ell_{1}$-optimal.

We will show that all other optimal plans enjoy this same property. If not, there would exist at least one optimal plan $\bar{\pi}$ such
that for some $\bar{Q}_{1}\subset\bar{Q}$ of positive measure and
for some $S\subset\R$ , with $\bar{Q}_{1}\times S\subset\text{Dom}(g)$,
we have
\[
\bar{\pi}(\{(g(\alpha,s),g(\alpha',s'):\alpha,\alpha'\in\bar{Q}_{1},s,s'\in S,~\text{and}~\alpha\neq\alpha'\})>0.
\]

\noindent Consider the plan 
\[
\pi^{*}=\frac{\pi+\bar{\pi}}{2},
\]
which is trivially optimal for $\mu_{0},\mu_{1}.$ We will construct
a set of branching points with positive measure, which will lead to
a contradiction. 

Let $v$ be the Kantorovich potential associated to the $\ell_{1}$
optimal transport problem between $\mu_{0},\mu_{1}$ from \eqref{ray measures} (possibly different
from the fixed $u$). Theorem \ref{endpoints measure zero} then implies that $\m(A_{\pm,v})=0.$
However, we see that 
\[
P_{1}(\{(g(\alpha,s),g(\alpha',s'):\alpha,\alpha'\in\bar{Q}_{1},s,s'\in S,~\text{and}~\alpha\neq\alpha'\})\subset A_{\pm,v},
\]
and since $\mu_{0}\ll\m$ this contradicts that $\bar{\pi}(\{(g(\alpha,s),g(\alpha',s'):\alpha,\alpha'\in\bar{Q}_{1},s,s'\in S,~\text{and}~\alpha\neq\alpha'\})>0.$

It follows that every optimal plan will have its support contained
in the set 
\[
A_{\bar{Q}}^{\epsilon}:=\cup_{\alpha\in\bar{Q}} \{(g(\alpha,[R_{0},R_{0}+\epsilon L_{0}])\times g(\alpha,[R_{1},R_{1}+\epsilon L_{1}]))\}
\]

\emph{Step 2}: Since absolutely continuous measures $\mu_{0},\mu_{1}$ from \eqref{ray measures} were constructed to be 
timelike 1-dualisable and have compact support, by the $TCD_{1}(K,N)$
condition there exists an optimal dynamical plan $\nu$ that is concentrated
on $TGeo^{\ell}(X)$ such that for $\mu_{t}:=(e_{t})_{\#}\nu$ the limiting entropy inequality from Definition \ref{Limiting TCD definition} is satisfied. By construction of $\mu_0,\mu_1$, and since
$\nu$ is concentrated on timelike geodesics, we see that for $\q$-a.e.
$\alpha\in Q,$ that $\mu_{t}=(e_t)_\#\nu $ is $0$ for $\m_{\alpha}$-a.e. outside
of the image defined by the ray map $g(\alpha,[R_{t},R_{t}+\epsilon L_{t}]).$

The rest of the proof follows by estimating from 
 both sides of the Entropic inequality in the $TCD_{1}(K,N)$
condition using the disintegration theorem and Jensen's inequality. We include the argument for completeness. We estimate the left hand side of \ref{Entropic Inequality} by the inequalities above, where we set $\mu_t = \rho_t\m$:

\begin{align*}
\int_X \rho_t^{1-\frac{1}{N}}d\m & = \int_{\bar{Q}}\int_{X_{\alpha}} \rho_t(x)^{1-\frac{1}{N}} \m_{\alpha}(dx)\q (d\alpha) = \int_{\bar{Q}}\int_{R_t}^{R_t+\epsilon L_t} \rho_t(g(\alpha,s))^{1-\frac{1}{N}} h(\alpha,s) ds\q (d\alpha) \\
 & \leq  \int_{\bar{Q}} \sup_{[R_t,R_t+\epsilon L_t]} h(\alpha,\cdot )^{\frac{1}{N}} \int_{R_t}^{R_t+\epsilon L_t} (\rho_t(g(\alpha,s) h(\alpha,s))^{1-\frac{1}{N}}  ds\q (d\alpha)\\
 &  \leq  \int_{\bar{Q}} \sup_{[R_t,R_t+\epsilon L_t]} h(\alpha,\cdot )^{\frac{1}{N}} \bigg(\int_{R_t}^{R_t+\epsilon L_t} (\rho_t(g(\alpha,s)h(\alpha,s )  ds\bigg) ^{1-\frac{1}{N}} \bigg(\int_{R_t}^{R_t+\epsilon L_t}  ds\bigg)^{\frac{1}{N}}\q (d\alpha)\\
&=\int_{\bar{Q}} \sup_{[R_t,R_t+\epsilon L_t]}h(\alpha,\cdot )^{\frac{1}{N}} \bigg(\frac{1}{\q(\bar{Q})} \bigg)^{1-1/N} (\epsilon L_t)^{1/N} \q(d\alpha)\\
 &\leq (\epsilon L_t \q(\bar{Q})^{\frac{1}{N}}\sup_{\alpha \in \bar{Q}}\bigg(\sup_{[R_t,R_t+\epsilon L_t]} h^{\frac{1}{N}}(\alpha,\cdot) \bigg).
 \end{align*}

\noindent Similarly, the right hand side of \ref{Entropic Inequality} is estimated as follows, where $\pi = (e_0,e_1)_{\#} \nu$:

\begin{align*}
\int_{X\times X} &\tau_{K,N}^{(1-t)}(\ell(x,y))\rho_{0}^{-1/N}(x)+\tau_{K,N}^{(t)}(\ell(x,y))\rho_{1}^{-1/N}(y)d\pi(x,y)\\
&\geq \inf_{A_{\bar{Q}}^{\epsilon}} \tau_{K,N}^{(1-t)}(\ell(x,y))\int_X\rho_{0}^{-1/N}(x)d\mu_0(x)+\inf_{A_{\bar{Q}}^{\epsilon}} \tau_{K,N}^{(t)}(\ell(x,y))\int_X \rho_{1}^{-1/N}(y)d\mu_1(y) \\
&\geq (\epsilon \q(\bar{Q}))^{\frac{1}{N}}\bigg[\inf_{A_{\bar{Q}}^{\epsilon}} \tau_{K,N}^{(1-t)}(\ell(x,y)) \inf_{\bar{Q}}\bigg( \inf_{[R_0,R_0+\epsilon L_0]}h(\alpha,\cdot )^{\frac{1}{N}} \bigg)(L_0)^{\frac{1}{N}} \\ 
&\quad \quad +\inf_{A_{\bar{Q}}^{\epsilon}} \tau_{K,N}^{(t)}(\ell(x,y)) \inf_{\bar{Q}}\bigg( \inf_{[R_1,R_1+\epsilon L_1]}h(\alpha,\cdot )^{\frac{1}{N}} \bigg)(L_1)^{\frac{1}{N}} \bigg].
 \end{align*}

Combining both of these estimates, we get:

\begin{align*}
(L_t)^{\frac{1}{N}}\sup_{\bar{Q}}\bigg(\sup_{[R_t,R_t+\epsilon L_t]} h(\alpha,\cdot )^{\frac{1}{N}}\bigg) \geq &\inf_{A_{\bar{Q}}^{\epsilon}} \tau_{K,N}^{(1-t)}(\ell(x,y)) \inf_{\bar{Q}}\bigg( \inf_{[R_0,R_0+\epsilon L_0]}h(\alpha,\cdot )^{\frac{1}{N}} \bigg)(L_0)^{\frac{1}{N}} \\ 
&\quad \quad +\inf_{A_{\bar{Q}}^{\epsilon}} \tau_{K,N}^{(t)}(\ell(x,y)) \inf_{\bar{Q}}\bigg( \inf_{[R_1,R_1+\epsilon L_1]}h(\alpha,\cdot )^{\frac{1}{N}} \bigg)(L_1)^{\frac{1}{N}},
\end{align*}

\noindent and sending $\epsilon\rightarrow 0,$ we conclude that 

\begin{equation}\label{estimate} \begin{aligned}
(L_t)^{\frac{1}{N}}\sup_{\bar{Q}} h(\alpha,\cdot )^{\frac{1}{N}}(R_t) &\geq (L_0)^{\frac{1}{N}} \inf_{A_{\bar{Q}}} \tau_{K,N}^{(1-t)}(\ell(x,y)) \inf_{\bar{Q}}h(\alpha,\cdot )^{\frac{1}{N}}(R_0) \\
&\quad \quad +(L_1)^{\frac{1}{N}}\inf_{A_{\bar{Q}}} \tau_{K,N}^{(t)}(\ell(x,y)) \inf_{\bar{Q}}h(\alpha,\cdot )^{\frac{1}{N}}(R_1),
\end{aligned}
\end{equation} where $A_{\bar{Q}}:=\cup_{\alpha\in\bar{Q}}\{(g(\alpha,R_0)\times g(\alpha,R_1))\}$ and the desired conclusion follows by recalling that $g(\alpha,\cdot)$ is an $\ell$-isometry.
\end{proof}

With this uniform estimate on the one dimensional density $h(\alpha,\cdot)$, the results of  Cavalletti-Gigli-Santarcangelo \cite[Proposition 3.4]{CavallettiGigliSantarcangelo2021} apply and it follows that  $h(\alpha,\cdot )$ is in fact a $CD(K,N)$ density on $X_{\alpha}$. Their proof goes through an appropriate choice of $L_0, R_0, L_1, R_1$ in equation \eqref{estimate}. Since a timelike transport ray is $\ell$-isometric to an interval in $\R$, optimal transport between two chronologically ordered measures on the ray is the monotone one-dimensional transport. The corresponding $TCD_p(K,N)$ entropy inequality is therefore characterized by the same density inequality, and is independent of the exponent $p\in(0,1)$. Hence, each ray satisfies the $TCD_p(K,N)$ condition for every $p\in (0,1)$. 

\begin{proposition}(\cite[Proposition 3.4]{CavallettiGigliSantarcangelo2021})
For $\q$-a.e. $\alpha\in Q,$ the  space $(X_{\alpha},\ell\restrict_{X_\alpha},\m_{\alpha})$
is a transport ray that satisfies the $TCD_p(K,N)$ condition, which is independent of $p\in(0,1)$.
\end{proposition}

To conclude, we have shown that the Limiting timelike curvature dimension $TCD_1(K,N)$ implies the Localized timelike curvature dimension condition $TCD^1(K,N)$ whenever the space is $p$-enb. 

\subsubsection{$TCD^1_{rLip}(K,N)\implies TCD_1(K,N):$ ``Localized" implies ``Limiting".}

Assume that the space satisfies the $TCD^{1}_{rLip}(K,N)$ condition. Given  an absolutely continuous timelike 2-dualisable pair
$(\mu_{0},\mu_{1})$  with
$\mu_{0},\mu_{1}\in\P_{c}(X)$, we construct an $\ell_{1}$-geodesic
that satisfies the Entropy inequality from the definition of $TCD_{1}(K,N)$

By assumption in the theorem, there exists an admissible continuous reverse 1-Lipschitz potential $u:Z\rightarrow \R$ and a chronological optimal coupling $\pi\in \Pi_\ll^{opt}(\mu_0,\mu_1)$, with $\pi(\Gamma_u)=1.$

Since $u$ is admissible and continuous reverse 1-Lipschitz,  the $TCD^{1}_{rLip}(K,N)$
condition provides the existence a family of transport rays $\{X_{\alpha}\}_{\alpha\in Q}\subset X$
and a strongly consistent disintegration of $\m\llcorner_{\T_{u}}$ on $\{X_{\alpha}\}_{\alpha\in Q}$
such that 
\[
\m\restrict_{\T_{u}}=\int_{Q}\m_{\alpha}\mathfrak{q}(d\alpha),
\]

\noindent with $\m_{\alpha}(X_{\alpha})=1$ for $\q$-a.e. $\alpha\in Q,$ and
$(X_{\alpha},\ell\restrict_{X_\alpha},\m_{\alpha})$ satisfies the $CD(K,N)$ condition.

It follows that $$\mu_{0}=\rho_{0}\m=\int_{Q}\rho_{0}\m_{\alpha}\q(d\alpha)=\int_{Q}\mu_{0,\alpha}\q_{0}(d\alpha)$$
where $\mu_{0}$ is concentrated on $\T_{u}^{nb}$, with $$\mu_{0,\alpha}=\rho_{0}\m_{\alpha}\cdot\bigg(\int\rho_{0}\m_{\alpha}\bigg)^{-1},$$
and $\q_{0}:=(f)_{\#}\mu_{0}$ where $f$ is the quotient map. Since $f$ is the quotient map associated with the transport ray equivalence, relation, for every $(x,y)\in (\Gamma_u\setminus \Delta_X)\cap (\T_u^{nb}\times \T_u^{nb})$ we have $f(x)=f(y).$ Hence, for any Borel $C\subset Q$, we have the following equality of sets:
\[
(X\times f^{-1}(C))\cap(\Gamma_{u}\backslash\{x=y\})\cap(\T_{u}^{nb}\times\T_{u}^{nb})
=
(f^{-1}(C)\times X)\cap(\Gamma_{u}\backslash\{x=y\})\cap(\T_{u}^{nb}\times\T_{u}^{nb}).
\]

\noindent Since the pair is timelike dualisable, and the optimal transport is timelike $\mu_{0}(\T_{u}^{nb})=\mu_{1}(\T_{u}^{nb})=1$
gives that $$\pi((\Gamma_{u}\backslash\{x=y\})\cap\T_{u}^{nb}\times\T_{u}^{nb})=1.$$

This implies that for any Borel $C\subset Q$, we have the following mass balance condition:

\begin{align*}
\mu_{0}(f^{-1}(C)))  & = \pi(\left(f^{-1}(C)\times X\right)\cap(\Gamma_{u}\backslash\{x-y\})) \\
  & =  \pi(\left(X\times f^{-1}(C)\right)\cap(\Gamma_{u}\backslash\{x-y\})) \\
 & =  \mu_{1}(f^{-1}(C))
\end{align*}

\noindent and in particular $\q_{0}=\q_{1}:=(f)_{\#}\mu_{1}.$ We can now decompose $\mu_{1}$ as 

\[
\mu_{1}=\rho_{1}\m=\int_{Q}\rho_{1}\m_{\alpha}\q(d\alpha)=\int_{Q}\mu_{1,\alpha}\q_{0}(d\alpha),
\]

\noindent with $$\mu_{1,\alpha}=\rho_{1}\m_{\alpha}\cdot\bigg(\int\rho_{1}\mu_{\alpha}\bigg)^{-1}.$$ 

By construction, for $\q$-a.e. $\alpha\in Q$, $\mu_{0,\alpha},\mu_{1,\alpha}$
are absolutely continuous with respect to $\m_\alpha$ on $X_{\alpha}$, By the $TCD_{rLip}^{1}(K,N)$ condition
on $(X_{\alpha},\ell\restrict_{X_\alpha},\m_{\alpha})$, there exists an optimal dynamical
plan $\nu_{\alpha}$ such that $\rho_{t,\alpha}\m_{\alpha}=\mu_{t,\alpha}=(e_{t})_{\#}\nu_{\alpha}$
is an $\ell_{1}$-geodesic that interpolates $\mu_{0,\alpha}$ and
$\mu_{1,\alpha}$, and satisfies
\begin{equation} \label{pointwise}
\rho_{t,\alpha}^{-1/N'}(\gamma_{t})\geq\tau_{K,N'}^{(1-t)}(\ell(\gamma_{0},\gamma_{1}))\rho_{0,\alpha}^{-\frac{1}{N'}}(\gamma_{0})+\tau_{K,N'}^{(t)}(\ell(\gamma_{0},\gamma_{1}))\rho_{1,\alpha}^{-\frac{1}{N'}}(\gamma_{1})\text{ }
\end{equation}

\noindent for every $N'\geq N$, $t\in[0,1]$, and  for $\nu_\alpha$ a.e. $\gamma$.

We can combine the 1-dimensional geodesics by defining $\nu:=\int_{Q}\nu_{\alpha}\q_{0}(d\alpha).$
Then define $\mu_{t}=(e_{t})_{\#}\nu$. We claim that $\{\mu_{t}\}$
is an $\ell_{1}$-geodesic interpolating between $\mu_{0}$ and $\mu_{1}.$
Indeed, for $0\leq s\leq t\leq1$, since $(e_0,e_1)_\#\nu$ is $\ell^1$- optimal, we have:

\[
\begin{aligned}\ell_{1}(\mu_{s},\mu_{t}) & \geq\int_{X\times X}\ell(x,y)(e_{t},e_{s})_{\#}\nu(dxdy)\\
 & =\int_{Q}\int_{X_{\alpha}\times X_{\alpha}}\ell(x,y)(e_{t},e_{s})_{\#}\nu_{\alpha}(dxdy)\q_{0}(d\alpha)\\
 & =(t-s)\int_{Q}\int_{X_{\alpha}\times X_{\alpha}}\ell(x,y)(e_{0},e_{1})_{\#}\nu_{\alpha}(dxdy)\q_{0}(d\alpha)\\
 & =(t-s)\int_{X\times X}\ell(x,y)(e_{0},e_{1})_{\#}\nu(dxdy)\\
 & =(t-s)\ell_{1}(\mu_{0},\mu_{1}).
\end{aligned}
\]

Now we show that $\mu_{t}$ satisfies the entropy convexity inequality from the definition \ref{Limiting TCD definition} of $TCD_1(K,N)$.
Indeed since $\mu_{t}=\rho_{t}\m$, this implies that $\rho_{t,\alpha}=\rho_{t}\cdot(\int\rho_{0}\m_{\alpha})^{-1}$. Plugging this into \eqref{pointwise} gives us

$$\rho_{t}^{-1/N'}(\gamma_{t})\geq\tau_{K,N}^{(1-t)}(\ell(\gamma_{0},\gamma_{1}))\rho_{0}^{-\frac{1}{N'}}(\gamma_{0})+\tau_{K,N}^{(t)}(\ell(\gamma_{0},\gamma_{1}))\rho_{1}^{-\frac{1}{N'}}(\gamma_{1})$$

\noindent for $\nu_{\alpha}$-a.e. $\gamma.$ Since for $\q_{0}$-a.e. $\alpha$
then inequality above holds for $\nu_{\alpha}$-a.e. $\gamma$, a
fortiori it holds true for $\nu$-a.e. $\gamma$; by integrating against $\nu$ the claim
is proved.\hfill $\square$

This subsection shows that $TCD_{rLip}^1(K,N)\implies TCD_1(K,N).$ Combined with the previous subsection, we conclude that $$TCD_{rLip}^1(K,N)\iff TCD_1(K,N).$$

\appendix
\section{Appendix}

We now prove Theorem \ref{Brenier Map}. The theorem extends Braun's result \cite[Theorem 4.16, 4.17]{BraunTCD}, whereby we now  require only the forward $p$-enb assumption, instead of the full $p$-enb assumption. The proofs are inspired by   the results of Kell \cite{Kell2017} and Cavalletti--Huesmann \cite{CavallettiHuesmann2015}.

Throughout, we let the spacetime $(X,d,\m, \ell)$ be a gh LLS, and fix $p\in(0,1).$ We start with a preliminary lemma.

\begin{lemma}[Forward non-mixing]\label{Forward non-mixing} Assume the spacetime is forward $p$-enb, and let $\boldsymbol{\pi}^i\in OptTGeo_p(\mu_0^i,\mu_1^i)$ for $i=1,2$ be two $\ell^p$-optimal dynamical plans such that $\pi^i:=(e_0,e_1)_\# \boldsymbol{\pi}^i$ are both concentrated on a common $\ell^p$-cyclically monotone set $\Gamma\subset K_0\times K_1\subset X_\ll^2$, where $K_0,K_1$ are compact. Assume that $$\mu_1^1\perp \mu_1^2,\quad (e_t)_\# \boldsymbol{\pi}^i\ll\m\quad \forall t\in [0,1), i=1,2.$$ Then $(e_t)_\# \boldsymbol{\pi}^1 \perp (e_t)_\# \boldsymbol{\pi}^2$ for all $ t\in(0,1).$
\end{lemma}

\begin{proof}
Suppose by contradiction that for some $t\in (0,1)$ we have $\nu_t^1:=(e_t)_\# \boldsymbol{\pi}^1 \not\perp (e_t)_\# \boldsymbol{\pi}^2:=\nu_t^2$, where by assumption we have $\nu_t^i=\rho_t^i\m$. We will construct a new $\ell^p$-optimal dynamical plan which follows the same geodesics up to time $t$ but splits into two disjoint families at $s>t$, where restricting this plan to $[0,s]$ will contradict the forward $p$-enb assumption.

Set  $\lambda:=\min\{ \rho_t^1,\rho_t^2\}\m$, and $c=\lambda(X)>0,$  and define the probability measure $\nu:=\frac{\lambda}{c},$ which by construction satisfies $\nu\leq \frac{\nu_t^i}{c}$ for $i=1,2.$ Disintegrating $\boldsymbol{\pi}^i$ with respect to $e_t$  as in \cite[Appendix B.4]{BraunTCD}, we obtain $$\boldsymbol{\pi}^i =\int_X\boldsymbol{\pi}_z^i\nu_t^i (dz),$$ where $\boldsymbol{\pi}^i_z$ is concentrated on geodesics passing through $z$ at time $t,$ i.e. $\boldsymbol{\pi}^i_z(\{ \g:\g_t =z \})=1$ for $\nu_t^i$-a.e. $z.$

Since $\nu\ll\nu_t^i$, the conditional probabilities above are also well defined for $\nu$-a.e. $z$, and we can define $\boldsymbol{q}$ a probability measure on $TGeo^\ell(X)^2$ $$\boldsymbol{q}:=\int_X \boldsymbol{\pi}^1_z\otimes\boldsymbol{\pi}^2_z\nu(dz),$$ where $\boldsymbol{q}(\{ (\g^1,\g^2):\g_t^1=\g_t^2\})=1$. In particular, its projections satisfy $$(P_i)_\#\boldsymbol{q}=\int_X \boldsymbol{\pi}^i_z\nu(dz)=\frac{d\nu}{d\nu_t^i}(e_t)\boldsymbol{\pi}^i,\quad i=1,2,$$ and combining with $\nu \leq \frac{\nu_t^i}{c}$ we get \begin{equation}\label{restricted a.c.}
(P_i)_\# \boldsymbol{q}\leq \frac{\boldsymbol{\pi}^i}{c},\quad i=1,2.\end{equation} 
Both projected subplans have absolutely continuous marginals at every time $r<1$, and moreover because the terminal marginals of $\boldsymbol{\pi}^1$ and $\boldsymbol{\pi}^2$ are mutually singular, we have $$\boldsymbol{q}(\g_1^1 \neq \g_1^2)=1,$$ and hence the geodesics selected by $\boldsymbol{q}$ meet at time $t$ but have different subsequent values. By continuity of timelike geodesics, there exists $s\in (t,1)$ such that $\boldsymbol{q}(\g_s^1\neq \g_s^2)>0$, where by separability of $X$ we can choose disjoint Borel $A,B\subset X$ such that $$\boldsymbol{q}(Q)>0,\quad Q:=\{ (\g^1,\g^2):\g_s^1\in A,\g_s^2\in B\}.$$ Denoting the normalization as $\bar{\boldsymbol{q}}$, which retains absolute continuity from \eqref{restricted a.c.} which we write as \begin{equation}\label{normalized q} \bar{\boldsymbol{q}}:=\frac{\boldsymbol{q}\restrict_Q}{\boldsymbol{q}(Q)},\quad (P_i)_\# \bar{\boldsymbol{q}}\leq \frac{\boldsymbol{\pi}^i}{c\boldsymbol{q}(Q)},
    \end{equation}
where for $\bar{\boldsymbol{q}}$-a.e. pair, the two geodesics agree at time $t$, but their values at time $s$ now belong to disjoint sets $A,B$.

We now show that the future of $\g^2$ can be attached to a past of $\g^1$ without losing optimality. For $\bar{\boldsymbol{q}}$-a.e. pair, set $L_i:=\ell(\g_0^i,\g_1^i)$ and note that this endpoint pairs belong to the same $\ell^p$-cyclically monotone relation $\Gamma$. The reverse triangle inequality and $\ell^p$-cyclical monotonicity gives \begin{equation}\label{reverse ell^p triangle}
L_1^p+L_2^p\geq \ell(\g_0^1,\g_1^2)^p+\ell(\g_0^2,\g_1^1)^p\geq [tL_1+(1-t)L_2]^p+[tL_2+(1-t)L_1]^p\geq L_1^p+L_2^p,
\end{equation}
and since $p\in(0,1),$ by equality and strict concavity we get that $L_1=L_2$, and equality throughout yields  $\ell(\g_0^1,\g_1^2)=\ell(\g_0^2,\g_1^1)$. Consequently, the cross concatenated geodesic $$\tilde\g_r:=\begin{cases}
    \g_r^1,\quad 0\leq r\leq t,\\
    \g_r^2,\quad t<r\leq 1,
\end{cases}$$
is again a maximizing timelike geodesic, and equality in \eqref{reverse ell^p triangle} shows that $\Gamma\cup \{(\g_0^1,\g_1^2),(\g_0^2,\g_1^1) \}$ is still $\ell^p$-cyclically monotone.

Define $$\boldsymbol{\eta}:=\frac{1}{2}(P_1)_\#\bar{\boldsymbol{q}}+\frac{1}{2} S_\# \bar{\boldsymbol{q}},$$ where $S(\g^1,\g^2):=\tilde \g.$ This plan is an $\ell^p$-optimal dynamical plan by construction. Restricting $\boldsymbol{\eta}$ to $[0,s]$ and reparametrizing onto $[0,1]$, denote the resulting measure $\bar{\boldsymbol{\eta}}$, which is again an $\ell^p$-optimal dynamical plan whose first marginal satisfies $(e_0)_\#\bar{\boldsymbol{\eta}}\ll\m$ by \eqref{normalized q}, and whose second marginal is $$(e_1)_\# \bar{\boldsymbol{\eta}}=\frac{1}{2}(e_s)_\# (P_1)_\#\bar{ \boldsymbol{q}}+\frac{1}{2}(e_s)_\# (P_2)_\#\bar{ \boldsymbol{q}}\ll\m.$$ Thus, the initial and final measures are compactly supported and timelike $p$-dualisable, and so the forward $p$-enb assumption applies to $\bar{\boldsymbol{\eta}}$. But this leads to a contradiction, since for $\boldsymbol{q}$-a.e. pair, the two curves selected by $\bar{\boldsymbol{\eta}}$ agree on $[0,t/s]$ while their terminal points belong to the disjoint sets $A$ and $B$. 

\end{proof}

We now prove Theorem \ref{Brenier Map}. The proof follows Section 4 from Braun \cite{BraunTCD}, where Lemma \ref{Forward non-mixing} replaces Braun's Lemma 2.23, allowing us to establish mutual singularity of intermediate marginals under the weaker assumption of forward $p$-enb. 

\begin{theorem}\label{A:Brenier Map}
    Set $K\in \R, N\in [1,\infty),  p\in(0,1),$ and let $(X,d,\m,\ell)$ be a forward $p$-enb gh LLS that satisfies the $TMCP^+(K,N)$ condition. Let $\mu_0$ be absolutely continuous and let $(\mu_0,\mu_1)\in \P_c(X)^2$ be timelike $p$-dualisable. Then, there exists a unique $\ell^p$-optimal coupling $\Pi_{\ll}^{p\textrm{-opt}}(\mu_0,\mu_1)$ and moreover there exists a $\mu_0$-measurable map $T:X\rightarrow X$ such that $\pi = (Id,T)_\#\mu_0$ and $$\ell_p(\mu_0,\mu_1)^p=\int \ell(x,T(x))^pd\mu_0(x).$$ Additionally, there exists a unique chronological timelike $\ell^p$ optimal dynamical plan $\boldsymbol{\alpha} \in OptTGeo_p(\mu_0,\mu_1)$ that is induced by a $\mu_0$-measurable map into $TGeo(X)$, such that $(e_t)_\# \boldsymbol{\alpha}\ll\m$ for $t\in[0,1).$
    \end{theorem}
\begin{proof}

We first prove the uniqueness of the $\ell^p$ optimal couplings for finite Dirac targets, which we generalize to general target measures. We then prove the uniqueness of the optimal dynamical plan. 

\textbf{Step 1: Finite targets.} Let $\mu_0 = \rho_0\m\in\P_c(X)$ and suppose that $\mu_1$ is a finite probability measure $\mu_1 = \sum_{i=1}^n\lambda_i\delta_{z_i}$. Let $\pi \in \Pi^{p\textrm{-opt}}_\ll (\mu_0,\mu_1)$. We claim that $\pi$ is induced by a map. 

Suppose for a contradiction that it is not, then disintegrating $\pi$ with respect to the first marginal, we may write $$\pi = \int_X\delta_x\otimes\pi_x\mu_0(dx),$$ and since $\pi$ is not induced by a map, the set of $x$ such that $\pi_x$ is not a Dirac mass has $\mu_0$-positive measure. Since the target is finite sum of Diracs, there exist two distinct $z_1,z_2$ such that \begin{equation}\label{mu_0 first}
\mu_0(\{ x:\pi_x(\{z_1 \})>0, \pi_x(\{ z_2\})>0)\}>0.\end{equation}

Let $\Gamma\subset X_\ll^2$ be a Borel $\ell^p$-cyclically monotone set on which $\pi$ is concentrated. By disintegration, its complement $\Gamma^c$ satisfies $$\pi(\Gamma^c)=\int_X\pi_x(\Gamma_x^c)\mu_0(dx)=0,$$ and thus for $\mu_0$-a.e. $x$, every which has  positive $\pi_x$-mass belongs to $\Gamma_x$. In particular, \eqref{mu_0 first} becomes $$\mu_0(\{ x:(x,z_1),(x,z_2)\in \Gamma\})>0.$$ Restricting to a region where $0<\eps\leq \rho_0\leq \eps^{-1}$, by inner regularity we get a compact set $E\subset X$ such that \begin{equation}\label{positive compact E}
0<\m(E)<\infty,\quad E\times \{ z_1,z_2 \}\subset \Gamma\subset X_\ll^2.
\end{equation} Replacing now the source measure $\mu_0$ by the uniform normalized measure on $E$, which we denote by $\nu_0:=\frac{\m\restrict_E}{\m(E)}$ we get that the two couplings $$\pi^i:=\nu_0\otimes \delta_{z_i},\quad i=1,2,$$ are concentrated on the same $\ell^p$-cyclically monotone set, and thus $\pi^i$ and their average are all $\ell^p$-optimal.

By the $TMCP^+(K,N)$ condition and Braun \cite[Theorem 4.13]{BraunTCD}, there exists an $\ell^p$-optimal dynamical plans     $\boldsymbol{\alpha}^i\in OptTGeo_p(\nu_0,\delta_{z_i})$ with $(e_t)_\#\boldsymbol{\alpha}^i=\rho_t^i\m$ for $t\in [0,1)$ which satisfy 

\begin{equation}\label{sup bound rho_t}
\| \rho_t^i\|_{L^\infty(\m)} \leq \frac{1}{(1-t)^N}e^{tD\sqrt{K^-N}}\frac{1}{\m(E)},
\end{equation}
where $D:=\sup_{x\in E,i=1,2}\ell(x,z_i)$.

At this point we use Lemma \ref{Forward non-mixing}: since $\delta_{z_1},\delta_{z_2}$ are mutually singular, the intermediate measures are absolutely continuous by \eqref{sup bound rho_t}, and the  endpoint couplings share the same $\ell^p$-cyclically monotone relation.  Therefore, by Lemma \ref{Forward non-mixing}, we get $$(e_t)_\#\boldsymbol{\alpha}^1\perp (e_t)_\# \boldsymbol{\alpha}^2,\quad t\in (0,1),$$ 
and the two intermediate marginals are mutually singular, which combined with \eqref{sup bound rho_t} gives us $$(e_t)_\# \boldsymbol{\alpha}^1+(e_t)_\# \boldsymbol{\alpha}^2\leq \frac{C_D(t)}{\m(E)}\m,$$ where $C_D(t):=(1-t)^{-N}e^{tD\sqrt{K^-N}}\rightarrow 1$ as $t\downarrow 0$. Since $(e_t)_\#\boldsymbol{\alpha}^i\weak\nu_0$ as $t\downarrow 0$, passing to the limit yields $$2\nu_0\leq \frac{\m}{\m(E)},$$ where we recall that$\nu_0 = \m\restrict_E/\m(E)$. Restricting both sides to $E$ gives the contradiction $2\leq 1$. Hence $\pi$ is induced by a map, and uniqueness follows by averaging two optimal couplings.

\textbf{Step 2: General compactly supported targets.} This step follows \cite[Lemma 4.15]{BraunTCD}. Let $\mu_0=\rho_0\m$ with $\rho_0\in L^\infty(\m)$, and let $\mu_1\in \P_c(X)$ satisfy $$\spt \mu_0\times \spt \mu_1\subset X_\ll^2.$$ Approximate $\mu_1$ by finite probability measures as $$\mu_1^n:=\sum_{i=1}^{k_n}\lambda_i^n\delta_{z_i^n}\weak \mu_1,$$ with $z_i^n\in \spt\mu_1.$ By Step 1, the unique optimal coupling from $\mu_0$ to $\mu_1^n$ is induced by a map $T_n$. Set $$E_i^n:=T_n^{-1}(\{z_i^n\}),\quad \nu_0^{n,i}:=\frac{\mu_0\restrict_{E_i^n}}{\lambda_i^n}.$$ This gives us $\mu_0 = \sum_i\lambda_i^n\nu_0^{n,i}$ with $\lambda_i^n\rho_0^{n,i}=\rho_0 1_{E_i^n}$. Theorem 4.13 from \cite{BraunTCD} gives existence of optimal dynamical plans $\boldsymbol{\alpha}^{n,i}$ from $\nu_0^{n,i}$ to $\delta_{z_i^n}$ which satisfy $$\| \rho_t^{n,i}\|_{L^\infty(\m)}\leq C_D(t)\| \rho_0^{n,i}\|_{L^\infty(\m)},\quad t<1.$$ Since these couplings are restrictions of the same optimal coupling and have mutually singular terminal measures, we get by Lemma \ref{Forward non-mixing} $$(e_t)_\#\boldsymbol{\alpha}^{n,i}\perp (e_t)_\# \boldsymbol{\alpha}^{n,j},\quad i\neq j,\,\, 0<t<1,$$ and consequently, the optimal dynamical plan $$\boldsymbol{\alpha}_n:=\sum_i\lambda_i^n\boldsymbol{\alpha}^{n,i}$$ satisfies $$\| \rho_t^n\|_\infty=\max_i\lambda_i^n\| \rho_t^{n,i}\|_\infty \leq C_D(t)\max_i\lambda_i^n\|\rho_0^{n,i}\|_\infty\leq C_D(t)\|\rho_0\|_\infty.$$ Since this bound is independent of $n$, by global hyperbolicity, compactness and stability argument of Lemma 4.15 in \cite{BraunTCD}, a subsequence converges to an optimal dynamical plan $\boldsymbol{\alpha}$ from $\mu_0$ to $\mu_1$ and satisfies \begin{equation}\label{Step 2 sup estimate}
(e_t)_\#\boldsymbol{\alpha}=\rho_t\m,\quad \| \rho_t\|_\infty \leq C_D(t)\| \rho_0\|_\infty,\quad t\in [0,1).\end{equation}

\textbf{Step 3: General targets.} Let $\pi\in\Pi_\ll^{p\textrm{-opt}}(\mu_0,\mu_1)$ with $\mu_0\ll \m$, and let $\Gamma\subset X_\ll^2$ be an $\ell^p$-cyclically monotone set with $\pi(\Gamma) = 1.$ We claim that this $\pi$ is induced by a map. Suppose otherwise, and consider the set $$E:=\{ x\in X:\#\Gamma_x\geq 2\},\quad \Gamma_x:=\{ y:(x,y)\in \Gamma\}$$ which now has positive $\mu_0$-measure. By the von Neumann selection theorem, there exist measurable maps $T_1, T_2:E\rightarrow X$ such that $$(x,T_i(x))\in \Gamma,\quad T_1(x)\neq T_2(x),\quad i=1,2.$$ By Lusin's theorem, inner regularity and the fact that $X_\ll^2$ is open, we may restrict to a compact set $C\subset E$ with $0<\m(C)<\infty$ such that $T_1,T_2$ are continuous on $C,$ and their images are contained in disjoint compact sets $B_1,B_2$ with $C\times(B_1\cup B_2)\subset X_\ll^2$. Define $$\nu_0:=\frac{\m\restrict_C}{\m(C)},\quad \nu_1^i:=(T_i)_\# \nu_0,\quad i=1,2.$$ The couplings $(Id,T_i)_\#\nu_0$ are both concentrated on $\Gamma,$ and hence they are jointly $\ell^p$-optimal, where by construction we also have $\nu_1^1\perp\nu_1^2$. By step 2, there exist optimal dynamical plans $\boldsymbol{\alpha}^i\in OptTGeo_p(\nu_0,\nu_1^i)$ such that $(e_t)_\#\boldsymbol{\alpha}^i = \rho_t^i \m$ with $\|\rho_t^i\|_\infty \leq C_D(t)\m(C)^{-1}$ for $t\in[0,1)$, with $D:=\sup_{(x,y)\in C\times (B_1\cup B_2)}\ell(x,y).$ Define $$\pi^i:=(Id,T_i)_\#\nu_0,\quad \eta^i :=(e_0,e_1)_\# \boldsymbol{\alpha}^i,\quad i=1,2$$ and since $\pi^1,\pi^2$ are both concentrated on $\Gamma$, their average $\bar \pi :=\frac{1}{2}(\pi^1+\pi^2)$ is $\ell^p$-optimal between $\nu_0$ and $\frac{1}{2}(\nu_1^1+\nu_1^2).$ Furthermore, since $\eta^i$ and $\pi^i$ are both optimal with the same marginals for the $\ell^p$ problem, we get that $$\int_{X\times X}\ell^pd\bigg(\frac{1}{2}(\eta^1+\eta^2)\bigg)=\int_{X\times X}\ell^pd\bar \pi.$$ Thus, $\frac{1}{2}(\eta^1+\eta^2)$ is optimal and it is concentrated on some $\ell^p$-cyclically monotone set which we denote by  $\tilde \Gamma$, which may be chosen inside of $C\times(B_1\cup B_2).$ Thus, since $\nu_1^1\perp\nu^2_1$, we can apply Lemma \ref{Forward non-mixing} to conclude that  $$(e_t)_\#\boldsymbol{\alpha}^1\perp (e_t)_\# \boldsymbol{\alpha}^2,\quad t\in (0,1).$$ Consequently, we get  $$(e_t)_\#\boldsymbol{\alpha}^1+ (e_t)_\# \boldsymbol{\alpha}^2\leq \frac{C_D(t)}{\m(C)}\m,$$ and sending $t\downarrow 0$ we get $2\nu_0\leq \frac{\m}{\m(C)}$, whereby restricting to $C$ gives the contradiction $2\leq 1$. Thus, $\pi$ is induced by a $\mu_0$-measurable map $T$, and uniqueness follows again by averaging two optimal couplings.

\textbf{Step 4: Uniqueness of the optimal dynamical plan.} We follow Theorem 4.17 in \cite{BraunTCD}. Let $\boldsymbol{\alpha}\in OptTGeo_p(\mu_0,\mu_1)$, where for every $t\in (0,1)$, we get that by restriction that $$(e_0,e_t)_\# \boldsymbol{\alpha}\in \Pi_\ll^{p\textrm{-opt}}(\mu_0,(e_t)_\# \boldsymbol{\alpha}).$$ Since $\mu_0\ll \m$, by Step 3 we know that this coupling is induced by a map. Disintegrating $\boldsymbol{\alpha}=\int_X \boldsymbol{\alpha}_x\mu_0(dx),$ we conclude that $(e_t)_\#\boldsymbol{\alpha}_x$ is a Dirac for $\mu_0$-a.e. $x$, and applying this to every rational $t\in (0,1)$ we get by continuity of geodesics that $\boldsymbol{\alpha}_x$ itself is a Dirac for $\mu_0$-a.e. $x$, and thus there exists a $\mu_0$-measurable map $G:X\rightarrow TGeo(X)$ such that $$\boldsymbol{\alpha}=G_\# \mu_0.$$ Uniqueness follows by averaging two optimal dynamical plans, where their average must also be optimal and thus is induced by a geodesic map above.

We now prove the absolute continuity of intermediate marginals. Suppose first that $\mu_0=\rho_0\m$ with $\rho_0\in L^{\infty}(\m)$ and $\spt \mu_0\times \spt \mu_1\subset X_\ll^2$. By Step 2, there exists a chronological optimal dynamical plan whose intermediate marginals are absolutely continuous and satisfy the $L^\infty$ estimate \eqref{Step 2 sup estimate}, and since the optimal dynamical plan is unique the claim follows. The general case follows by a localization and exhaustion argument as in \cite[Proposition 3.38]{BraunTCD} and \cite[Theorem 5.12]{octet+}.
\end{proof}

We now  adapt the arguments of Kell \cite{Kell2017} and Rajala-Sturm \cite{RajalaSturm2014} to show that forward $p$-enb $TCD_p(K,N)$ spacetimes are also backward $p$-enb.

\begin{theorem}[Timelike Nonbranching Reversibility]\label{forward p-enb implies backward p-enb}
    Let $p\in (0,1),$ $K\in \R,$ $N\in [1,\infty)$, and let $(X,d,\m,\ell)$ be a forward $p$-enb gh LLS satisfying the $TCD_p(K,N)$ condition. Then $(X,d,\m,\ell)$ is also backwards $p$-enb.
\end{theorem}
\begin{proof}
    Since $TCD_p(K,N)$ implies $TMCP^+(K,N),$ the above Theorem \ref{A:Brenier Map} applies. In particular, for every compactly supported timelike $p$-dualisable pair $(\mu_0,\mu_1)$ with $\mu_0\ll \m$, there exists a unique chronological $p$-optimal dynamical plan $\boldsymbol{\alpha}$ with $(e_t)_\#\boldsymbol{\alpha}\ll \m$ for $t\in[0,1).$ Furthermore, if $\boldsymbol{\beta}\ll \boldsymbol{\alpha}$ is any normalized dynamical subplan, then $(e_t)_\#\boldsymbol{\beta}\ll(e_t)_\#\boldsymbol{\alpha}\ll\m$ for $t\in [0,1).$

    Suppose for a contradiction that $X$ is not backward $p$-enb. Then there exist compactly supported absolutely continuous  timelike $p$-dualisable pair $(\mu_0,\mu_1)$ such that their unique chronological $p$-optimal dynamical plan $\boldsymbol{\alpha}$ is not concentrated on a backward nonbranching set. By a measurable selection argument, there exists an $r>0$ such that for every $\epsilon>0$, we can find rational times $r<a<b<1-r$ with $b-a<\epsilon$, and two normalized dynamical subplans $\boldsymbol{\alpha}^L,\boldsymbol{\alpha}^R\ll \alpha$ where, denoting $\mu_t^i:=(e_t)_\#\boldsymbol{\alpha}^i\ll \m$ for $i\in \{L,R\}$, and $t\in[0,1]$ we have $$\mu_a^L\perp\mu_a^R,\quad \mu_t^L=\mu_t^R\quad \forall t\in[b,1].$$ Indeed, let $R_b(\g):=\g\restrict_{[b,1]}$, and disintegrate: $$\boldsymbol{\alpha}=\int \boldsymbol{\alpha}_z^b \nu_b(dz),\quad \nu_b = (R_b)_\#\boldsymbol{\alpha}.$$ Since $\boldsymbol{\alpha}$ is not concentrated on a backward non-branching set, a countable rational-time argument and continuity of geodesics show that, for arbitrarily close rational $a<b$  in some fixed interval $(r,1-r)$, the measures $(e_a)_\#\boldsymbol{\alpha}_z^b$ are non-Dirac on a set of positive $\nu_b$-measure, since otherwise by taking intersections of the corresponding full-measure sets would give a backward non-branching family of full $\boldsymbol{\alpha}$-measure. By separability, we may choose disjoint Borel sets $B_L,B_R$ such that $$w_i(z) :=\boldsymbol{\alpha}^b_z\{ \g:\g_a\in B_i\}>0,\quad i=L,R$$ happens simultaneously on a set of positive $\nu_b$-measure. Set $w:=\min\{ w_L,w_R\}$, and $c=\int wd\nu_b>0$, and define $$d\boldsymbol{\alpha}^i(\g)=\frac{1_{\{ \g\in B_i\}}}{c}\frac{w(R_b\g)}{w_i(R_b\g)}d\boldsymbol{\alpha}(\g),$$ which we set to zero when the denominator is zero. Then we get $$\boldsymbol{\alpha^i}\leq \frac{\boldsymbol{\alpha}}{c},\quad \mu_a^L\perp \mu_a^R,\quad (R_b)_\# \boldsymbol{\alpha}^L=(R_b)_\# \boldsymbol{\alpha}^R = \frac{w}{c}\nu_b,$$ and in particular we get $\mu_t^L=\mu_t^R$ for all $t\in [b,1]$ as required.
    
    Since $\boldsymbol{\alpha}^i\ll \boldsymbol{\alpha}, $ the intermediate marginals of $\boldsymbol{\alpha}$ are absolutely continuous  for $t<1$, and the initial and final marginal is absolutely continuous by assumption. By construction, the restrictions of $\boldsymbol{\alpha}^L,\boldsymbol{\alpha}^R,$ as well as their average are chronological $p$-optimal dynamical plans between their respective initial and final marginals, and moreover these initial and final marginals form  timelike $p$-dualisable pairs. Hence, by Theorem \ref{A:Brenier Map}, each of these restricted transport problems admit a unique chronological $p$-optimal dynamical plan. The $TCD_p(K,N)$ condition gives existence of the optimal dynamical plans satisfying the corresponding R\'enyi entropy inequality for the same initial and final marginals, and uniqueness implies that the entropy inequality holds along these restricted plans themselves.

We now apply the $TCD_p(K,N)$ entropy inequality. Fix $N'\geq N$ with $1<N'<\infty$,  and let $\mu_b :=\mu_b^L=\mu_b^R$, and define $$A_a :=\frac{1}{2}(S_{N'}(\mu_a^L)+S_{N'}(\mu_a^R)), \quad B_b:=S_{N'}(\mu_b).$$ Since the measure are compactly supported and absolutely continuous, we have that $A_a,B_b\in (-\infty,0)$. Denote the averaged dynamical plan as $\bar{\boldsymbol{\alpha}}:=\frac{1}{2}(\boldsymbol{\alpha}^L+\boldsymbol{\alpha}^R)$, with $\bar\mu_t:=(e_t)_\#\bar{\boldsymbol{\alpha}}=\frac{1}{2}(\mu_t^L+\mu_t^R),$ where $\mu_a^L\perp \mu_a^R$ implies \begin{equation}\label{factor 2 entropy}
S_{N'}(\bar \mu_a)=-\int_X\bigg(\frac{\rho_a^L+\rho_a^R}{2}\bigg)^{1-1/N'}d\m = 2^{1/N'}\frac{1}{2}(S_{N'}(\mu_a^L)+S_{N'}(\mu_a^R))=2^{1/N'}A_a.
\end{equation}

We will apply the $TCD_p(K,N)$ inequality to these restricted and averaged plans.  First note  that since $r<a<b<1-r$, we fix $h\in(0,r/2)$ and set $$s:=a-h,\quad u = b+h,\quad \delta:=b-a,\quad \lambda_\delta :=\frac{h}{h+\delta},$$ where are chosen since they satisfy $\frac{a-s}{b-s}=\frac{u-b}{u-a}=\lambda_\delta\rightarrow 1$ as $\delta\downarrow 0$. Since $\boldsymbol{\alpha}$ has compactly supported endpoints, there exists a $D<\infty$ such that $\ell(\g_0,\g_1)\leq D$ for $\boldsymbol{\alpha}$-a.e. $\g$, and in particular $\ell(\g_s,\g_b),\ell(\g_a,\g_u)\leq (h+\delta)D$. Recalling the standard bound on the distortion coefficients $$\tau_{K,N'}^{(\lambda)}(\theta)\geq \lambda \exp\big(-(1-\lambda)\theta\sqrt{K^-/N'}\big),$$ where $K^-:=\max\{-K,0 \},$ by choosing $c_\delta:=\lambda_\delta \exp\big(-\delta D\sqrt{K^-/N'}\big)$, we get the bounds \begin{equation}\label{tau bounds}
\tau_{K,N'}^{(\lambda_\delta)}(\ell(\g_s,\g_b))\geq c_\delta,\quad \tau_{K,N'}^{(\lambda_\delta)}(\ell(\g_a,\g_u))\geq c_\delta,\end{equation}
where in particular $c_\delta \rightarrow 1$ as $\delta\downarrow0$. For $i\in \{ L,R\}$, the $TCD_p(K,N)$ entropy inequality applied to the restriction of $\boldsymbol{\alpha}^i$ to the interval $[s,b]$, with, where we take $t=a$ as the intermediate time and $\lambda_\delta=\frac{a-s}{b-s}$ reads as $$S_{N'}(\mu_a^i)\leq -\int \bigg[\tau_{K,N'}^{(1-\lambda_\delta)}(\ell(\g_s,\g_b))(\rho_s^i(\g_s))^{-1/N'}+\tau_{K,N'}^{(\lambda_\delta)}(\ell(\g_s,\g_b))(\rho_b(\g_b)^{-1/N'}\bigg]d\boldsymbol{\alpha}^i(\g).$$ Keeping only the final marginal, and by \eqref{tau bounds} we get \begin{align*}
S_{N'}(\mu_a^i) &\leq -\int \tau_{K,N'}^{(\lambda_\delta)}(\ell(\g_s,\g_b))(\rho_b(\g_b)^{-1/N'}d\boldsymbol{\alpha}^i(\g)\\ 
&\leq -c_\delta \int(\rho_b(\g_b))^{-1/N'}d\boldsymbol{\alpha}^i(\g) \\
&=-c_\delta \int_X\rho_b^{1-1/N'}d\m \\
&=c_\delta S_{N'}(\mu_b),
\end{align*}
where we have used that $(e_b)_\# \boldsymbol{\alpha}^i=\mu_b$ for $i=L,R$. Averaging over $i=L,R$ we obtain \begin{equation}\label{First entropy inequality}
    A_a\leq c_\delta B_b.
\end{equation}

Applying now the $TCD_p(K,N)$ inequality to $\bar{\boldsymbol{\alpha}}$ restricted to $[a,u]$, where we take the intermediate time $t=b$ with the interpolation parameter $1-\lambda_\delta=\frac{b-a}{u-a}$, we get 
$$S_{N'}(\bar\mu_b)\leq -\int \bigg[\tau_{K,N'}^{(\lambda_\delta)}(\ell(\g_a,\g_u))(\bar \rho_a(\g_a))^{-1/N'}+\tau_{K,N'}^{(1-\lambda_\delta)}(\ell(\g_a,\g_u))(\bar\rho_u(\g_u)^{-1/N'}\bigg]d\bar{\boldsymbol{\alpha}}(\g).$$ Keeping only the initial marginal, recalling \eqref{tau bounds} and that $\bar\mu_b=\mu_b$, we get \begin{align*}
S_{N'}(\mu_b) &\leq -\int \tau_{K,N'}^{(\lambda_\delta)}(\ell(\g_a,\g_u))(\bar\rho_a(\g_a)^{-1/N'}d\bar{\boldsymbol{\alpha}}(\g)\\ 
&\leq -c_\delta \int(\bar\rho_a(\g_a))^{-1/N'}d\bar{\boldsymbol{\alpha}}(\g) \\
&=-c_\delta \int_X\bar\rho_a^{1-1/N'}d\m \\
&=c_\delta S_{N'}(\bar \mu_a)\\ &=2^{1/N'}c_\delta A_a,
\end{align*} where in the last line we have used the \eqref{factor 2 entropy}. Therefore $B_b\leq 2^{1/N'}c_\delta A_a,$ and combining this with \eqref{First entropy inequality} we get $$A_a\leq c_\delta B_b\leq 2^{1/N'}c_\delta^2A_a,$$ and dividing by $A_a$ reverses the inequality since $A_a<0$, giving $2^{1/N'}c_\delta^2\leq 1.$ By construction of the two subplans, we can choose $a<b$ with $\delta = b-a$ arbitrarily small and $r,h,D$ fixed.  Sending $\delta\downarrow0$ and recalling that $c_\delta\rightarrow 1$, we  get $2^{1/N'}\leq 1$, which is impossible since $N'<\infty$. This contradiction shows that $\boldsymbol{\alpha}$ must be concentrated on a backward nonbranching set, and hence $X$ is in fact backward $p$-enb.
    
\end{proof}

\printbibliography[heading=none]

@article {KunzingerSaemann2018,
    AUTHOR = {Kunzinger, Michael and S\"amann, Clemens},
     TITLE = {Lorentzian length spaces},
   JOURNAL = {Ann. Global Anal. Geom.},
  FJOURNAL = {Annals of Global Analysis and Geometry},
    VOLUME = {54},
      YEAR = {2018},
    NUMBER = {3},
     PAGES = {399--447},
      ISSN = {0232-704X,1572-9060},
   MRCLASS = {53C23 (53B30 53C50 53C80)},
  MRNUMBER = {3867652},
MRREVIEWER = {Benjam\'in\ Olea},
       DOI = {10.1007/s10455-018-9633-1},
       URL = {https://doi.org/10.1007/s10455-018-9633-1},
}

@book{Wald1984,
  author    = {Wald, Robert M.},
  title     = {General Relativity},
  publisher = {University of Chicago Press},
  address   = {Chicago},
  year      = {1984}
}

@article{CavallettiMondino2020,
	author  = {Cavalletti, Fabio and Mondino, Andrea},
	title   = {New formulas for the Laplacian of distance functions and applications},
	journal = {Analysis \& PDE},
	volume  = {13},
	number  = {7},
	pages   = {2091--2148},
	year    = {2020},
	doi     = {10.2140/apde.2020.13.2091},
	eprint  = {1803.09687},
	archivePrefix = {arXiv}
}

@article {CavallettiGigliSantarcangelo2021,
    AUTHOR = {Cavalletti, Fabio and Gigli, Nicola and Santarcangelo, Flavia},
     TITLE = {Displacement convexity of entropy and the distance cost
              optimal transportation},
   JOURNAL = {Ann. Fac. Sci. Toulouse Math. (6)},
  FJOURNAL = {Annales de la Facult\'e{} des Sciences de Toulouse.
              Math\'ematiques. S\'erie 6},
    VOLUME = {30},
      YEAR = {2021},
    NUMBER = {2},
     PAGES = {411--427},
      ISSN = {0240-2963,2258-7519},
   MRCLASS = {53C23 (49Q22 53C21)},
  MRNUMBER = {4297384},
MRREVIEWER = {Daniele\ Semola},
       DOI = {10.5802/afst.1679},
       URL = {https://doi.org/10.5802/afst.1679},
}

@article {McCann2020,
    AUTHOR = {McCann, Robert J.},
     TITLE = {Displacement convexity of {B}oltzmann's entropy characterizes
              the strong energy condition from general relativity},
   JOURNAL = {Camb. J. Math.},
  FJOURNAL = {Cambridge Journal of Mathematics},
    VOLUME = {8},
      YEAR = {2020},
    NUMBER = {3},
     PAGES = {609--681},
      ISSN = {2168-0930,2168-0949},
   MRCLASS = {53C50 (49Q22 53C21 58Z05 82C35 83C99)},
  MRNUMBER = {4192570},
       DOI = {10.4310/CJM.2020.v8.n3.a4},
       URL = {https://doi.org/10.4310/CJM.2020.v8.n3.a4},
}

@article {CavallettiMondino2024,
    AUTHOR = {Cavalletti, Fabio and Mondino, Andrea},
     TITLE = {Optimal transport in {L}orentzian synthetic spaces, synthetic
              timelike {R}icci curvature lower bounds and applications},
   JOURNAL = {Camb. J. Math.},
  FJOURNAL = {Cambridge Journal of Mathematics},
    VOLUME = {12},
      YEAR = {2024},
    NUMBER = {2},
     PAGES = {417--534},
      ISSN = {2168-0930,2168-0949},
   MRCLASS = {53C23 (49Q22 53C50 53C80 83C75)},
  MRNUMBER = {4779676},
}

@article {CaffarelliFeldmanMcCann2002,
    AUTHOR = {Caffarelli, Luis A. and Feldman, Mikhail and McCann, Robert
              J.},
     TITLE = {Constructing optimal maps for {M}onge's transport problem as a
              limit of strictly convex costs},
   JOURNAL = {J. Amer. Math. Soc.},
  FJOURNAL = {Journal of the American Mathematical Society},
    VOLUME = {15},
      YEAR = {2002},
    NUMBER = {1},
     PAGES = {1--26},
      ISSN = {0894-0347,1088-6834},
   MRCLASS = {49Q20 (58E17 90C48)},
  MRNUMBER = {1862796},
MRREVIEWER = {J.\ E.\ Brothers},
       DOI = {10.1090/S0894-0347-01-00376-9},
       URL = {https://doi.org/10.1090/S0894-0347-01-00376-9},
}

@book {Srivastava1998,
    AUTHOR = {Srivastava, S. M.},
     TITLE = {A course on {B}orel sets},
    SERIES = {Graduate Texts in Mathematics},
    VOLUME = {180},
 PUBLISHER = {Springer-Verlag, New York},
      YEAR = {1998},
     PAGES = {xvi+261},
      ISBN = {0-387-98412-7},
   MRCLASS = {04A15 (28A05 54-01 54H05)},
  MRNUMBER = {1619545},
MRREVIEWER = {Marek\ Balcerzak},
       DOI = {10.1007/978-3-642-85473-6},
       URL = {https://doi.org/10.1007/978-3-642-85473-6},
}

@article {Cavalletti2014,
    AUTHOR = {Cavalletti, Fabio},
     TITLE = {Monge problem in metric measure spaces with {R}iemannian
              curvature-dimension condition},
   JOURNAL = {Nonlinear Anal.},
  FJOURNAL = {Nonlinear Analysis. Theory, Methods \& Applications. An
              International Multidisciplinary Journal},
    VOLUME = {99},
      YEAR = {2014},
     PAGES = {136--151},
      ISSN = {0362-546X,1873-5215},
   MRCLASS = {49Q20 (53C23 58E30)},
  MRNUMBER = {3160530},
MRREVIEWER = {Luca\ Granieri},
       DOI = {10.1016/j.na.2013.12.008},
       URL = {https://doi.org/10.1016/j.na.2013.12.008},
}

@article {BianchiniCavalletti2013,
    AUTHOR = {Bianchini, Stefano and Cavalletti, Fabio},
     TITLE = {The {M}onge problem for distance cost in geodesic spaces},
   JOURNAL = {Comm. Math. Phys.},
  FJOURNAL = {Communications in Mathematical Physics},
    VOLUME = {318},
      YEAR = {2013},
    NUMBER = {3},
     PAGES = {615--673},
      ISSN = {0010-3616,1432-0916},
   MRCLASS = {58E35 (28A75 49Q20 58C35)},
  MRNUMBER = {3027581},
MRREVIEWER = {Luca\ Granieri},
       DOI = {10.1007/s00220-013-1663-8},
       URL = {https://doi.org/10.1007/s00220-013-1663-8},
}

@article {CavallettiMilman2021,
    AUTHOR = {Cavalletti, Fabio and Milman, Emanuel},
     TITLE = {The globalization theorem for the curvature-dimension
              condition},
   JOURNAL = {Invent. Math.},
  FJOURNAL = {Inventiones Mathematicae},
    VOLUME = {226},
      YEAR = {2021},
    NUMBER = {1},
     PAGES = {1--137},
      ISSN = {0020-9910,1432-1297},
   MRCLASS = {49Q22 (49Q20 53C23)},
  MRNUMBER = {4309491},
MRREVIEWER = {Luca\ Granieri},
       DOI = {10.1007/s00222-021-01040-6},
       URL = {https://doi.org/10.1007/s00222-021-01040-6},
}

@article {CavallettiMondino2017,
    AUTHOR = {Cavalletti, Fabio and Mondino, Andrea},
     TITLE = {Sharp and rigid isoperimetric inequalities in metric-measure
              spaces with lower {R}icci curvature bounds},
   JOURNAL = {Invent. Math.},
  FJOURNAL = {Inventiones Mathematicae},
    VOLUME = {208},
      YEAR = {2017},
    NUMBER = {3},
     PAGES = {803--849},
      ISSN = {0020-9910,1432-1297},
   MRCLASS = {53C23 (49Q05 53C21)},
  MRNUMBER = {3648975},
MRREVIEWER = {Renjin\ Jiang},
       DOI = {10.1007/s00222-016-0700-6},
       URL = {https://doi.org/10.1007/s00222-016-0700-6},
}

@article {McCann2024,
    AUTHOR = {McCann, Robert J.},
     TITLE = {A synthetic null energy condition},
   JOURNAL = {Comm. Math. Phys.},
  FJOURNAL = {Communications in Mathematical Physics},
    VOLUME = {405},
      YEAR = {2024},
    NUMBER = {2},
     PAGES = {Paper No. 38, 24},
      ISSN = {0010-3616,1432-0916},
   MRCLASS = {51K10 (53C50 83C75)},
  MRNUMBER = {4703452},
MRREVIEWER = {Barry\ Minemyer},
       DOI = {10.1007/s00220-023-04908-1},
       URL = {https://doi.org/10.1007/s00220-023-04908-1},
}

@article {LottVillani2009,
    AUTHOR = {Lott, John and Villani, C\'edric},
     TITLE = {Ricci curvature for metric-measure spaces via optimal
              transport},
   JOURNAL = {Ann. of Math. (2)},
  FJOURNAL = {Annals of Mathematics. Second Series},
    VOLUME = {169},
      YEAR = {2009},
    NUMBER = {3},
     PAGES = {903--991},
      ISSN = {0003-486X,1939-8980},
   MRCLASS = {53C23 (49Q15)},
  MRNUMBER = {2480619},
MRREVIEWER = {Alessio\ Figalli},
       DOI = {10.4007/annals.2009.169.903},
       URL = {https://doi.org/10.4007/annals.2009.169.903},
}

@article {SturmI,
    AUTHOR = {Sturm, Karl-Theodor},
     TITLE = { On the geometry of metric measure spaces. I},
   JOURNAL = {Acta Math.},
    VOLUME = {196},
      YEAR = {2006},
     PAGES = {133-177},
}

@misc{BraunMcCann+,
  author    = {Braun, Mathias and McCann, Robert J.},
  title     = {\textnormal{"Causal convergence conditions through variable timelike Ricci curvature bounds"}},
  howpublished = {To appear in \textit{Memoirs of the European Mathematical Society.}},
}

@unpublished{AkdemirTCDIndependence,
  author = {Akdemir, Afiny},
  note   = {In preparation},
  year   = {2026}
}

@article{HawkingPenrose1970,
  author  = {Hawking, Stephen W. and Penrose, Roger},
  title   = {The Singularities of Gravitational Collapse and Cosmology},
  journal = {Proceedings of the Royal Society of London.
             Series A, Mathematical and Physical Sciences},
  volume  = {314},
  number  = {1519},
  pages   = {529--548},
  year    = {1970},
  doi     = {10.1098/rspa.1970.0021}
}

@article{Hawking1966,
  author  = {Hawking, Stephen W.},
  title   = {The Occurrence of Singularities in Cosmology},
  journal = {Proceedings of the Royal Society of London.
             Series A, Mathematical and Physical Sciences},
  volume  = {294},
  number  = {1439},
  pages   = {511--521},
  year    = {1966},
  doi     = {10.1098/rspa.1966.0221}
}

@book{HawkingEllis1973,
  author    = {Hawking, Stephen W. and Ellis, George F. R.},
  title     = {The Large Scale Structure of Space-Time},
  publisher = {Cambridge University Press},
  year      = {1973},
  series    = {Cambridge Monographs on Mathematical Physics}
}

@article{Andersson-Galloway,
  author  = {Andersson, Lars and Galloway, Gregory J.},
  title   = {{dS/CFT} and spacetime topology},
  journal = {Advances in Theoretical and Mathematical Physics},
  volume  = {6},
  number  = {2},
  pages   = {307--327},
  year    = {2002},
  doi     = {10.4310/ATMP.2002.v6.n2.a4},
  eprint  = {hep-th/0202161},
  eprinttype = {arxiv}
}

@article{Borde,
  author  = {Borde, Arvind},
  title   = {Open and closed universes, initial singularities, and inflation},
  journal = {Physical Review D},
  volume  = {50},
  number  = {6},
  pages   = {3692--3702},
  year    = {1994},
  doi     = {10.1103/PhysRevD.50.3692},
  eprint  = {gr-qc/9403049},
  eprinttype = {arxiv}
}

@incollection{Brenier-Extended-MK,
  author    = {Brenier, Y.},
  title     = {Extended Monge-Kantorovich theory},
  booktitle = {Optimal transportation and applications},
  editor    = {Not specified},
  series    = {Lecture Notes in Mathematics},
  volume    = {1813},
  pages     = {91--121},
  year      = {2003},
  publisher = {Springer},
  address   = {Berlin},
  note      = {From the conference held in Martina Franca, 2001}
}

@article{Bertrand-Pratelli-Puel,
  author    = {Bertrand, J. and Pratelli, A. and Puel, M.},
  title     = {Kantorovich potentials and continuity of total cost for relativistic cost functions},
  journal   = {Journal de Math{\'e}matiques Pures et Appliqu{\'e}es},
  series    = {9},
  volume    = {110},
  pages     = {93--122},
  year      = {2018}
}

@article{Bertrand-Puel,
  author    = {Bertrand, J. and Puel, M.},
  title     = {The optimal mass transport problem for relativistic costs},
  journal   = {Calculus of Variations and Partial Differential Equations},
  volume    = {46},
  number    = {1-2}, 
  pages     = {353--374},
  year      = {2013}
}

@article{McCann-Puel,
  author    = {McCann, R. J. and Puel, M.},
  title     = {Constructing a relativistic heat flow by time transport steps},
  journal   = {Annales de l'Institut Henri Poincar{\'e} - Analyse Non Lin{\'e}aire},
  volume    = {26},
  pages     = {2539--2580},
  year      = {2009}
}

@article{Suhr,
  author    = {Suhr, S.},
  title     = {Theory of optimal transport for Lorentzian cost functions},
  journal   = {M{\"u}nster Journal of Mathematics},
  volume    = {11},
  pages     = {13--47},
  year      = {2018},
}

@misc{Cavalletti-Mondino2025+,
      title={A sharp isoperimetric-type inequality for Lorentzian spaces satisfying timelike Ricci lower bounds}, 
      author={Fabio Cavalletti and Andrea Mondino},
      year={2025},
      eprint={2401.03949},
      archivePrefix={arXiv},
      primaryClass={math.MG},
      url={https://arxiv.org/abs/2401.03949}, 
}

@article{Kell-Suhr,
  author    = {Kell, M. and Suhr, S.},
  title     = {On the existence of dual solutions for Lorentzian cost functions},
  journal   = {Annales de l'Institut Henri Poincar{\'e} - Analyse Non Lin{\'e}aire},
  volume    = {37},
  number    = {2},
  pages     = {343--372},
  year      = {2020},
}

@article{Eckstein-Miller,
  author    = {Eckstein, M. and Miller, T.},
  title     = {Causality for nonlocal phenomena},
  journal   = {Annales Henri Poincar{\'e}},
  volume    = {18},
  pages     = {3049--3096},
  year      = {2017},
}

@inbook{Cavalletti-Overview,
url = {https://doi.org/10.1515/9783110550832-003},
title = {An Overview of L1 optimal transportation on metric measure spaces},
booktitle = {Measure Theory in Non-Smooth Spaces},
author = {Fabio Cavalletti},
publisher = {De Gruyter Open Poland},
address = {Warsaw, Poland},
pages = {98--144},
doi = {doi:10.1515/9783110550832-003},
isbn = {9783110550832},
year = {2017},
lastchecked = {2026-01-17}
}

@article {EvansGangbo,
    AUTHOR = {Evans, L. C. and Gangbo, W.},
     TITLE = {Differential equations methods for the {M}onge-{K}antorovich
              mass transfer problem},
   JOURNAL = {Mem. Amer. Math. Soc.},
  FJOURNAL = {Memoirs of the American Mathematical Society},
    VOLUME = {137},
      YEAR = {1999},
    NUMBER = {653},
     PAGES = {viii+66},
      ISSN = {0065-9266,1947-6221},
   MRCLASS = {35Q99 (49J10)},
  MRNUMBER = {1464149},
MRREVIEWER = {John\ Urbas},
       DOI = {10.1090/memo/0653},
       URL = {https://doi.org/10.1090/memo/0653},
}

@ARTICLE{BraunTCD,
  title     = "R{\'e}nyi's entropy on Lorentzian spaces. Timelike
               curvature-dimension conditions",
  author    = "Braun, Mathias",
  journal   = "J. Math. Pures Appl.",
  publisher = "Elsevier BV",
  volume    =  177,
  pages     = "46--128",
  month     =  sep,
  year      =  2023,
  language  = "en"
}

@book{Monge,
  author = {Monge, Gaspard},
  description = {Mentioned by Filippo Simini, Marta C. González, Amos Maritan, and Albert-László Barabási in "A universal model for mobility and migration patterns" to mention the gravity model.},
  publisher = {De l'Imprimerie Royale},
  title = {M{\'e}moire sur la th{\'e}orie des d{\'e}blais et des remblais},
  year = 1781
}

@article{Kantorovich,
  author = {Kantorovich, Leonid},
  JOURNAL = {(Doklady) Acad. Sci. URSS (N.S.)},
  title = {On the translocation of masses},
  year = 1942,
VOLUME =  {37},
  PAGES = {199–201}
}

@article {Sudakov,
    AUTHOR = {Sudakov, V. N.},
     TITLE = {Geometric problems in the theory of infinite-dimensional
              probability distributions},
   JOURNAL = {Proc. Steklov Inst. Math.},
  FJOURNAL = {Proceedings of the Steklov Institute of Mathematics},
      YEAR = {1979},
    NUMBER = {2},
     PAGES = {i--v, 1--178},
      ISSN = {0081-5438},
   MRCLASS = {60G15 (28A35 52A05)},
  MRNUMBER = {530375},
}

@misc{AKP,
  author    = {Alberti, G. and Kircheim, B. and Preiss, D.},
  note      = {Presented in a lecture by Kircheim at the Scuola Normale Superiore workshop, October 27, 2000.},
}

@article{Mondino-Suhr,
   title={An optimal transport formulation of the Einstein equations of general relativity},
   volume={25},
   ISSN={1435-9863},
   url={http://dx.doi.org/10.4171/JEMS/1188},
   DOI={10.4171/jems/1188},
   number={3},
   journal={Journal of the European Mathematical Society},
   publisher={European Mathematical Society - EMS - Publishing House GmbH},
   author={Mondino, Andrea and Suhr, Stefan},
   year={2022},
   month=jan, pages={933–994} }

@article{Feldman-McCann,
  title     = "Monge's transport problem on a Riemannian manifold",
  author    = "Feldman, Mikhail and McCann, Robert J.",
  journal   = "Trans. Am. Math. Soc.",
  publisher = "American Mathematical Society (AMS)",
  volume    =  354,
  number    =  4,
  pages     = "1667--1697",
  month     =  dec,
  year      =  2001,
}

@incollection{Ambrosio,
  author    = {Ambrosio, Luigi},
  title     = {Lecture Notes on Optimal Transport Problems},
  booktitle = {Mathematical Aspects of Evolving Interfaces},
  series    = {Lecture Notes in Mathematics},
  volume    = {1812},
  pages     = {1--52},
  year      = {2003},
  publisher = {Springer},
  address   = {Berlin},
  doi       = {10.1007/978-3-540-39189-0_1}
}

@ARTICLE{Trudinger-Wang,
  title     = "On the Monge mass transfer problem",
  author    = "Trudinger, Neil S. and Wang, Xu-Jia",
  journal   = "Calc. Var. Partial Differ. Equ.",
  publisher = "Springer Science and Business Media LLC",
  volume    =  13,
  number    =  1,
  pages     = "19--31",
  month     =  aug,
  year      =  2001,
  language  = "en"
}

@book {VillaniTopics,
    AUTHOR = {Villani, C\'edric},
     TITLE = {Topics in optimal transportation},
    SERIES = {Graduate Studies in Mathematics},
    VOLUME = {58},
 PUBLISHER = {American Mathematical Society, Providence, RI},
      YEAR = {2003},
     PAGES = {xvi+370},
      ISBN = {0-8218-3312-X},
   MRCLASS = {90-02 (28D05 35B65 35J60 49N90 49Q20 90B20)},
  MRNUMBER = {1964483},
       DOI = {10.1090/gsm/058},
       URL = {https://doi.org/10.1090/gsm/058},
}

@article{Kell2017,
  author = {Kell, M.},
  title = {Transport maps, non-branching sets of geodesics and measure rigidity},
  journal = {Adv. Math.},
  volume = {320},
  year = {2017},
  pages = {520--573}
}

@article{RajalaSturm2014,
  author  = {Rajala, Tapio and Sturm, Karl-Theodor},
  title   = {Non-branching geodesics and optimal maps in strong {CD(K,$\infty$)}-spaces},
  journal = {Calc. Var. Partial Differential Equations},
  volume  = {50},
  pages   = {831--846},
  year    = {2014},
  doi     = {10.1007/s00526-013-0657-x}
}

@article{CavallettiHuesmann2015,
  author = {Cavalletti, F. and Huesmann, M.},
  title = {Existence and uniqueness of optimal transport maps},
  journal = {Ann. I. H. Poincar\'e AN},
  volume = {32},
  year = {2015},
  pages = {1367--1377}
}

@article{vonRenesseSturm,
author = {von Renesse, Max-K. and Sturm, Karl-Theodor},
title = {Transport inequalities, gradient estimates, entropy and Ricci curvature},
journal = {Communications on Pure and Applied Mathematics},
volume = {58},
number = {7},
pages = {923-940},
doi = {https://doi.org/10.1002/cpa.20060},
url = {https://onlinelibrary.wiley.com/doi/abs/10.1002/cpa.20060},
eprint = {https://onlinelibrary.wiley.com/doi/pdf/10.1002/cpa.20060},
year = {2005}
}

@book {MR1835418,
    AUTHOR = {Burago, Dmitri and Burago, Yuri and Ivanov, Sergei},
     TITLE = {A course in metric geometry},
    SERIES = {Graduate Studies in Mathematics},
    VOLUME = {33},
 PUBLISHER = {American Mathematical Society, Providence, RI},
      YEAR = {2001},
     PAGES = {xiv+415},
      ISBN = {0-8218-2129-6},
   MRCLASS = {53C23},
  MRNUMBER = {1835418},
MRREVIEWER = {Mario\ Bonk},
       DOI = {10.1090/gsm/033},
       URL = {https://doi.org/10.1090/gsm/033},
}

@misc{octet+,
      title={A nonlinear d'Alembert comparison theorem and causal differential calculus on metric measure spacetimes}, 
      author={Tobias Beran and Mathias Braun and Matteo Calisti and Nicola Gigli and Robert J. McCann and Argam Ohanyan and Felix Rott and Clemens Sämann},
      year={2025},
      eprint={2408.15968},
      archivePrefix={arXiv},
      primaryClass={math.DG},
      url={https://arxiv.org/abs/2408.15968}, 
}

@article{Klartag2017,
  author  = {B. Klartag},
  title   = {Needle decomposition in {Riemannian} geometry},
  journal = {Mem. Amer. Math. Soc.},
  volume  = {249},
  year    = {2017},
  number  = {1180}
}

@article{Beran-Harvey-Napper-Rott,
  author  = {Beran, Tobias and Harvey, John and Napper, Lewis and Rott, Felix},
  title   = {A Toponogov globalisation result for Lorentzian length spaces},
  journal = {Mathematische Annalen},
  volume  = {392},
  pages   = {3447--3478},
  year    = {2025},
  doi     = {10.1007/s00208-025-03167-w},
  url     = {https://doi.org/10.1007/s00208-025-03167-w}
}
\end{document}